\documentclass[11pt,a4paper,reqno]{amsart}
    \usepackage{vmargin,color}
    \usepackage[latin1]{inputenc}
    \usepackage{amsmath,amsfonts,amsthm,epsfig,graphicx,amssymb}
    \usepackage{esint}
    \usepackage{caption}

    \newcommand{\CC}{\mathbf{C}}
    \newcommand{\DD}{\mathbf{D}}

    \newcommand{\F}{\mathcal F}

    \renewcommand{\H}{\mathcal{H}}
    \renewcommand{\L}{\mathcal{L}}
    
    \renewcommand{\P}{\mathcal{P}}
    \newcommand{\Sc}{\mathcal S}

    \newcommand{\N}{\mathbb{N}}
    \newcommand{\Z}{\mathbb{Z}}
    \newcommand{\Q}{\mathbb{Q}}
    \newcommand{\R}{\mathbb{R}}
    \renewcommand{\SS}{\mathbb{S}}

    \newcommand{\Om}{\Omega}
    
    \renewcommand{\a}{\alpha}
    \renewcommand{\b}{\beta}
    \newcommand{\g}{\gamma}
    \newcommand{\de}{\delta}
    \newcommand{\e}{\varepsilon}
    
    \renewcommand{\l}{\lambda}
    \newcommand{\s}{\sigma}
    
    \newcommand{\om}{\omega}
    \newcommand{\vphi}{\varphi}
    \newcommand{\Lip}{{\rm Lip}}

    \newcommand{\Div}{{\rm div}\,}
    \newcommand{\Id}{{\rm Id}\,}
    \newcommand{\dist}{{\rm dist}}

    \newcommand{\diam}{{\rm diam}\,}

    \newcommand{\spt}{{\rm spt}}

    \newcommand{\weakstar}{\stackrel{\scriptscriptstyle{*}}{\rightharpoonup}}
    \newcommand{\toloc}{\stackrel{\scriptscriptstyle{{\rm loc}}}{\to}}

    \newcommand{\pa}{\partial}

    \newcommand{\cc}{\subset\!\subset}

    \newcommand{\cl}{\mathrm{cl}\,}

    \newcommand{\KK}{\mathcal{K}}
    \newcommand{\T}{\mathcal{T}}

    \newcommand{\pp}{\mathbf{p}}

    \newcommand{\var}{\mathbf{var}\,}

    \newcommand{\intci}{\mathrm{int}\, U_j^\one}

    \newcommand{\C}{\mathcal{C}}
    
    \newcommand{\A}{\mathcal{A}}
    \newcommand\restr[2]{{
      \left.\kern-\nulldelimiterspace 
      #1 
      \right|_{#2} 
      }}
    \newcommand{\K}{\mathcal{K}}

    \newcommand{\RR}{\mathcal{R}}

    \newcommand{\one}{{\scriptscriptstyle{(1)}}}
    \newcommand{\zero}{{\scriptscriptstyle{(0)}}}
    \newcommand{\half}{{\scriptscriptstyle{(1/2)}}}

    \newcommand{\mres}{\mathbin{\vrule height 1.6ex depth 0pt width 
    0.13ex\vrule height 0.13ex depth 0pt width 1.3ex}}

    \theoremstyle{plain}
    \newtheorem{theorem}{Theorem}[section]
    \newtheorem{lemma}[theorem]{Lemma}

    \newtheorem*{theorem*}{Theorem}
    \newtheorem*{corollary*}{Corollary}

    \theoremstyle{definition}
    \newtheorem{definition}{Definition}

    \newtheorem{remark}[theorem]{Remark}

    \newtheorem*{notation*}{Notation}

    \numberwithin{equation}{section}
    \numberwithin{figure}{section}
    \newcommand{\id}{{\rm id}\,}

    \newcommand{\wire}{\mathbf{W}}
    \newcommand{\shn}{\overset{\scriptscriptstyle{\H^n}}{\subset}}
    \newcommand{\ehn}{\overset{\scriptscriptstyle{\H^n}}{=}}

    \newcommand{\tnl}{\theta^{n}_*}

    \allowdisplaybreaks

    \title[]{Plateau borders in soap films \\
    and {G}auss' capillarity theory}

    \author{Francesco Maggi}
    \address{Department of Mathematics, The University of Texas at Austin, Austin, TX, United States of America}
    \email{maggi@math.utexas.edu}
    \author{Michael Novack}
    \address{Department of Mathematical Sciences, Carnegie Mellon University, Pittsburgh, PA, United States of America}
    \email{mnovack@andrew.cmu.edu}
    \author{Daniel Restrepo}
    \address{Department of Mathematics, Johns Hopkins University, Baltimore, MD, United States of America}
    \email{drestre1@jh.edu}

\begin{document}

    \begin{abstract}
    {\rm We provide, in the setting of Gauss' capillarity theory, a rigorous derivation of the equilibrium law for the three dimensional structures known as {\it Plateau borders} which arise in ``wet'' soap films and foams. A key step in our analysis is a complete measure-theoretic overhaul of the homotopic spanning condition introduced by Harrison and Pugh in the study of Plateau's laws for two-dimensional area minimizing surfaces (``dry'' soap films). This new point of view allows us to obtain effective compactness theorems and energy representation formulae for the homotopic spanning relaxation of Gauss' capillarity theory which, in turn, lead to prove sharp regularity properties of energy minimizers. The equilibrium law for Plateau borders in wet foams is also addressed as a (simpler) variant of the theory for wet soap films.}
    \end{abstract}

    \maketitle

    \setcounter{tocdepth}{1}

    \tableofcontents

\section{Introduction} \subsection{Overview}\label{section overview} Equilibrium configurations of soap films and foams are governed, at leading order, by the balance between surface tension forces and atmospheric pressure. This balance is expressed by the {\it Laplace--Young law of pressures}, according to which such systems can be decomposed into smooth interfaces with constant mean curvature equal to the pressure difference across them, and by the {\it Plateau laws}, which precisely postulate which arrangements of smooth interfaces joined together along lines of ``singular'' points are stable, and thus observable.

\medskip

The physics literature identifies two (closely related) classes of soap films and foams, respectively labeled as ``dry'' and ``wet''. This difference is either marked in terms of the amount of liquid contained in the soap film/foam \cite[Section 1.3]{weaireBOOK}, or in terms of the scale at which the soap film/foam is described \cite[Chapter 2, Section 3 and 4]{foamchapter}.

\medskip

In the dry case, Plateau laws postulates that (i) interfaces can only meet in three at a time forming 120-degree angles along lines of ``$Y$-points''; and (ii) lines of $Y$-points can only meet in fours at isolated ``$T$-points'', where six interfaces asymptotically form a perfectly symmetric tetrahedral angle; see, e.g. \cite[Equilibrium rules A1, A2, page 24]{weaireBOOK}.

\medskip

In the wet case, small but positive amounts of liquid are bounded by negatively curved interfaces, known as {\it Plateau borders}, and arranged near ideal lines of $Y$-points or isolated $T$-points;  see Figure
\begin{figure}
\input{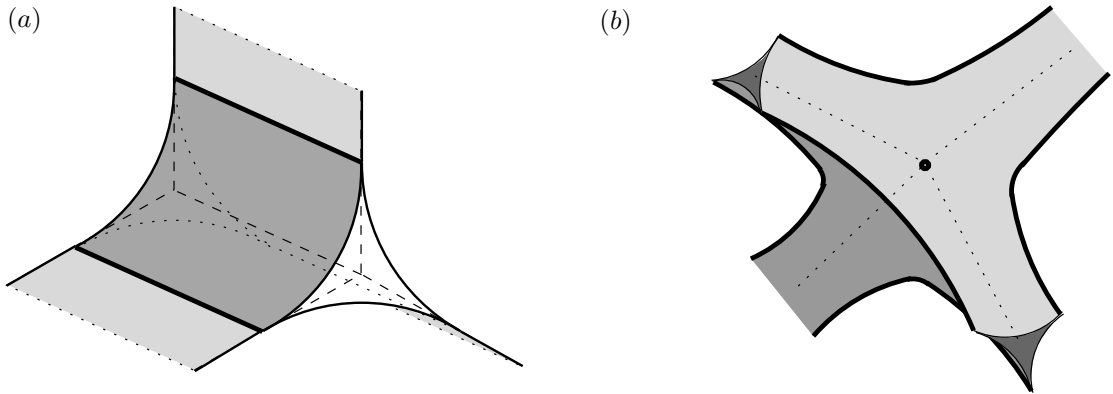}
\caption{\small{(a) A Plateau border develops around a ``wet'' line of $Y$-points. The wet region is bounded by interfaces of {\it negative} constant mean curvature. The equilibrium condition which needs to hold across the transition lines (here depicted in bold) between the negatively curved interfaces of a Plateau border and the incoming dry interfaces is that these interfaces meet tangentially. In the case of soap films, where the dry interfaces have zero mean curvature, the jump in the mean curvature across the transition lines implies a discontinuity in the gradient of the unit normal. (b) An arrangement of Plateau borders near a tetrahedral singularity. The transition lines are again depicted in bold. The incoming dry interfaces are omitted for clarity.}}
\label{fig plateauborder}
\end{figure}
\ref{fig plateauborder} and \cite[Fig. 1.8 and Fig. 1.9]{weaireBOOK}. A ``third Plateau law'' is then postulated to hold across the transition lines between wet and dry parts of soap films/foams, and can be formulated as follows:
\begin{eqnarray}\label{third pl}
  &&\mbox{{\it the unit normal to a soap film/foam changes continuously}}
  \\\nonumber
  &&\mbox{{\it across the transition lines between wet and dry interfaces}}\,;
\end{eqnarray}
see, e.g., \cite[Equilibrium rule B, page 25]{weaireBOOK} and \cite[Section 4.1.4]{foamchapter}. It is important to recall that Plateau borders play a crucial role in determining the mechanical properties of the many physical and biological systems in which they are observed. As a sample of older and newer papers discussing Plateau borders, we mention here \cite{Leonard196518,JohaPugh,Bhakta19971,Koehler19994232,LU19991469,EXAMPLEchem,EXAMPLEbio,Singh2015467}. Postulate \eqref{third pl} is assumed in all these works.

\medskip

The goal of this paper is answering the natural problem of rigorously deriving the equilibrium condition for Plateau borders \eqref{third pl} in the context of Gauss' capillarity theory. Since the case of soap films is much harder and interesting from the mathematical viewpoint, we will postpone the discussion of foams until the very last section of this introduction. The main highlight is that, in addressing Plateau borders of soap films, we will develop a new ``theory of spanning'' for surfaces of geometric measure theory (GMT) which will find further applications in the two companion papers \cite{MNR2,MNR3}; see the closing of this overview for more details about these additional applications.

\medskip

We now give an informal description of our approach. The starting point is \cite{MaggiScardicchioStuvard}, where the idea  is introduced  of modeling soap films as regions $E$ of positive volume $|E|=v$ contained in the complement $\Omega=\R^{n+1}\setminus\wire$ of a ``wire frame'' $\wire$ ($n=2$ is the physical case, although the planar case $n=1$ is also quite interesting in applications). We associate to $E$ the surface tension energy $\H^n(\Om\cap\pa E)$ (where $\H^n$ stands for $n$-dimensional (Hausdorff) measure, i.e., area when $n=2$ and length when $n=1$), and minimize $\H^n(\Om\cap\pa E)$ under the constraints that $|E|=v$ (for some given $v>0$) and
\begin{equation}
  \label{meaning of}
  \mbox{$\Om\cap\pa E$ is spanning $\wire$}\,.
\end{equation}
From the mathematical viewpoint the meaning assigned to \eqref{meaning of} is, of course, the crux of the matter. In the informal spirit of this overview, we momentarily leave the concept of ``spanning'' only intuitively defined.

\medskip

As proved in \cite{KingMaggiStuvard}, this minimization process leads to the identification of {\it generalized minimizers} in the form of pairs $(K,E)$ with $E\subset\Om$, $|E|=v$, and such that
\begin{equation}
  \label{meaning of generalized}
  \mbox{$\Om\cap\pa E\subset K$ and $K$ is spanning $\wire$}\,.
\end{equation}
These pairs are minimizing in the sense that
\begin{equation}
  \label{kms minimizers}
  \H^n(\Om\cap\pa E)+2\,\H^n(K\setminus\pa E)\le \H^n(\Om\cap\pa E')\,,
\end{equation}
whenever $E'\subset\Om$, $|E'|=v$ and $\Om\cap\pa E'$ is spanning $\wire$.

\medskip

If $K=\Om\cap\pa E$, then generalized minimizers are of course minimizers in the proper sense. If not, the {\it collapsed interface} $K\setminus\pa E$ is a surface whose positive area has to be counted with a multiplicity factor $2$ (which arises from the asymptotic collapsing along $K\setminus\pa E$ of oppositely oriented boundaries in minimizing sequences $\{E_j\}_j$, see
\begin{figure}
\input{collapsing.pstex_t}\caption{{\small Emergence of collapsing along a minimizing sequence $\{E_j\}_j$ for the minimization of $\H^n(\Om\cap\pa E)$ among sets $E\subset\Omega=\R^{n+1}\setminus\wire$ with $|E|=v$ and $\Om\cap\pa E$ spanning $\wire$, when $n=1$ and $\wire$ is the union of three disks in the plane. Notice that for this choice of $\wire$ the minimization of $\H^n(S)$ among $S\subset\Om$ such that $S$ is spanning $\wire$ is solved by three segments meeting at $Y$-point. Collapsing is intuitively related to the presence of $Y$-type and $T$-type singularities.}
}\label{fig collapsing}
\end{figure}
Figure \ref{fig collapsing}). We expect collapsing to occur whenever the Plateau problem for $\wire$ admits one minimizer $S$ with Plateau-type singularities. Whenever this happens, a {\it wetting conjecture} is made: sequences $\{(K_{v_j},E_{v_j})\}_j$ of generalized minimizers with $|E_{v_j}|=v_j\to 0^+$ as $j\to\infty$ will be such that the set of Plateau's singularities $\Sigma(S)$ of $S$ is such that $\sup\{\dist(x,E_{v_j}):x\in\Sigma(S)\}\to 0$. Thus we expect that Plateau's singularities are never ``left dry'' in the small volume capillarity approximation of the Plateau problem.

\medskip

A lot of information about generalized minimizers can be extracted from \eqref{kms minimizers}, and this is the content of \cite{KingMaggiStuvard,KMS2,KMS3}. With reference to the cases when $n=1$ or $n=2$, one can deduce from \eqref{kms minimizers} that if $\H^n(K\setminus\pa E)>0$, then $K\setminus\pa E$ is a smooth minimal surface (a union of segments if $n=1$) and that $\pa E$ contains a regular part $\pa^*E$ that is a smooth constant mean curvature surface (a union of circular arcs if $n=1$) with {\it negative} curvature. This is of course strongly reminiscent of the behavior of Plateau borders, and invites to analyze the validity of \eqref{third pl} in this context. A main obstacle is that, due to serious technical issues (described in more detail later on) related to how minimality is expressed in \eqref{kms minimizers}, it turns out to be very difficult to say much about the ``transition line''
\[
\pa E\setminus\pa^*E
\]
between the zero and the negative constant mean curvature interfaces in $K$, across which one should check the validity of \eqref{third pl}. More precisely, all that descends from \eqref{kms minimizers} and a direct application of Allard's regularity theorem \cite{Allard} is that {\it $\pa E\setminus\pa^*E$ has empty interior in $K$}. Far from being a line in dimension $n=2$, or a discrete set of points when $n=1$, the transition line $\pa E\setminus\pa^*E$ could very well have positive $\H^n$-measure and be everywhere dense in $K$! With such poor understanding of $\pa E\setminus\pa^*E$, proving the validity of \eqref{third pl} -- that is, the continuity of the unit normals to $K\setminus\pa E$ and $\pa^*E$ in passing across $\pa E\setminus\pa^*E$ -- is of course out of question.

\medskip

We overcome these difficulties by performing a major measure-theoretic overhaul of the Harrison--Pugh homotopic spanning condition \cite{harrisonpughACV,harrisonpughGENMETH} used in \cite{MaggiScardicchioStuvard,KingMaggiStuvard,KMS2,KMS3} to give a rigorous meaning to \eqref{meaning of}, and thus to formulate the homotopic spanning relaxation of Gauss' capillarity discussed above.

\medskip

The transformation of this purely topological concept into a measure-theoretic one is particularly powerful. Its most important consequence for the problem discussed in this paper is that it allows us to upgrade the partial minimality property \eqref{kms minimizers} of $(K,E)$ into the full minimality property
\begin{equation}
  \label{full minimizers}
  \H^n(\Om\cap\pa E)+2\,\H^n(K\setminus\pa E)\le  \H^n(\Om\cap\pa E')+2\,\H^n(K'\setminus\pa E')
\end{equation}
whenever $E'\subset\Om$, $|E'|=v$, $\Om\cap\pa E'\subset K'$ and $K'$ is spanning $\wire$. The crucial difference between \eqref{kms minimizers} and \eqref{full minimizers} is that the latter is much more efficient than the former when it comes to study the regularity of generalized minimizers $(K,E)$, something that is evidently done by energy comparison with competitors $(K',E')$. Such comparisons are immediate when working with \eqref{full minimizers}, but they are actually quite delicate to set up when we only have \eqref{kms minimizers}. In the latter case, given a competitor $(K',E')$, to set up the energy comparison with $(K,E)$ we first need to find a sequence of non-collapsed competitors $\{E'_j\}_j$ (with $E'_j\subset\Om$, $|E'_j|=v$, and $\Om\cap\pa E'_j$ spanning $\wire$) such that $\H^n(\Om\cap\pa E'_j)\to\H^n(\Om\cap\pa E')+2\,\H^n(K'\setminus\pa E')$. Intuitively, $E_j'$ needs to be a $\de_j$-neighborhood of $K'\cup E'$ for some $\de_j\to 0^+$ and the energy approximation property has to be deduced from the theory of Minkowski content. But applying the theory of Minkowski content to $(K',E')$ (which is the approach followed, e.g., in \cite{KMS3}) requires $(K',E')$ to satisfy rectifiability and uniform density properties that substantially restrict the class of available competitors $(K',E')$.

\medskip

In contrast, once the validity of \eqref{full minimizers} is established, a suitable generalization (Theorem \ref{theorem decomposition intro}) of the partition theorem of sets of finite perimeter into indecomposable components \cite[Theorem 1]{ambrosiocaselles} combined with a subtle variational argument (see Figure \ref{fig lambda}) allows us to show that, in any ball $B\cc\Omega$ with sufficiently small radius and for some sufficiently large constant $\Lambda$ (both depending just on $(K,E)$), the connected components $\{U_i\}_i$ of $B\setminus (K\cup E)$ satisfy a perturbed area minimizing property of the form
\begin{equation}
  \label{lambda minimality}
  \H^n(B\cap\pa U_i)\le\H^n(B\cap \pa V)+\Lambda\,|U_i\Delta V|\,,
\end{equation}
with respect to {\it completely arbitrary perturbations} $V\subset B$, $V\Delta U_i\cc B$. By a classical theorem of De Giorgi \cite{DeGiorgiREG,tamanini}, \eqref{lambda minimality} implies (away from a closed singular set of codimension at least $8$, which is thus empty if $n\le 6$) the $C^{1,\a}$-regularity of $B\cap\pa U_i$ for each $i$, and thus establishes {\it the continuity of the normal stated in \eqref{third pl}}. In fact, locally at each $x$ on the transition line, $K$ is the union of the graphs of two $C^{1,\a}$-functions $u_1\le u_2$ defined on an $n$-dimensional disk, having zero mean curvature above the interior of $\{u_1=u_2\}$, and opposite constant mean curvature above $\{u_1<u_2\}$. We can thus exploit the regularity theory for double-membrane free boundary problems devised in \cite{silvestre,FocardiGelliSpadaro} to deduce that the transition line $\pa E\setminus\pa^*E$ is indeed $(n-1)$-dimensional, and to improve the $C^{1,\a}$-regularity of $B\cap\pa U_i$ to $C^{1,1}$-regularity. Given the mean curvature jump across $\pa E\setminus\pa^*E$ we have thus established the {\it sharp} degree of regularity for minimizers of the homotopic spanning relaxation of Gauss' capillarity theory.

\medskip

The measure-theoretic framework for homotopic spanning conditions laid down in this paper provides the starting point for additional investigations that would otherwise seem unaccessible. In two forthcoming companion papers we indeed establish (i) the convergence towards Plateau-type singularities of energy-minimizing diffused interface solutions of the Allen--Cahn equation \cite{MNR2}, and (ii) some sharp convergence theorems for generalized minimizers in the homotopic spanning relaxation of Gauss' capillarity theory in the vanishing volume limit, including a proof of the above mentioned wetting conjecture \cite{MNR3}.

\medskip

The rest of this introduction is devoted to a rigorous formulation of the results presented in this overview. We begin in Section \ref{section plateau intro} with a review of the Harrison and Pugh homotopic spanning condition in relation to the classical Plateau problem and to the foundational work of Almgren and Taylor \cite{Almgren76,taylor76}. In Section \ref{section measth homotopic span} we introduce the new measure-theoretic formulation of homotopic spanning and discuss its relation to the measure-theoretic notion of {\it essential connectedness} introduced by Cagnetti, Colombo, De Philippis and the first-named author in the study of symmetrization inequalities \cite{CCDPM17,CCDPMSteiner}. In Section \ref{section main results} we introduce the {\it bulk} and {\it boundary} spanning relaxations of Gauss' capillarity theory, state a general closure theorem for ``generalized soap films'' that applies to both relaxed problems (Theorem \ref{theorem first closure theorem intro}). In Section \ref{subsection existence and convergence} we prove the existence of generalized soap film minimizers (Theorem \ref{thm existence EL for bulk}) and their convergence in energy to solutions to the Plateau problem. A sharp regularity theorem (Theorem \ref{thm regularity for bulk}) for these minimizers, which validates \eqref{third pl}, is stated in Section \ref{section validation}. Finally, in Section \ref{section foams intro} we reformulate the above results in the case of foams, see in particular Theorem \ref{theorem foams}.

\subsection{Homotopic spanning: from Plateau's problem to Gauss' capillarity}\label{section plateau intro} The theories of currents and of sets of finite perimeter, i.e. the basic distributional theories of surface area at the center of GMT, fall short in the task of modeling Plateau's laws. Indeed, two-dimensional area minimizing currents in $\R^3$ are carried by smooth minimal surfaces, and thus cannot model $Y$-type\footnote{Currents modulo $3$ are compatible with $Y$-type singularities, but not with $T$-type singularities.} and $T$-type singularities. This basic issue motivated the introduction of {\bf Almgren minimal sets} as models for soap films in \cite{Almgren76}: these are sets $S\subset\R^{n+1}$ that are relatively closed in a given open set $\Om\subset\R^{n+1}$, and satisfy $\H^n(S)\le\H^n(f(S))$ whenever $f:\Om\to\Om$ is a {\it Lipschitz} (not necessarily injective) map with $\{f\ne\id\}\cc\Om$. Taylor's historical result \cite{taylor76} validates the Plateau laws in this context, by showing that, when\footnote{Similar regularity assertions hold when $n=1$ (by elementary methods) and, in much more recent developments, when $n\ge 3$ \cite{ColomboEdelenSpolaor}.} $n=2$, Almgren minimal sets are locally $C^{1,\a}$-diffeomorphic either to planes, to $Y$-cones, or to $T$-cones.

\medskip

The issue of proposing and solving a formulation of Plateau's problem whose minimizers are Almgren minimal sets, and indeed admit Plateau-type singularities, is quite elusive, as carefully explained in \cite{davidshouldwe}. In this direction, a major breakthrough has been obtained by Harrison and Pugh in \cite{harrisonpughACV} with the introduction of a new spanning condition, which, following the presentation in \cite{DLGM}, can be defined as follows:

\begin{definition}[Homotopic spanning (on closed sets)]\label{def homot span closed}{\rm
Given a closed set $\wire \subset \mathbb{R}^{n+1}$ (the ``wire frame''), a {\bf spanning class for $\wire$} is a family $\C$ of smooth embeddings of $\mathbb{S}^1$ into
\[
\Om=\mathbb{R}^{n+1}\setminus \wire
\]
that is {\it closed under homotopies in $\Omega$}, that is, if $\Phi:[0,1]\times\SS^1\to\Om$ is smooth family of embeddings $\Phi_t=\Phi(t,\cdot):\SS^1\to\Om$ with $\Phi_0\in\C$, then $\Phi_t\in\C$ for every $t\in(0,1]$. A set $S$, contained and relatively closed in $\Om$, is said to be {\bf $\C$-spanning $\wire$} if
\[
S \cap \gamma \neq \varnothing\,,\qquad \forall \gamma \in \C\,.
\]
Denoting by $\Sc(\C)$ the class of sets $S$ $\C$-spanning $\wire$, one can correspondingly formulate the {\bf Plateau problem} (with homotopic spanning)
\begin{equation}
\label{hp problem intro}
\ell=\ell(\C):=\inf\big\{\H^n(S):S\in\Sc(\C)\big\}\,.
\end{equation}}
\end{definition}

Existence of minimizers of $\ell$ holds as soon as $\ell<\infty$, and minimizers $S$ of $\ell$ are Almgren minimal sets in $\Om$ \cite{harrisonpughACV,DLGM} that are indeed going to exhibit Plateau-type singularities (this is easily seen in the plane, but see also \cite{bernsteinmaggi} for a higher dimensional example). Moreover, given a same $\wire$, different choices of $\C$ are possible and can lead to different minimizers, see
\begin{figure}\input{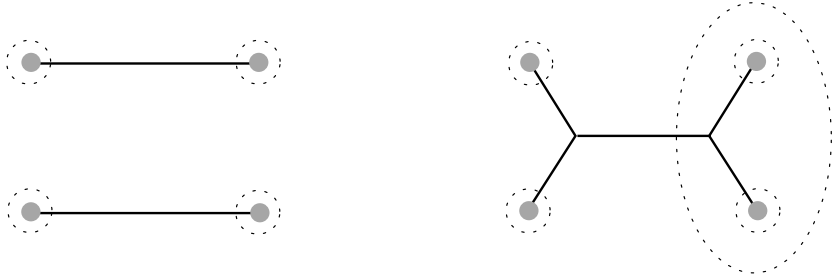}\caption{\small{The dashed lines denote the embeddings of $\SS^1$ whose homotopy classes relative to $\Om$ generate different spanning classes $\C$, to which there correspond different minimizers of $\ell$.}}\label{fig choice}\end{figure}
Figure \ref{fig choice}. Finally, the approach is robust enough to provide the starting point for several important extensions \cite{dPdRghira,derosaSIAM,harrisonpughGENMETH,FangKola,delederosaghira,DDG2}, including higher codimension, anisotropic energies, etc.

\medskip

The study of soap films as minimizers of Gauss's capillarity energy with small volume and under homotopic spanning conditions has been initiated in \cite{MaggiScardicchioStuvard,KingMaggiStuvard}, with the introduction of the model
\begin{equation}
\label{def of psi}
\psi(v):=\inf\Big\{\H^n(\Om\cap\pa E):|E|=v\,,\,\,\mbox{$\Om\cap\pa E$ is $\C$-spanning $\wire$}\Big\}\,,
\end{equation}
where $E\subset\Omega$ is an open set with smooth boundary. Without the spanning condition, at small volumes, minimizers of $\H^n(\Om\cap\pa E)$ would be small diffeomorphic images of half-balls \cite{maggimihaila}. However, the introduction of the $\C$-spanning constraint rules out small droplets, and forces the exploration of a different part of the energy landscape of $\H^n(\Om\cap\pa E)$. As informally discussed in Section \ref{section overview}, this leads to the emergence of generalized minimizers $(K,E)$. More precisely, in \cite{KingMaggiStuvard} the existence is proved of $(K,E)$ in the class
\begin{eqnarray} \label{def of K}
&&\K=\Big\{(K,E):\mbox{$K$ is relatively closed and $\H^n$-rectifiable in $\Om$, $E$ is open,}
\\\nonumber
&&\hspace{2.7cm}\mbox{$E$ has finite perimeter in $\Om$, and $\Om\cap\cl(\pa^*E)=\Om\cap\pa E\subset K$}\Big\}\,,
\end{eqnarray}
(where $\pa^* E$ denotes the reduced boundary of $E$) such that, for every competitor $E'$ in $\psi(v)$, it holds
\begin{equation}
  \label{kms minimizers intro}
  \H^n(\Om\cap\pa^*E)+2\,\H^n(\Om\cap(K\setminus\pa^*E))\le \H^n(\Om\cap\pa E')\,.
\end{equation}
Starting from \eqref{kms minimizers intro} one can apply Allard's regularity theorem \cite{Allard} and various {\it ad hoc} comparison arguments \cite{KMS2,KMS3} to prove that $\Om\cap\pa^*E$ is a smooth hypersurface with constant mean curvature (negative if $\H^n(K\setminus\pa^*E)>0$), $\Om\cap(\pa E\setminus\pa^*E)$ has empty interior in $K$, and that $K\setminus(\Sigma\cup\pa E)$ is a smooth minimal hypersurface, where $\Sigma$ is a closed set with codimension at least $8$.

\subsection{Measure theoretic homotopic spanning}\label{section measth homotopic span} In a nutshell, the idea behind our measure theoretic revision the Harrison--Pugh homotopic spanning condition is the following. Rather than asking that $S\cap\g(\SS^1)\ne\varnothing$ for every $\g\in\C$, as done in Definition \ref{def homot span closed}, we shall replace $\g$ with an open ``tube'' $T$ containing $\g(\SS^1)$, and ask that $S$, with the help of a generic ``slice'' $T[s]$ of $T$, ``disconnects'' $T$ itself into two nontrivial regions $T_1$ and $T_2$; see
    \begin{figure}\input{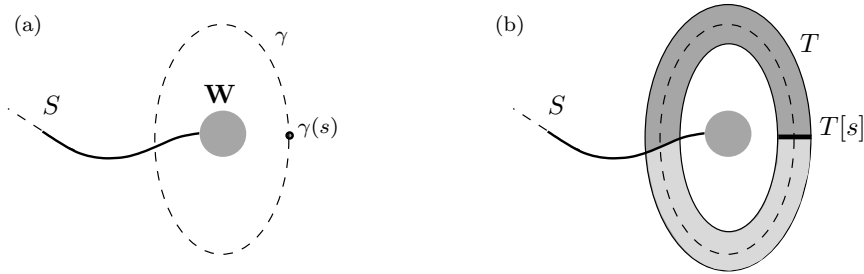}\caption{\small{(a) Homotopic spanning according to Harrison--Pugh: $S$ must intersect every curve $\g\in\C$, in particular, the $\C$-spanning property may be lost by removing a single point from $S$; (b) Homotopic spanning based on essential connectedness: for a.e. section $T[s]$ of the tube $T$ around a curve $\g\in\C$, the union $T[s]\cup S$ (essentially) disconnects $T$ (i.e., divides $T$ into two non-trivial parts, depicted here with two different shades of gray).}}\label{fig disco}\end{figure}
    Figure \ref{fig disco}. The key to make this idea work is, of course, giving a proper meaning to the word ``disconnects''.

    \medskip

    To this end, we recall the  notion of {\bf essential connectedness} introduced in \cite{CCDPM17,CCDPMSteiner} in the study of the rigidity of equality cases in Gaussian and Euclidean perimeter symmetrization inequalities. Essential connectedness is the ``right'' notion to deal with such problems since it leads to the formulation of sharp rigidity theorems, and can indeed be used to address other rigidity problems (see \cite{CagnettiPeruginiStoger,perugini,domazakis2023rigidity}). This said, it seems remarkable that the very same notion of what it means for ``one Borel set to disconnect another Borel set'' proves to be extremely effective also in the context of the present paper, which is of course very far from the context of symmetrization theory.

    \medskip

    Denoting by $T^{{\scriptscriptstyle{(t)}}}$ ($0\leq t \leq 1$) the {\bf points of density $t$} of a Borel set $T\subset\R^{n+1}$ (i.e., $x\in T^{{\scriptscriptstyle{(t)}}}$ if and only if $|T\cap B_r(x)|/\om_{n+1}\,r^{n+1}\to t$ as $r\to 0^+$, where $\om_k$ is the Lebesgue measure of the unit ball in $\R^k$), and by $\partial^e T= \mathbb{R}^{n+1}\setminus (T^{\zero} \cup T^{\one})$ the {\bf essential boundary} of $T$, given Borel sets $S$, $T$, $T_1$ and $T_2$ in $\mathbb{R}^{n+1}$, and given $n\geq 0$, we say that {\bf $S$  essentially disconnects $T$ into $\{T_1,T_2\}$}, if
    \begin{equation}\label{def essential connectedness}
    \begin{split}
    &\mbox{$\{T_1, T_2\}$ is a non-trivial Borel partition of $T$}\,,
    \\
    &\mbox{and $T^{\one}\cap \partial^e T_1 \cap \partial^e T_2$ is $\H^n$-contained in $S$}\,.
    \end{split}
    \end{equation}
    (For example, if $K$ is a set of full $\L^1$-measure in $[-1,1]$, then $S=K\times\{0\}$ essentially disconnects the unit disk in $\R^2$.)
    Moreover, we say that $T$ is {\bf essentially connected}\footnote{Whenever $T$ is of locally finite perimeter, being essentially connected is equivalent to being indecomposable.} if $\varnothing$ does not essentially disconnect $T$. The requirement that $\{T_1,T_2\}$ is a non-trivial Borel partition of $T$ means that $|T\Delta (T_1\cup T_2)|=0$ and $|T_1|\,|T_2|>0$. By saying that ``$E$ is $\H^n$-contained in $F$'' we mean that $\H^n(E\setminus F)=0$. We also notice that, in \eqref{def essential connectedness}, we have
    $T^{\one}\cap \partial^e T_1 \cap \partial^e T_2=T^{\one}\cap \partial^e T_i$ ($i=1,2$), a fact that is tacitly and repeatedly considered in the use of \eqref{def essential connectedness} in order to shorten formulas.

\medskip

With this terminology in mind, we introduce the following definition:

\begin{definition}[Measure theoretic homotopic spanning]\label{def homot span borel}
        Given a closed set $\wire$ and a spanning class $\C$ for $\wire$, the {\bf tubular spanning class} $\mathcal{T}(\C)$ associated to $\C$ is the family of triples $(\gamma, \Phi, T)$ such that $\g\in\C$, $T = \Phi(\mathbb{S}^1 \times B_1^n)$, and\footnote{Here $B_1^n=\{x\in\R^n:|x|<1\}$ and $\SS^1=\{s\in\R^2:|s|=1\}$.}
        \[
        \textup{$\Phi:\mathbb{S}^1 \times \cl B_1^n \to  \Omega$ is a diffeomorphism with $\restr{\Phi}{\mathbb{S}^1\times \{0\}}= \gamma$}\,.
        \]
        When $(\gamma, \Phi, T)\in\T(\C)$, the {\bf slice of $T$} defined by $s\in\mathbb{S}^1$ is
        \[
        T[s]=\Phi(\{s\}\times B_1^n)\,.
        \]
        Finally, we say that a Borel set $S \subset\Omega$ is {\bf $\C$-spanning} $\wire$ if for each $(\gamma,\Phi, T)\in \mathcal{T}(\C)$, $\H^1$-a.e. $s\in \mathbb{S}^1$ has the following property:
        \begin{eqnarray}\nonumber
        &&\mbox{for $\H^n$-a.e. $x\in T[s]$}
        \\\label{spanning borel intro}
          &&\mbox{$\exists$ a partition $\{T_1,T_2\}$ of $T$ s.t. $x\in\partial^e T_1 \cap \partial^e T_2$}\\ \nonumber
          &&\mbox{and s.t. $S \cup  T[s]$ essentially disconnects $T$ into $\{T_1,T_2\}$}\,.
        \end{eqnarray}
\end{definition}

Before commenting on \eqref{spanning borel intro}, we notice that the terminology of Definition \ref{def homot span borel} is coherent with that of Definition \ref{def homot span closed} thanks to the following theorem.

\begin{theorem}\label{theorem definitions equivalence intro}
Given a closed set $\wire\subset\R^{n+1}$, a spanning class $\C$ for $\wire$, and a set $S$ relatively closed in $\Om$, then $S$ is $\C$-spanning $\wire$ in the sense of Definition \ref{def homot span closed} if and only if $S$ is $\C$-spanning $\wire$ in the sense of Definition \ref{def homot span borel}.
\end{theorem}

Theorem \ref{theorem definitions equivalence intro} is proved in Appendix \ref{appendix equivalence of}. There we also comment on the delicate reason why, in formulating \eqref{spanning borel intro}, the partition $\{T_1,T_2\}$ must be allowed to depend on specific points $x\in T[s]$. This would not seem necessary by looking at the simple situation depicted in Figure \ref{fig disco}, but it is actually so when dealing with more complex situations; see Figure \ref{fig triple}.

\medskip

Homotopic spanning according to Definition \ref{def homot span borel} is clearly stable under modifications of $S$ by $\H^n$-negligible sets, but there is more to it.
    \begin{figure}
    \input{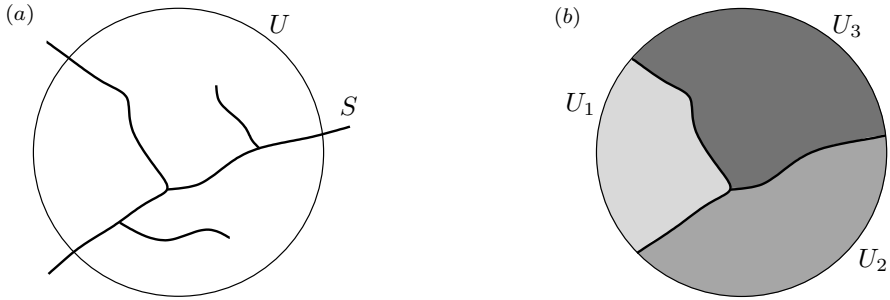}
    \caption{\small{An example of induced essential partition. The union of the boundaries of the $U_i$'s (inside of $U$) is contained in $S$, and the containment may be strict. However, the part of $S$ not contained in $U\cap\bigcup_i\pa U_i$ is not such to disconnect any of the $U_i$'s. In particular, each $U_i$ is essentially connected.}}
    \label{fig esspart}
    \end{figure}
Indeed, even a notion like ``$\H^n(S\cap T)>0$ for every $T\in\T(\C)$'' would be stable under modifications by $\H^n$-negligible sets, and would probably look more appealing in its simplicity. The catch, of course, is finding an extension of Definition \ref{def homot span closed} for which compactness theorems, like Theorem \ref{theorem first closure theorem intro} below, hold true. This is evidently not the case, for example, if one tries to work with a notion like ``$\H^n(S\cap T)>0$ for every $T\in\T(\C)$''.

\medskip

The first key insight on Definition \ref{def homot span borel} is that, if restricted to Borel sets $S$ that are locally $\H^n$-finite in $\Om$, then it can be reformulated in terms of partitions into indecomposable sets of finite perimeter. This is the content of the following theorem, whose case $S=\varnothing$ corresponds to the standard decomposition theorem for sets of finite perimeter \cite[Theorem 1]{ambrosiocaselles}. For an illustration of this result, see Figure \ref{fig esspart}.

    \begin{theorem}[Induced essential partitions (Section \ref{section essential partitions})]\label{theorem decomposition intro}
        If $U\subset\R^{n+1}$ is a bounded set of finite perimeter and $S\subset\R^{n+1}$ is a Borel set with $\H^n(S \cap U^\one)<\infty$, then there exists a unique\footnote{Uniqueness is meant modulo relabeling and modulo Lebesgue negligible modifications of the $U_i$'s.} {\bf essential partition $\{U_i\}_i$ of $U$ induced by $S$}, that is to say, $\{U_i\}_i$ is a countable partition of $U$ modulo Lebesgue negligible sets such that, for each $i$, $S$ does not essentially disconnect $U_i$.
    \end{theorem}

Given $U$ and $S$ as in the statement of Theorem \ref{theorem decomposition intro} we can define\footnote{Uniquely modulo $\H^n$-null sets thanks to Federer's theorem recalled in \eqref{federer theorem} below.} the {\bf union of the} (reduced) {\bf boundaries} (relative to $U$)  {\bf of the essential partition} induced by $S$ on $U$ by setting\footnote{Given a Borel set $E$, we denote by $\pa^*E$ its reduced boundary relative to the maximal open set $A$ wherein $E$ has locally finite perimeter.}
    \begin{equation}
    \label{def of UBEP}
    {\rm UBEP}(S;U)=U^\one\cap\bigcup_i\pa^*U_i\,.
    \end{equation}
    Two properties of ${\rm UBEP}$'s which well illustrate the concept are: first, if $\RR(S)$ denotes the rectifiable part of $S$, then ${\rm UBEP}(S;U)$ is $\H^n$-equivalent to ${\rm UBEP}(\RR(S);U)$; second, if $S^*$ is $\H^n$-contained in $S$, then ${\rm UBEP}(S;U)$ is $\H^n$-contained in ${\rm UBEP}(S;U)$; both properties are proved in Theorem \ref{theorem decomposition} (an expanded restatement of Theorem \ref{theorem decomposition intro}).

    \medskip

    We can use the concepts just introduced to provide an alternative and technically more workable characterization of homotopic spanning in the measure theoretic setting. This is the content of our first main result, which is illustrated in Figure \ref{fig Scase}.

    \begin{theorem}[Homotopic spanning for locally $\H^n$-finite sets (Section \ref{section spanning and partitions})]\label{theorem spanning with partition S case}
    If $\wire\subset\R^{n+1}$ is a closed set in $\R^{n+1}$, $\C$ is a spanning class for $\wire$, and $S\subset\Om$ is locally $\H^n$-finite in $\Om$, then $S$ is $\C$-spanning $\wire$ if and only if for every $(\gamma,\Phi, T)\in \mathcal{T}(\C)$ we have that, for $\H^1$-a.e. $s\in\SS^1$,
    \begin{eqnarray}\label{spanning and the S partition equation S case}
    &&\mbox{$T[s]$ is $\H^n$-contained in ${\rm UBEP}(S\cup T[s];T)$}\,.
    \end{eqnarray}
    \end{theorem}

      \begin{figure}
    \input{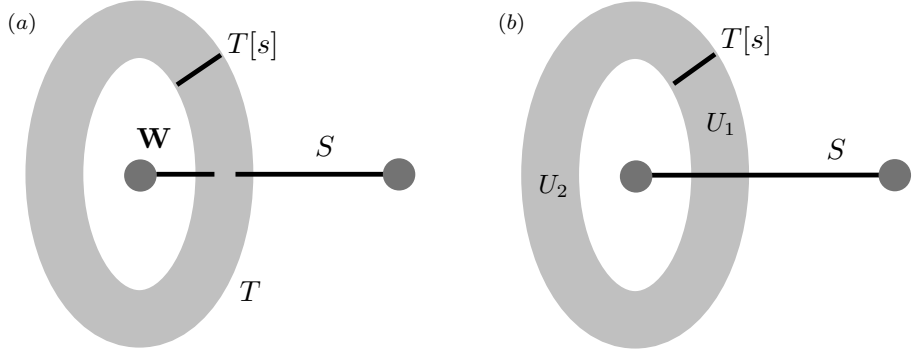}
    \caption{\small{With $\wire$ consisting of two disks in the plane, and $T$ a test tube for the $\C$-spanning condition: (a) $S$ consists of a segment with a gap: since the gap is inside of $T$, the essential partition of $T$ induced by $S\cup T[s]$ consists of only one set, $U_1=T$, so that $T\cap\pa^*U_1=\varnothing$ and \eqref{spanning and the S partition equation S case} cannot hold; (b) $S$ consists of a full segment and in this case (with the possible exception of a choice of $s$ such that $T[s]$ is contained in $S$), the essential partition of $T$ induced by $S\cup T[s]$ consists of two sets $\{U_1,U_2\}$, such that $T[s]\subset T\cap\pa^*U_1\cap\pa^*U_2$; in this case \eqref{spanning and the S partition equation S case} holds.}}
    \label{fig Scase}
    \end{figure}

\subsection{Direct Method on generalized soap films and Gauss' capillarity}\label{section main results} The most convenient setting for addressing the existence of minimizers in Gauss' capillarity theory is of course that of sets of finite perimeter \cite{FinnBOOK,maggiBOOK}. However, if the notion of homotopic spanning is limited to closed sets, as it is the case when working with Definition \ref{def homot span closed}, then one cannot directly use homotopic spanning on sets of finite perimeter, and this is the reason behind the specific formulation \eqref{def of psi} of $\psi(v)$ used in  \cite{MaggiScardicchioStuvard,KingMaggiStuvard}. Equipped with Definition \ref{def homot span borel} we can now formulate Gauss' capillarity theory with homotopic spanning conditions directly on sets of finite perimeter. We shall actually consider {\it two} different possible formulations
\begin{eqnarray*}
&&\psi_{\rm bk}(v)=\inf\Big\{\H^n(\Om\cap\pa^*E):\mbox{$|E|=v$ and $\Om\cap(\pa^* E\cup E^\one)$ is $\C$-spanning $\wire$}\Big\}\,,
\\
&&\psi_{\rm bd}(v)=\inf\Big\{\H^n(\Om\cap\pa^*E):\mbox{$|E|=v$ and $\Om\cap \pa^*E$ is $\C$-spanning $\wire$}\Big\}\,,
\end{eqnarray*}
where the subscripts ``bk'' and ``bd'' stand to indicate that the spanning is prescribed via the {\it bulk} of $E$ (that is, in measure theoretic terms, via the set $\Om\cap(\pa^*E\cup E^\one)$ or via the (reduced) boundary of $E$. Inspired by the definition of the class $\K$ introduced in \eqref{def of K}, we also introduce the class $\K_{\rm B}$ of {\bf generalized soap films} defined by
\begin{eqnarray}\label{def of KB}
&&\mathcal{K}_{\rm B}=\Big\{(K,E):\mbox{$K$ and $E$ are Borel subsets of $\Om$,}
\\\nonumber
&&\hspace{2.7cm}\mbox{$E$ has locally finite perimeter in $\Om$ and $\partial^*E\cap \Omega \shn K$}\Big\}\,.
\end{eqnarray}
Here the subscript ``B'' stands for ``Borel'', and $\K_{\rm B}$ stands as a sort of measure-theoretic version of $\K$.

\medskip

In the companion paper \cite{novackGENMIN} the following relaxation formulas for problems $\psi_{\rm bk}$ and $\psi_{\rm bd}$ are proved,
\begin{equation}
  \label{relaxation formulas}
  \psi_{\rm bk}(v)=\Psi_{\rm bk}(v)\,,\qquad   \psi_{\rm bd}(v)=\Psi_{\rm bd}(v)\,,\qquad\forall v>0\,,
\end{equation}
where the following minimization problems on $\KK_{\rm B}$ are introduced
\begin{eqnarray}
\label{def of Psibk}
&&\Psi_{\rm bk}(v)=\inf\Big\{\F_{\rm bk}(K,E):(K,E)\in\mathcal{K}_{\rm B}\,,|E|=v\,,\mbox{$K\cup E^{\one}$ is $\C$-spanning $\wire$}\Big\}\,,\hspace{1cm}
\\
\label{def of Psibd}
&&\Psi_{\rm bd}(v)=\inf\Big\{\F_{\rm bd}(K,E):(K,E)\in\mathcal{K}_{\rm B}\,,|E|=v\,,\mbox{$K$ is $\C$-spanning $\wire$}\Big\}\,.
\end{eqnarray}
Here $\F_{\rm bk}$ and $\F_{\rm bd}$ are the relaxed energies defined for $(K,E)\in\KK_{\rm B}$ and $A\subset\Om$ as
\begin{eqnarray}
      \label{def of Fb}
      &&\F_{\rm bk}(K,E;A)=2\,\H^n(A\cap K\cap E^{\zero})+\H^n(A\cap\pa^*E)\,,
      \\\label{def of F}
      &&\F_{\rm bd}(K,E;A)=2\,\H^n(A\cap K\setminus\pa^*E)+\H^n(A\cap\pa^*E)\,,
\end{eqnarray}
(We also set, for brevity, $\F_{\rm bk}(K,E):=\F_{\rm bk}(K,E;\Om)$ and $\F_{\rm bd}(K,E):=\F_{\rm bd}(K,E;\Om)$.) We refer to these problems, respectively, as the ``bulk-spanning'' or ``boundary-spanning'' Gauss' capillarity models. In this paper we shall directly work with these relaxed models. In particular, the validity of \eqref{relaxation formulas}, although of definite conceptual importance, is not playing any formal role in our deductions.

\medskip

A first remark concerning the advantage of working with the relaxed problems $\Psi_{\rm bk}$ and $\Psi_{\rm bd}$ rather than with their ``classical'' counterparts $\psi_{\rm bk}$ and $\psi_{\rm bd}$ is that while the latter two with $v=0$ are trivial (sets with zero volume have zero distributional perimeter), the problems $\Psi_{\rm bk}(0)$ and $\Psi_{\rm bd}(0)$ are actually non-trivial, equal to each other, and amount to a measure-theoretic version of the Harrison--Pugh formulation of Plateau's problem $\ell$ introduced in \eqref{hp problem intro}: more precisely, if we set
\begin{equation}
    \label{def of ellB}
    \ell_{\rm B}:=\frac{\Psi_{\rm bk}(0)}2=\frac{\Psi_{\rm bd}(0)}2=\inf\Big\{\H^n(S):\mbox{$S$ is a Borel set $\C$-spanning $\wire$}\Big\}\,,
    \end{equation}
then, by Theorem \ref{theorem definitions equivalence intro}, we evidently have $\ell_{\rm B}\le\ell$; and, as we shall prove in the course of our analysis, we actually have that $\ell=\ell_{\rm B}$ as soon as $\ell<\infty$.

\medskip

Our second main result concerns the applicability of the Direct Method on the competition classes of $\Psi_{\rm bk}(v)$ and $\Psi_{\rm bd}(v)$.

    \begin{theorem}[Direct Method for generalized soap films (Sections \ref{section closure theorem basic} and \ref{section closure theorems sharp})]\label{theorem first closure theorem intro} Let $\wire$ be a closed set in $\mathbb{R}^{n+1}$, $\C$ a spanning class for $\wire$, $\{(K_j,E_j)\}_j$ be a sequence in $\K_{\rm B}$ such that $\sup_j\H^n(K_j)<\infty$, and let a Borel set $E$ and Radon measures $\mu_{\rm bk}$ and $\mu_{\rm bd}$ in $\Om$ be such that $E_j\toloc E$ and
    \begin{eqnarray*}
    &&\H^n\mres (\Om\cap\pa^*E_j) + 2\,\H^n \mres (\mathcal{R}(K_j) \cap E_j^\zero) \weakstar \mu_{\rm bk}\,,
    \\
    &&\H^n\mres (\Om\cap\pa^*E_j) + 2\,\H^n \mres (\mathcal{R}(K_j) \setminus \pa^* E_j) \weakstar \mu_{\rm bd}\,,
    \end{eqnarray*}
    as $j\to\infty$. Then:

    \medskip

    \noindent {\bf (i) Lower semicontinuity:} the sets
    \begin{eqnarray*}
      K_{\rm bk}\!\! &:=&\!\!\big(\Om\cap\partial^* E\big) \cup \Big\{x\in \Omega \cap E^\zero : \theta^n_*(\mu_{\rm bk})(x)\geq 2 \Big\}\,,
      \\
      K_{\rm bd}\!\!&:=&\!\!\big(\Om\cap\partial^* E\big) \cup \Big\{x\in \Omega \setminus\pa^*E : \theta^n_*(\mu_{\rm bd})(x)\geq 2 \Big\}\,,
    \end{eqnarray*}
    are such that $(K_{\rm bk},E),(K_{\rm bd},E)\in\K_{\rm B}$ and
    \begin{eqnarray*}
      \mu_{\rm bk}\!\!&\ge&\!\! \H^n\mres (\Om\cap\pa^*E) + 2\,\H^n \mres (K_{\rm bk} \cap E^\zero)\,,
      \\
      \mu_{\rm bd}\!\!&\ge&\!\!\H^n\mres (\Om\cap\pa^*E) + 2\,\H^n \mres (K_{\rm bd} \setminus\pa^*E)\,,
    \end{eqnarray*}
    with
    \begin{equation}\nonumber
      \liminf_{j\to\infty}\F_{\rm bk}(K_j,E_j)\ge\F_{\rm bk}(K_{\rm bk},E)\,,\qquad
        \liminf_{j\to\infty}\F_{\rm bd}(K_j,E_j)\ge\F_{\rm bd}(K_{\rm bd},E)\,.
    \end{equation}

    \medskip

    \noindent {\bf (ii) Closure:} we have that
    \begin{eqnarray*}
      &&\mbox{if $K_j\cup E_j^\one$ is $\C$-spanning $\wire$ for every $j$,}
      \\
      &&\mbox{then $K_{\rm bk}\cup E^\one$ is $\C$-spanning $\wire$}\,,
    \end{eqnarray*}
    and that
    \begin{eqnarray*}
      &&\mbox{if $K_j$ is $\C$-spanning $\wire$ for every $j$,}
      \\
      &&\mbox{then $K_{\rm bd}$ is $\C$-spanning $\wire$}\,.
    \end{eqnarray*}
    \end{theorem}

    The delicate part of Theorem \ref{theorem first closure theorem intro} is proving the closure statements. This will require first to extend the characterization of homotopic spanning from locally $\H^n$-finite sets to generalized soap films (Theorem \ref{theorem spanning with partition}), and then to discuss the behavior under weak-star convergence of the associated Radon measures of the objects appearing in conditions like \eqref{spanning and the S partition equation S case} (Theorem \ref{theorem basic closure for homotopic}).

\subsection{Existence of minimizers in $\Psi_{\rm bk}(v)$ and convergence to $\ell$}\label{subsection existence and convergence} From this point onward, we focus our analysis on the bulk-spanning relaxation $\Psi_{\rm bk}(v)$ of Gauss' capillarity. There are a few important reasons for this choice: (i) from the point of view of physical modeling, working with the boundary or with the bulk spanning conditions seem comparable; (ii) the fact that $\Psi_{\rm bk}(0)=\Psi_{\rm bd}(0)$ suggest that, at small values of $v$, the two problems should actually be equivalent (have the same infima and the same minimizers); (iii) the bulk spanning variant is the one which is relevant for the approximation of Plateau-type singularities with solutions of the Allen--Cahn equations discussed in \cite{MNR2}; (iv) despite their similarities, carrying over the following theorems for both problems would require the repeated introduction of two versions of many arguments, with a significant increase in length, and possibly with at the expense of clarity.

\medskip

The following theorem provides the starting point in the study of $\Psi_{\rm bk}(v)$.

    \begin{theorem}[Existence of minimizers and vanishing volume limit for $\Psi_{\rm bk}$ (Section \ref{section existence theorems})]\label{thm existence EL for bulk}
    If $\wire$ is a compact set in $\mathbb{R}^{n+1}$ and $\C$ is a spanning class for $\wire$ such that $\ell<\infty$, then
    \begin{equation}
      \label{ellb equals ell}
    \ell_{\rm B}=\ell\,,
    \end{equation}
    and, moreover:

    \medskip

    \noindent {\bf (i) Existence of minimizers and Euler--Lagrange equation:} for every $v>0$ there exist minimizers $(K,E)$ of $\Psi_{\rm bk}(v)$ such that $(K,E)\in\KK$ and both $E$ and $K$ are bounded; moreover, there is $\l\in\R$ such that
    \begin{align}\label{euler lagrange equation bulk}
        \lambda \int_{\partial^* E} X \cdot \nu_{E} \,d\H^n = \int_{\partial^* E} \mathrm{div}^{K}\,X\,d\H^n + 2\int_{K \cap E^{\zero}} \mathrm{div}^{K}\, X\,d\H^n\,,
    \end{align}
    for every $X\in C^1_c(\mathbb{R}^{n+1};\mathbb{R}^{n+1})$ with $X\cdot \nu_\Omega =0$ on $\partial \Omega$;

    \medskip

    \noindent {\bf (ii) Regularity from the Euler--Lagrange equations:} if $(K,E)\in\KK$ is a minimizer of either $\Psi_{\rm bk}(v)$,
     then there is a closed set $\Sigma\subset K$, with empty interior in $K$, such that $K\setminus\Sigma$ is a smooth hypersurface; moreover, $K\setminus(\Sigma\cup\pa E)$ is a smooth minimal hypersurface, $\Om\cap\pa^*E$ is a smooth hypersurface with mean curvature constantly equal to $\l$, and $\H^n(\Sigma\setminus\pa E)=0$; in particular, $\Om\cap(\pa E\setminus\pa^*E)$ has empty interior in $K$;

    \medskip

    \noindent {\bf (iii) Convergence to the Plateau problem:} if $(K_j,E_j)$ is a sequence of minimizers for $\Psi_{\rm bk}(v_j)$ with $v_j\to 0^+$, then there exists a minimizer $S$ of $\ell$ such  that, up to extracting subsequences, as Radon measures in $\Om$,
    \begin{align}
    \label{convergence as measures}
            \H^n\mres(\partial^* E_j \cap \Omega) + 2\H^n\mres (K_j \cap E_j^{\zero}) \weakstar 2\H^n\mres S\,,
    \end{align}
    as $j\to\infty$;
    In particular, $\Psi_{\rm bk}(v)\to 2\,\ell=\Psi_{\rm bk}(0)$ as $v\to 0^+$.
    \end{theorem}

The conclusions of Theorem \ref{thm existence EL for bulk} about $\Psi_{\rm bk}(v)$ can be read in parallel to the conclusions about $\psi(v)$ obtained in \cite{KingMaggiStuvard}. The crucial difference is that, in place of the ``weak'' minimality inequality \eqref{kms minimizers intro}, which in this context would be equivalent to $\F_{\rm bk}(K,E)\le\H^n(\Om\cap\pa^*E')$ for every competitor $E'$ in $\psi_{\rm bk}(v)$, we now have the proper minimality inequality
\begin{equation}
  \label{proper min}
  \F_{\rm bk}(K,E)\le \F_{\rm bk}(K',E')
\end{equation}
for every competitor $(K',E')$ in $\Psi_{\rm bk}(v)$. Not only the final conclusion is stronger, but the proof is also entirely different: whereas \cite{KingMaggiStuvard} required the combination of a whole bestiary of specific competitors (like the cup, cone, and slab competitors described therein) with the full force of Preiss' theorem, the approach presented here seems more robust as it does not exploit any specific geometry, and it is squarely rooted in the basic theory of sets of finite perimeter.

\subsection{Equilibrium across transition lines in wet soap films}\label{section validation} We now formalize the validation of \eqref{third pl} for soap films in the form of a sharp regularity theorem for minimizers $(K,E)$ of $\Psi_{\rm bk}(v)$.

\medskip

The starting point to obtain this result is the connection between homotopic spanning and partitions into indecomposable sets of finite perimeter established in Theorem \ref{theorem spanning with partition S case}/Theorem \ref{theorem spanning with partition}. This connection hints at the possibility of showing that if $(K,E)$ is a minimizer of $\Psi_{\rm bk}(v)$, then the elements $\{U_i\}_i$ of the essential partition of $\Om$ induced by $K\cup E^\one$ are actually $(\Lambda,r_0)$-minimizers of the perimeter in $\Om$, i.e., there exist $\Lambda$ and $r_0$ positive constants such that
    \[
    P(U_i;B_r(x))\le P(V;B_r(x))+\Lambda\,|V\Delta U_i|\,,
    \]
whenever $V\Delta U_i\cc \Om$ and $\diam(V\Delta U_i)<r_0$. The reason why this property is not obvious is that proving the $(\Lambda,r_0)$-minimality of $U_i$ requires working with {\it arbitrary local competitors} $V_i$ of $U_i$. However, when working with homotopic spanning conditions, checking the admissibility of competitors is the notoriously delicate heart of the matter -- as reflected in the fact that only very special classes of competitors have been considered in the literature (see, e.g., the cup and cone competitors and the Lipschitz deformations considered in \cite{DLGM}, the slab competitors and exterior cup competitors of \cite{KingMaggiStuvard}, etc.).

\medskip

The idea used to overcome this difficulty, which is illustrated in
    \begin{figure}
    \input{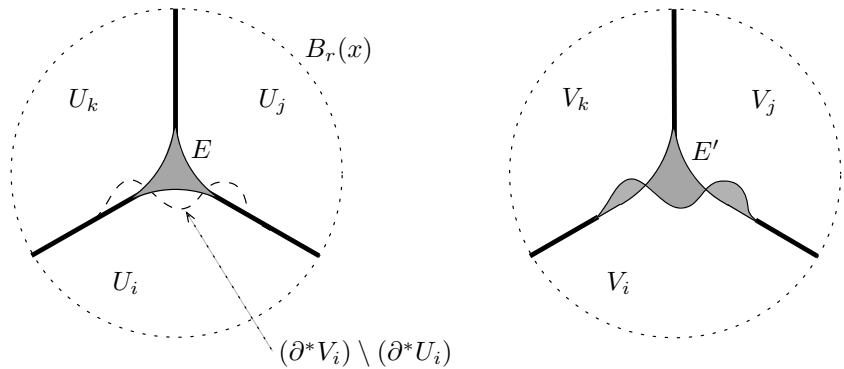}
    \caption{\small{On the left, a minimizer $(K,E)$ of $\Psi_{\rm bk}(v)$, and the essential partition induced by $(K,E)$ in a ball $B_r(x)$; the multiplicity $2$ part of $K\cap B_r(x)$ are depicted with bold lines, to distinguish them from the multiplicity one parts in $B_r(x)\cap\pa^*E$. On the right, a choice of $(K',E')$ that guarantees both the energy gap identity \eqref{energy difference} and the $\H^n$-containment \eqref{still contains} needed to preserve homotopic spanning. The volume constraint can of course be restored as a lower order perimeter perturbation by taking a diffeomorphic image of $(K',E')$, an operation that trivially preserves homotopic spanning.}}
    \label{fig lambda}
    \end{figure}
    Figure \ref{fig lambda}, is the following. By Theorem \ref{theorem decomposition intro}, we can locally represent $\F_{\rm bk}(K,E;B_r(x))$ as the sum of perimeters $P(U_i;B_r(x))+P(U_j;B_r(x))+P(U_k;B_r(x))$. Given a local competitor $V_i$ for $U_i$ we can carefully define a competitor $(K',E')$ so that the elements of the essential partition induced by $K'\cup (E')^\one$ in $\Om$, that can be used to represent  $\F_{\rm bk}(K',E';B_r(x))$ as the sum $P(V_i;B_r(x))+P(V_j;B_r(x))+P(V_k;B_r(x))$, are such that
    \begin{eqnarray}\label{energy difference}
    \F_{\rm bk}(K',E';B_r(x))-\F_{\rm bk}(K,E;B_r(x))=P(V;B_r(x)) - P(U_i;B_r(x))\,.
    \end{eqnarray}
    The trick is that by suitably defining $K'$ and $E'$ we can recover the entirety of $B_r(x)\cap\pa^*U_j$ and $B_r(x)\cap\pa^*U_k$ by attributing different parts of these boundaries to different terms in the representation of $\F_{\rm bk}(K',E';B_r(x))$. In other words we are claiming that things can be arranged so that we still have
    \begin{equation}
      \label{still contains}
      B_r(x)\cap\big(\pa^*U_j\cap\pa^*U_k)\shn K'\cup (E')^\one\,.
    \end{equation}
    The fact that we have been able to preserve all but one reduced boundary among those of the elements of the essential partition of $B_r(x)$ induced by $(K,E)$ is enough to shows that $K'\cup (E')^\one$ is still $\C$-spanning $\wire$ by means of Theorem \ref{theorem spanning with partition S case}/Theorem \ref{theorem spanning with partition}.

    \medskip

    By the regularity theory of $(\Lambda,r_0)$-perimeter minimizers (see, e.g. \cite[Part III]{maggiBOOK}) we can deduce the $C^{1,\a}$-regularity of the elements of the partition (away from a closed singular set with area minimizing dimensional bounds). This is already sufficient to prove the continuity of the normal across $\Om\cap(\pa E\setminus\pa^*E)$, but it also allows us to invoke the regularity theory for free boundaries in the double membrane problem, and to obtain the following sharp regularity result, with which we conclude our introduction.

    \begin{theorem}[Equilibrium along transition lines for soap films (Section \ref{section regularity transition regions})]\label{thm regularity for bulk}
    If $\wire$ is a compact set in $\mathbb{R}^{n+1}$, $\mathcal{C}$ is a spanning class for $\wire$ such that $\ell<\infty$, $v>0$, and $(K_*,E_*)$ is a minimizer of $\Psi_{\rm bk}(v)$, then there is $(K,E)\in\KK$ such that $K$ is $\H^n$-equivalent to $K_*$, $E$ is Lebesgue equivalent to $E_*$, $(K,E)$ is a minimizer of $\Psi_{\rm bk}(v)$, both $E$ and $K$ are bounded, $K\cup E$ is $\C$-spanning $\wire$, and
    \begin{equation}
      \label{no collapsing}
      K\cap E^\one=\varnothing\,;
    \end{equation}
    in particular, $K$ is the disjoint union of $\Om\cap\pa^*E$, $\Om\cap(\pa E \setminus\pa^*E)$, and $K\setminus\pa E$.

    \medskip

    Moreover, there is a closed set $\Sigma\subset K$ with the following properties:

    \medskip

    \noindent {\bf (i):} $\Sigma=\varnothing$ if $1\leq n \leq 6$, $\Sigma$ is locally finite in $\Omega$ if $n=7$, and $\mathcal{H}^{s}(\Sigma)=0$ for every $s>n-7$ if $n\geq 8$;

    \medskip

    \noindent {\bf (ii):} $(\Om\cap\pa^*E)\setminus\Sigma$ is a smooth hypersurface with constant mean curvature (denoted by $\l$ if computed with respect to $\nu_E$);

    \medskip

    \noindent {\bf (iii):} $(K\setminus\pa E)\setminus\Sigma$ is a smooth minimal hypersurface;

    \medskip

    \noindent {\bf (iv):} if $\Om\cap(\pa E\setminus\pa^*E)\setminus\Sigma\ne\varnothing$, then $\l<0$; moreover, for every $x\in\Om\cap(\pa E\setminus\pa^*E)\setminus\Sigma$, $K$ is the union of two $C^{1,1}$-hypersurfaces that detach tangentially at $x$; more precisely, there are $r>0$, $\nu\in\SS^n$, $u_1,u_2\in C^{1,1}(\DD_r^\nu(x))$ such that
    \[
    u_1(x)=u_2(x)=0\,,\qquad\mbox{$u_1\le u_2$ on $\DD_r^\nu(x)$}\,,
    \]
    with $\{u_1<u_2\}$ and ${\rm int}\{u_1=u_2\}$ both non-empty, and
    \begin{eqnarray}\label{grafico K}
    \CC_r^\nu(x)\cap K&=&\cup_{i=1,2}\big\{y+u_i(y)\,\nu:y\in\DD_r^\nu(x)\big\}\,,
    \\\label{grafico pa E}
    \CC_r^\nu(x)\cap \pa^*E&=&\cup_{i=1,2}\big\{y+u_i(y)\nu:y\in\{u_1<u_2\}\big\}\,,
    \\\label{grafico E}
    \CC_r^\nu(x)\cap E&=&\big\{y+t\,\nu:t\in\big(u_1(y),u_2(y)\big)\big\}\,.
    \end{eqnarray}
    Here,
    \begin{eqnarray*}
      &&\DD_\nu^r(x)=x+\{y\in\nu^\perp:|y|<r\}\,,
      \\
      &&\CC_\nu^r(x)=x+\{y+t\,\nu:y\in\nu^\perp\,,|y|<r\,,|t|<r\}\,.
    \end{eqnarray*}

    \medskip

\noindent {\bf (v):} we have
\[
\Gamma:=\Om\cap(\pa E\setminus\pa^*E)=\Gamma_{\rm reg}\cup\Gamma_{\rm sing}\,,\qquad \Gamma_{\rm reg}\cap\Gamma_{\rm sing}=\varnothing\,,
\]
where: $\Gamma_{\rm reg}$ is relatively open in $\Gamma$ and for every $x\in\Gamma_{\rm reg}$ there are $r>0$ and $\beta\in(0,1)$ such that $\Gamma_{\rm reg}\cap B_r(x)$ is a $C^{1,\b}$-embedded $(n-1)$-dimensional manifold; $\Gamma_{\rm sing}$ is relatively closed in $\Gamma$ and can be partitioned into a family $\{\Gamma_{\rm sing}^k\}_{k=0}^{n-1}$ where, for each $k$, $\Gamma_{\rm sing}^k$ is locally $\H^k$-rectifiable in $\Om$.
\end{theorem}

\subsection{Equilibrium across transition lines in wet foams}\label{section foams intro} Based on the descriptions provided in \cite{weaireBOOK,foamchapter}, an effective mathematical model for dry foams at equilibrium in a container is that of locally perimeter minimizing clusters, originating with different terminology in \cite{Almgren76}, and presented in
\cite[Part IV]{maggiBOOK} as follows. Given an open set $\Om\subset\R^{n+1}$, a locally perimeter minimizing clusters is a finite Lebesgue partition $\{U_i\}_i$ of $\Om$ into sets of finite perimeter such that, for some $r_0>0$,
\begin{equation}
  \label{dry foam}
  \sum_i P(U_i;B)\le \sum_iP(V_i;B)
\end{equation}
whenever $B\cc\Om$ is a ball with radius less than $r_0$, and $\{V_i\}_i$ is a Lebesgue partition of $\Om$ with $V_i\Delta U_i\cc B$ and $|V_i|=|U_i|$ for every $i$. The previously cited results of Almgren and Taylor \cite{Almgren76,taylor76} imply that, up to modification of the $U_i$'s by sets of zero Lebesgue measure, when $n=2$, $K=\Om\cap\bigcup_i\pa U_i$ is a closed subset of $\Om$ that is locally $C^{1,\a}$-diffeomorphic to a plane, a $Y$-cone, or a $T$-cone; moreover, the part of $K$ that is a surface is actually smooth and each of its connected component has constant mean curvature. Similar results holds when $n=1$ (by elementary methods) and when $n\ge 3$ (by exploiting \cite{ColomboEdelenSpolaor}).

\medskip

The theory for the relaxed capillarity energy $\F_{\rm bk}$ developed in this paper provides an option for modeling wet foams. Again based on the descriptions provided in \cite{weaireBOOK,foamchapter}, the following seems to be a reasonable model for wet foams at equilibrium in a container. Given an open set $\Omega\subset\R^{n+1}$ we model wet foams by introducing the class
\[
\KK_{{\rm foam}}
\]
of those $(K,E)\in\KK_{\rm B}$ such that, for some positive constants $\Lambda_0$ and $r_0$,
\begin{equation}
  \label{def of wet foams}
  \F_{\rm bk}(K,E;B)\le \F_{\rm bk}(K',E';B)+\,\Lambda_0\,|E\Delta E'|
\end{equation}
whenever $B$ is a ball compactly contained in $\Om$ and with radius less than $r_0$, and $(K',E')\in\KK_{\rm B}$ is such that $(K\Delta K')\cup (E\Delta E')\cc B$ and there are finite Lebesgue partitions $\{U_i\}_i$ and $\{U'_i\}_i$ of $B$ induced, respectively, by $K\cup E^\one$ and by $K'\cup(E')^\one$, such that $|U_i|=|U_i'|$ for every $i$. Notice that inclusion of the term $\Lambda_0\,|E\Delta E'|$ in \eqref{def of wet foams} allows for the inclusion of energy perturbations due to gravity or other forces. Lemma \ref{lemma breakdown of F via open components} will clarify that by taking $(K,E)\in\KK_{\rm foam}$ with $|E|=0$ we obtain a slightly more general notion of dry foam than the one proposed in \eqref{dry foam}.

\begin{theorem}[Equilibrium along transition lines for soap films (Section \ref{section foams proof})]
  \label{theorem foams} If $\Omega\subset\R^{n+1}$ is open and $(K_*,E_*)\in\KK_{\rm foam}$, then there is $(K,E)\in\KK\cap\KK_{\rm foam}$ such that $K$ is $\H^n$-equivalent to $K_*$, $E$ Lebesgue equivalent to $E_*$, $K\cap E^\one=\varnothing$, and such that, for every ball $B\cc\Om$, the open connected components $\{U_i\}_i$ of $B\setminus (K\cup E)$ are such that each $U_i$ is (Lebesgue equivalent to an) open set with $C^{1,\a}$-boundary in $B\setminus\Sigma$. Here $\Sigma$ is a closed subset of $\Om$ with $\Sigma=\varnothing$ if $1\leq n \leq 6$, $\Sigma$ locally finite in $\Omega$ if $n=7$, and $\mathcal{H}^{s}(\Sigma)=0$ for every $s>n-7$ if $n\geq 8$.
\end{theorem}

\subsection*{Organization of the paper} The sections of the paper contain the proofs of the main theorems listed above, as already specified in the statements. To these section we add three appendices. In Appendix \ref{appendix equivalence of}, as already noted, we prove the equivalence of Definition \ref{def homot span closed} and Definition \ref{def homot span borel}. In Appendix \ref{appendix every} we prove that, with some regularity of $\partial\Om$, {\it every} minimizing sequence of $\Psi_{\rm bk}(v)$ is converging to a minimizers, without need for modifications at infinity: this is, strictly speaking, not needed to prove Theorem \ref{thm existence EL for bulk}, but it is a result of its own conceptual interest, it will be crucial for the analysis presented in \cite{MNR2}, and it is easily discussed here in light of the proof of Theorem \ref{thm existence EL for bulk}. Finally, Appendix \ref{sec: geometric remark appendix} contains an elementary lemma concerning the use of homotopic spanning in the plane that, to our knowledge, has not been proved in two dimensions.

\subsection*{Acknowledgements} We thank Guido De Philippis, Darren King, Felix Otto, Antonello Scardicchio, Salvatore Stuvard, and Bozhidar Velichkov for several interesting discussions concerning these problems. FM has been supported by NSF Grant DMS-2247544. FM, MN, and DR have been supported by NSF Grant DMS-2000034 and NSF FRG Grant DMS-1854344. MN has been supported by NSF RTG Grant DMS-1840314.

\subsection*{Notation} {\bf Sets and measures:} We denote by $B_r(x)$ (resp., $B_r^k(x)$) the open ball of center $x$ and radius $r$ in $\R^{n+1}$ (resp., $\R^k$), and omit $(x)$ when $x=0$. We denote by $\cl(X)$, ${\rm int}(X)$, and $I_r(X)$ the closure, interior and open $\e$-neighborhood of $X\subset\R^k$. We denote by $\L^{n+1}$ and $\H^s$ the Lebesgue measure and the $s$-dimensional Hausdorff measure on $\R^{n+1}$, $s\in[0,n+1]$. If $E\subset\R^k$, we set $|E|=\L^k(E)$ and $\om_k=|B_1^k|$. We denote by $E^{{\scriptscriptstyle{(t)}}}$, $t\in[0,1]$, the {\bf points of density $t$} of a Borel set $E\subset\R^{n+1}$, so that $E$ is $\L^{n+1}$-equivalent to $E^{\one}$, and, for every pair of Borel sets $E,F\subset\R^{n+1}$,
\begin{equation}
    \label{X cup Y zero}
      (E\cup F)^\zero=E^\zero\cap F^\zero\,.
\end{equation}
We define by $\pa^eE=\R^{n+1}\setminus(E^{\zero}\cup E^{\one})$ the {\bf essential boundary} of $E$. Given Borel sets $E_j,E\subset \Om$ we write
\[
    E_j\to E\,,\qquad E_j\toloc E\,,
    \]
    when, respectively, $|E_j\Delta E|\to 0$ or $|(E_j\Delta E)\cap \Om'|\to 0$ for every $\Om'\cc\Om$, as $j\to\infty$. Given a Radon measure $\mu$ on $\R^{n+1}$, the $k$-dimensional lower density of $\mu$ is the Borel function $\theta^k_*(\mu):\R^{n+1}\to [0,\infty]$ defined by
    \begin{align*}
        \theta^k_*(\mu)(x) = \liminf_{r\to 0^+}\frac{\mu(\cl(B_r(x)))}{\omega_k r^k}\,.
    \end{align*}
    We repeatedly use the fact that, if $\theta^k_*(\mu)\ge\l$ on some Borel set $K$ and for some $\l\ge0$, then $\mu\ge\l\,\H^k\llcorner K$; see, e.g. \cite[Theorem 6.4]{maggiBOOK}.

    \medskip

    \noindent {\bf Rectifiable sets:} Given an integer $0\le k\le n+1$, a Borel set $S\subset\R^{n+1}$ is {\bf locally $\H^k$-rectifiable} in an open set $\Om$ if $S$ is locally $\H^k$-finite in $\Om$ and $S$ can be covered, modulo $\H^k$-null sets, by a countable union of Lipschitz images of $\R^k$ in $\R^{n+1}$. We say that $S$ is {\bf purely $\H^k$-unrectifiable} if $\H^k(S\cap M)=0$ whenever $M$ is a Lipschitz image of $\R^k$ into $\R^{n+1}$. Finally, we recall that if $S$ is a locally $\H^k$-finite set in $\Om$, then there is a pair $(\RR(S),\P(S))$ of Borel sets, uniquely determined modulo $\H^k$-null sets, and that are thus called, with a slight abuse of language, {\it the} {\bf rectifiable part} and {\it the} {\bf unrectifiable part} of $S$, so that $\RR(S)$ is locally $\H^k$-rectifiable in $\Om$, $\P(S)$ is purely $\H^k$-unrectifiable, and $S=\RR(S)\cup\P(S)$; see, e.g. \cite[13.1]{Simon}.

    \medskip

    \noindent {\bf Sets of finite perimeter:} If $E$ is a Borel set in $\R^{n+1}$ and $D1_E$ is the distributional derivative of the characteristic function of $E$, then we set $\mu_E=-D1_E$. If $A$ is the {\it largest open set} of $\R^{n+1}$ such that $\mu_E$ is a Radon measure in $A$ (of course it could be $A=\varnothing$), then $E$ is of locally finite perimeter in $A$ and the reduced boundary $\pa^*E$ of $E$ is defined as the set of those $x\in A\cap\spt\mu_E$ such that $\mu_E(B_r(x))/|\mu_E|(B_r(x))$ has a limit $\nu_E(x)\in \SS^n$ as $r\to 0^+$. Moreover, we have the general identity (see \cite[(12.12) $\&$ pag. 168]{maggiBOOK})
    \begin{equation}
      \label{sofp1}
      A\cap\cl(\pa^*E)=A\cap\spt\mu_E=\Big\{x\in A:0<|E\cap B_r(x)|<|B_r(x)|\,\,\forall r>0\Big\}\subset A\cap\pa E\,.
    \end{equation}
    By De Giorgi's rectifiability theorem, $\pa^*E$ is locally $\H^n$-rectifiable in $A$, $\mu_E=\nu_E\,\H^n\mres(A\cap\pa^*E)$ on $A$, and $\pa^*E\subset A\cap E^\half\subset A\cap \pa^eE$, and
    \begin{equation}
      \label{degiorgi rect th}
      (E-x)/r \toloc H_{E,x}:=\big\{y\in\R^{n+1}:y\cdot\nu_E(x)<0\big\}\,,\qquad\mbox{as $r\to 0^+$}\,.
    \end{equation}
    By a result of Federer,
    \begin{equation}
      \label{federer theorem}
      \mbox{$A$ is $\H^n$-contained in $E^{\zero}\cup E^{\one}\cup \pa^*E$}\,;
    \end{equation}
    in particular, $\pa^*E$ is $\H^n$-equivalent to $A\cap\pa^eE$, a fact frequently used in the following. By {\it Federer's criterion for finite perimeter}, if $\Om$ is open and $E$ is a Borel set, then
    \begin{equation}
      \label{federer criterion}
      \H^n(\Om\cap\pa^eE)<\infty\qquad\Rightarrow\qquad\mbox{$E$ is of finite perimeter in $\Om$}\,,
    \end{equation}
    see \cite[4.5.11]{F}. If $E$ and $F$ are of locally finite perimeter in $\Om$ open, then so are $E\cup F$, $E\cap F$, and $E\setminus F$, and by \cite[Theorem 16.3]{maggiBOOK}, we have
    \begin{eqnarray}
      \label{E cup F}
      \Om\cap\pa^*(E\cup F)\ehn \Om\cap\Big\{\big(E^\zero\cap\pa^*F\big)\cup\big(F^\zero\cap\pa^*E\big)\cup\{\nu_E=\nu_F\}\Big\}\,,
      \\
      \label{E cap F}
      \Om\cap\pa^*(E\cap F)\ehn \Om\cap\Big\{\big(E^\one\cap\pa^*F\big)\cup\big(F^\one\cap\pa^*E\big)\cup\{\nu_E=\nu_F\}\Big\}\,,
      \\
      \label{E minus F}
      \Om\cap\pa^*(E\setminus F)\ehn \Om\cap\Big\{\big(E^\one\cap\pa^*F\big)\cup\big(F^\zero\cap\pa^*E\big)\cup\{\nu_E=-\nu_F\}\Big\}\,,
    \end{eqnarray}
    where $\{\nu_E=\pm\nu_F\}:=\{x\in\pa^*E\cap\pa^*F:\nu_E(x)=\pm\nu_F(x)\}$. By exploiting Federer's theorem \eqref{federer theorem}, \eqref{E cup F}, \eqref{E cap F}, and \eqref{E minus F} we can also deduce (the details are left to the reader)
    \begin{eqnarray}
      \label{X cap Y zero}
      (E\cap F)^\zero&\ehn&E^\zero\cup F^\zero\cup\{\nu_E=-\nu_F\}\,,
      \\
      \label{X minus Y zero}
      (E\setminus F)^\zero&\ehn&E^\zero\cup F^\one\cup\{\nu_E=\nu_F\}\,.
    \end{eqnarray}
    Finally, combining \eqref{E cup F}, \eqref{E minus F}, and \eqref{X minus Y zero}, we find
    \begin{equation}
      \label{E delta F}
      \pa^*(E\Delta F)\ehn(\pa^*E)\Delta(\pa^*F)\,.
    \end{equation}

    \medskip

    \noindent {\bf Partitions:} Given a Radon measure $\mu$ on $\R^{n+1}$ and Borel set $U\subset\R^{n+1}$ we say that  $\{U_i\}_i$ is a {\bf $\mu$-partition of $U$} if $\{U_i\}_i$ is an at most countable family of Borel subsets of $U$ such that
    \begin{equation}
      \label{mu partition}
      \mu\Big(U\setminus\bigcup_i U_i\Big)=0\,,\qquad \mu(U_i\cap U_j)=0\quad\forall i,j\,;
    \end{equation}
    and we say that $\{U_i\}_i$ is a {\bf monotone $\mu$-partition} if, in addition to \eqref{mu partition}, we also have $\mu(U_i)\ge\mu(U_{i+1})$ for every $i$. When $\mu=\L^{n+1}$ we replace ``$\mu$-partition'' with ``Lebesgue partition''. When $U$ is a set of finite perimeter in $\R^{n+1}$, we say that $\{U_i\}_i$ is a {\bf Caccioppoli partition} of $U$ if $\{U_i\}_i$ is a Lebesgue partition of $U$ and each $U_i$ is a set of finite perimeter in $\R^{n+1}$: in this case we have
    \begin{eqnarray}
    \label{caccioppoli partitions}
    \pa^*U\shn\bigcup_i\pa^*U_i\,,\qquad 2\, \H^n\Big(U^\one\cap\bigcup_i\pa^*U_i\Big)=\sum_i\H^n(U^\one\cap\pa^*U_i)\,,
    \end{eqnarray}
    see, e.g., \cite[Section 4.4]{AFP}; moreover,
    \begin{equation}
      \label{caccioppoli exactly two}
      1\le \#\Big\{i:x\in\pa^*U_i\Big\}\le 2\,,\qquad\forall x\in \bigcup_i\pa^*U_i\,,
    \end{equation}
    thanks to \eqref{degiorgi rect th} and to the fact that there cannot be three disjoint half-spaces in $\R^{n+1}$.

    \section{Induced essential partitions (Theorem \ref{theorem decomposition intro})}\label{section essential partitions} Given a Borel set $S$, we say that a Lebesgue partition $\{U_i\}_i$ of $U$ is {\bf induced by $S$} if, for each $i$,
    \begin{equation}
      \label{induced partition}
      \mbox{$U^\one\cap\pa^eU_i$ is $\H^n$-contained in $S$}\,.
    \end{equation}
    We say that $\{U_i\}_i$ is {\it an} {\bf essential partition of $U$ induced by $S$} if it is a Lebesgue partition of $U$ induced by $S$ such that, for each $i$,
    \begin{equation}
      \label{essential partition}
      \mbox{$S$ does not essentially disconnect $U_i$}\,.
    \end{equation}
    The next theorem, which expands the statement of Theorem \ref{theorem decomposition intro}, shows that when $\H^n$-finite sets uniquely determine induced essential partitions on sets of finite perimeter.

    \begin{theorem}[Induced essential partitions]\label{theorem decomposition}
        If $U\subset\R^{n+1}$ is a bounded set of finite perimeter and $S\subset\R^{n+1}$ is a Borel set with $\H^n(S \cap U^\one)<\infty$, then there exists an essential partition $\{U_i\}_i$ of $U$ induced by $S$ such that each $U_i$ is a set of finite perimeter and
        \begin{equation}
          \label{essential partition perimeter bound}
              \sum_iP(U_i;U^\one) \leq 2\,\H^n(S \cap U^\one)\,.
        \end{equation}
        Moreover: {\bf (a):} if $S^*$ is a Borel set with $\H^n(S^* \cap U^\one)<\infty$, $S^*$ is $\H^n$-contained in $S$, $\{V_j\}_j$ is a Lebesgue partition\footnote{Notice that here we are not requiring that $S^*$ does not essentially disconnect each $V_j$, i.e., we are not requiring that $\{V_j\}_j$ is an essential partition induced by $S^*$. This detail will be useful in the applications of this theorem.} of $U$ induced by $S^*$, and $\{U_i\}_i$ is the essential partition of $U$ induced by $S$, then
        \begin{equation}
          \label{essential partitions are monotone}
          \mbox{$\bigcup_j\,\pa^*V_j$ is $\H^n$-contained in $\bigcup_i\,\pa^*U_i$}\,;
        \end{equation}
        {\bf (b):} if $S$ and $S^*$ are $\H^n$-finite sets in $U^\one$, and either\footnote{Here $\RR(S)$ denotes the $\H^n$-rectifiable part of $S$.} $S^*=\RR(S)$ or $S^*$ is $\H^n$-equivalent to $S$, then $S$ and $S^*$ induce $\mathcal{L}^{n+1}$-equivalent essential partitions of $U$.
    \end{theorem}

    \begin{proof}
      [Proof of Theorem \ref{theorem decomposition intro}] Immediate consequence of Theorem \ref{theorem decomposition}.
    \end{proof}

    The proof of Theorem \ref{theorem decomposition} follows the main lines of the proof of \cite[Theorem 1]{ambrosiocaselles}, which is indeed the case $S=\varnothing$ of Theorem \ref{theorem decomposition}. We premise to this proof two lemmas that will find repeated applications in later sections too. To introduce the first lemma, we notice that, while it is evident that if $S$ is $\C$-spanning $\wire$ and $S$ is $\H^n$-contained into some Borel set $S^*$, then $S^*$ is also $\C$-spanning $\wire$, however, it is not immediately clear if the rectifiable part $\RR(S)$ of $S$ (which may not be $\H^n$-equivalent to $S$) retains the $\C$-spanning property.

    \begin{lemma}\label{lemma rectfiable spanning}
      If $\wire$ is compact, $\C$ is a spanning class for $\wire$, $S$ is $\C$-spanning $\wire$, and $\H^n\mres S$ is a Radon measure in $\Omega$, then $\mathcal{R}(S)$ is $\C$-spanning $\wire$. Moreover, the sets $T_1$ and $T_2$ appearing in \eqref{spanning borel intro} are sets of finite perimeter.
    \end{lemma}

    \begin{proof} We make the following {\it claim}: if $T$ is open, $T^{\one}\shn T$, $\H^n\mres Z$ is a Radon measure in an open neighborhood of $T$, and $Z$ essentially disconnects $T$ into $\{T_1,T_2\}$, then
    \begin{eqnarray}\label{rect disc 1}
      &&\mbox{$T_1$ and $T_2$ are of locally finite perimeter in $T$}\,,
      \\\label{rect disc 2}
      &&\mbox{$\RR(Z)$ essentially disconnects $T$ into $\{T_1,T_2\}$}\,.
    \end{eqnarray}
    Indeed: Since $T$ is open, we trivially have $T \subset T^{\one}$, and hence $T$ is $\H^n$-equivalent to $T^{\one}$. Taking also into account that $Z$ essentially disconnects $T$ into $\{T_1,T_2\}$, we thus find
    \begin{align}\notag
       T\cap\partial^e T_1 \cap \partial^e T_2 \ehn T^{\one}\cap \partial^e T_1 \cap \partial^e T_2  \shn Z \cap T^{\one} \shn Z \cap T\,.
    \end{align}
    By Federer's criterion \eqref{federer criterion} and the $\H^n$-finiteness of $Z$ in an open neighborhood of $T$ we deduce \eqref{rect disc 1}. By Federer's theorem \eqref{federer theorem}, $\partial^e T_i $ is $(\H^n\mres T)$-equivalent to $\partial^* T_i$, which combined with the $\H^n$-equivalence of $T^{\one}$ and $T$ gives
    \begin{align}\notag
        \partial^e T_1 \cap \partial^e T_2 \cap T^{\one} \ehn \partial^* T_1 \cap \partial^* T_2 \cap T\,. 
    \end{align}
    Since $\partial^* T_1 \cap \partial^* T_2 \cap T$ is $\H^n$-rectifiable and $\partial^e T_1 \cap \partial^e T_2 \cap T^{\one}\shn Z$, we conclude that $\H^n(\partial^e T_1 \cap \partial^e T_2 \cap T^{\one} \cap\mathcal{P}(Z))=0$. Hence,
    \[
    \partial^e T_1 \cap \partial^e T_2 \cap T^{\one}\shn \RR(Z)\,,
    \]
    and \eqref{rect disc 2} follows.

    \medskip

    To prove the lemma: Let $(\g,\Phi,T)\in\T(\C)$, let $J$ be of full measure such that \eqref{spanning borel} holds for every $s\in J$, so that, for every $s\in J$ one finds that for $\H^n$-a.e. $x\in T[s]$ there is a partition $\{T_1,T_2\}$ of $T$ with $x\in\partial^e T_1 \cap \partial^e T_2$ and such that $S \cup  T[s]$ essentially disconnects $T$ into $\{T_1,T_2\}$. By applying the claim with  $Z=S\cup T[s]$, we see that $\RR(S\cup T[s])$ essentially disconnects $T$ into $\{T_1,T_2\}$, and that $T_1$ and $T_2$ have locally finite perimeter in $T$. On noticing that $\RR(S\cup T[s])$ is $\H^n$-equivalent to $\RR(S)\cup T[s]$, we conclude the proof.
    \end{proof}

    The second lemma is just a simple compactness statement for finite perimeter partitions.

    \begin{lemma}[Compactness for partitions by sets of finite perimeter]\label{lemma partition compactness}
    If $U$ is a bounded open set and $\{\{U_i^j\}_{i=1}^\infty\}_{j=1}^\infty$ is a sequence of Lebesgue partitions of $U$ into sets of finite perimeter such that
    \begin{eqnarray}\label{uni per part bound}
    \sup_j\,\sum_{i=1}^\infty P(U_i^j) < \infty\,,
    \end{eqnarray}
    then, up to extracting a subsequence, there exists a Lebesgue partition $\{U_i\}_{i\in \mathbb{N}}$ of $U$ such that for every $i$ and every $A\subset U$ open,
    \begin{align}\label{partition convergence}
        \lim_{j\to\infty}|U_i^j\Delta U_i|=0\,,\qquad P(U_i;A) \leq \liminf_{j\to \infty} P(U_i^j;A)\,.
    \end{align}
    Moreover,
    \begin{align}\label{partition volume with power s}
        \lim_{i\to \infty}\limsup_{j\to \infty} \sum_{k=i+1}^\infty |U^j_k|^{s} =0\,,\qquad\forall s\in\Big(\frac{n}{n+1},1\Big)\,.
    \end{align}
    \end{lemma}

    \begin{proof} Up to a relabeling we can assume each $\{U_i^j\}_i$ is monotone. By \eqref{uni per part bound}  and the boundedness of $U$, a diagonal argument combined with standard lower semicontinuity and compactness properties of sets of finite perimeter implies that we can find a not relabeled subsequence in $j$ and a family $\{U_i\}_i$ of Borel subsets of $U$ with $|U_i|\ge |U_{i+1}|$ and $|U_i\cap U_j|=0$ for every $i\ne j$, such that \eqref{partition convergence} holds. We are thus left to prove \eqref{partition volume with power s} and
    \begin{align}\label{is a partition}
    \Big|U \setminus \bigcup_{i=1}^\infty U_i\Big|=0\,.
    \end{align}
    We start by noticing that for each $i$ there is $J(i)\in\N$ such that $|U_k^j|\le 2\,|U_k|$ for every $j\ge J(i)$ and $1\le k\le i$. Therefore if $k\ge i+1$ and $j\ge J(i)$ we find $|U_k^j|\le |U_i^j|\le 2\,|U_i|$, so that, if $j\ge J(i)$,
    \begin{align}\label{bad thing take two}
    \sum_{k=i+1}^\infty |U_k^{j}|^s \leq C(n)\,\sum_{k=i+1}^\infty P(U_{k}^{j})|U_k^{j}|^{s-(n/(n+1))}\le C\,|U_i|^{s-(n/(n+1))}\,,
    \end{align}
    where we have also used the isoperimetric inequality and \eqref{uni per part bound}. Since $|U_{i}|\to 0$ as $i\to\infty$ (indeed, $\sum_i|U_i|\le|U|<\infty$), \eqref{bad thing take two} implies \eqref{partition volume with power s}. To prove \eqref{is a partition}, we notice that if we set $M=|U \setminus \cup_i U_i|$, and we assume that $M$ is positive, then up to further increasing the value of $J(i)$ we can require that
    \begin{align}\label{preliminary bad thing}
         |U_k^j|\leq  |U_k| + \frac{M}{2^{k+2}}\,,\qquad\forall 1\leq k \leq i\,,\,\forall j\ge J(i)\,,
    \end{align}
    (in addition to $|U_k^j|\le 2\,|U_k|$). By \eqref{preliminary bad thing} we obtain that, if $j\ge J(i)$, then
    \begin{align}\label{bad thing}
    |U| - \sum_{k=i+1}^\infty |U_k^{j}|=\sum_{k=1}^i |U_k^{j}| \leq \sum_{k=1}^i  |U_k| + \frac{M}{2^{k+2}}\leq |U|-M+\sum_{k=1}^i  \frac{M}{2^{k+2}}\le |U|- \frac{M}{4}\,.
    \end{align}
    Rearranging \eqref{bad thing} and using the sub-additivity of $z\mapsto z^s$ we conclude that
    \begin{align}\notag
       (M/4)^s \leq \sum_{k=i+1}^\infty |U_k^{j}|^s\,.
    \end{align}
    We obtain a contradiction with $M>0$ by letting $i\to\infty$ and by using \eqref{partition volume with power s}.
    \end{proof}

    \begin{proof}[Proof of Theorem \ref{theorem decomposition}]
    Let $\mathcal{U}(S)$ be the set of all the monotone Lebesgue partitions of $U$ induced by $S$. We notice that $\mathcal{U}(S)\ne\varnothing$, since $\mathcal{U}(S)$ contains the trivial partition with $U_1=U$ and $U_i=\varnothing$ if $i\ge 2$. If $U_i\in\{U_i\}_i$ for $\{U_i\}_i\in \mathcal{U}(S)$, then $\partial^e U_i$ is $\H^n$-contained in $\partial^e U \cup (U^{(1)} \cap S)$, which, by Federer's criterion applied to $U$ and $\mathcal{H}^n(S \cap U^\one)<\infty$, has finite $\mathcal{H}^n$-measure; it follows then that $U_i$ is a set of finite perimeter due to Federer's criterion. We now fix $s\in(n/(n+1),1)$, and consider a maximizing sequence $\{\{U_i^j\}_i\}_j$ for
    \[
    m=\max\Big\{\sum_{i=1}^\infty |U_i|^{s}:\{U_i\}_i\in \mathcal{U}(S)\Big\}\,.
    \]
    By standard arguments concerning reduced boundaries of disjoint sets of finite perimeter (see, e.g. \cite[Chapter 16]{maggiBOOK}), we deduce from \eqref{induced partition} that for every $j$,
    \begin{eqnarray}\notag
    \sum_{i=1}^\infty \mathcal{H}^n\mres \partial^*
     U_i^j&=&\sum_{i=1}^\infty \mathcal{H}^n\mres (\partial^*U_i^j \cap U^{\one})+\sum_{i=1}^\infty \mathcal{H}^n\mres( \partial^*U_i^j\cap \partial^* U)\\ \label{actual perimeter bound 2}
     &\leq&2\,\H^n\mres(S\cap U^\one) + \mathcal{H}^n\mres \partial^* U\,.
    \end{eqnarray}
    Also, due to the sub-additivity of $z\mapsto z^s$ and the general fact that $\partial^e (A \cap B) \subset \partial^e A \cup \partial^e B$, we can refine $\{U^j_i\}_i$ by replacing each $U_i^j$ with the disjoint family
    \[
    \big\{U_i^j \cap U^\ell_k:k\ge1\,,1\le\ell<j\big\}\,,
    \]
    thus obtaining a new sequence in $\mathcal{U}(S)$ which is still maximizing for $m$. As a consequence of this remark, we can assume without loss of generality that the considered maximizing sequence $\{\{U_i^j\}_i\}_j$ for $m$ has the additional property that
    \begin{align}\label{nestedness}
        U\cap\bigcup_i \partial^* U_i^j \subset U\cap\bigcup_i \partial^* U_i^{j+1}\,,\qquad\forall j\,.
    \end{align}
    Thanks to \eqref{actual perimeter bound 2} we can apply Lemma \ref{lemma partition compactness} and, up to extracting a subsequence in $j$, we can find a Lebesgue partition $\{U_i\}_{i\in \mathbb{N}}$ of $U$ by sets of finite perimeter which satisfies \eqref{partition convergence} and \eqref{partition volume with power s}. Moreover, after taking a subsequence, we may assume that $\mathcal{H}^n\mres \partial^* U_i^j\weakstar\mu_i$ for some Radon measures $\mu_i$ such that $\mathcal{H}^n\mres \partial^* U_i \leq \mu_i$ \cite[Prop. 12.15]{maggiBOOK}. Therefore, by \eqref{partition convergence}, Federer's theorem for reduced boundaries, and by \eqref{induced partition} for $\{U^j_i\}_i$, we see that
    \begin{align}\notag
    \mathcal{H}^n&\mres ( \pa^* U)+\sum_{i=1}^\infty \mathcal{H}^n\mres (\partial^* U_i\cap U^\one )=\sum_{i=1}^\infty \mathcal{H}^n\mres (\partial^* U_i)\leq {\rm w}^*\lim_{j\to \infty} \sum_{i=1}^\infty \mathcal{H}^n\mres (\partial^* U_i^j)\\ \notag
    &={\rm w}^*\lim_{j\to \infty}\mathcal{H}^n\mres (\pa^* U) + \sum_{i=1}^\infty \H^n\mres (\partial^e U_i^j \cap U^\one ) \leq \mathcal{H}^n\mres (\pa^* U)+ 2\H^n\mres (S\cap U^\one) \,.
    \end{align}
    By subtracting $\mathcal{H}^n\mres (\pa^* U)$ from both sides, we deduce \eqref{essential partition perimeter bound}.

    \medskip

    We now show, first, that $\{U_i\}_i\in\mathcal{U}(S)$ (i.e., we check the validity of \eqref{induced partition} on $\{U_i\}_i$), and then that $S$ does not essentially disconnect any of the $U_i$. This will complete the proof of the first part of the statement.

    \medskip

    To prove that $U^\one \cap \partial^e U_i \shn S$, let us introduce the $\H^n$-rectifiable set $S_0$ defined by
    \begin{align}\label{s prime def}
        S_0 = U^\one\cap \bigcup_{i,j}\partial^* U_i^j\,.
    \end{align}
    By $\{U_i^j\}_i\in\mathcal{U}(S)$, $S_0$ is contained into $S$ modulo $\H^n$-null sets. Therefore, in order to prove \eqref{induced partition} it will be enough to show that
    \begin{align}\label{reduced containment}
       U^{\one}\cap \partial^* U_i\shn S_0\,,\qquad\forall i\,.
    \end{align}
    Should this not be the case, it would be $\H^n(U^\one \cap \partial^* U_i \setminus S_0)>0$ for some $i$. We could thus pick $x\in U^\one\cap \partial^* U_i$ such that
    \begin{align}\label{1 density at x}
        \theta^n(\H^n \mres (U^\one \cap \partial^* U_i \setminus S_0))(x) = 1\,.
    \end{align}
    Since $\theta^n(\H^n\mres\pa^*U_i)(x)=1$ and $S_0 \subset U^\one$ this implies $\H^n(S_0\cap B_r(x))={\rm o}(r^n)$, while $\pa^*U_i\subset U_i^\half$ gives $|U_i\cap B_r(x)|=(\om_{n+1}/2)\,r^{n+1}+{\rm o}(r^{n+1})$. Therefore, given $\de>0$ we can find $r>0$ such that
    \[
    \H^n( S_0 \cap B_r(x)) < \delta\, r^n\,,\qquad \min\big\{|U_i\cap B_r(x)|,|U_i\setminus B_r(x)|\big\}\ge \Big(\frac{\om_{n+1}}2-\de\Big)\,r^{n+1}\,,
    \]
    and then exploit the relative isoperimetric inequality and \eqref{partition convergence} to conclude that
    \begin{eqnarray*}
      c(n)\, \Big[\Big(\frac{\om_{n+1}}2-\de\Big)\,r^{n+1}\Big]^{n/(n+1)}&\le& P(U_i;B_r(x))\le\liminf_{j\to\infty}P(U_i^j;B_r(x))
      \\
      &\le&\H^n(S_0\cap B_r(x))\le \de\,r^n\,,
    \end{eqnarray*}
    where in the next to last inequality we have used the definition \eqref{s prime def} of $S_0$. Choosing $\de>0$ small enough we reach a contradiction, thus deducing that $\{U_i\}_i\in\mathcal{U}(S)$.

    \medskip

    Taking into account the subadditivity of $z\mapsto z^s$, in order to prove that $S$ does not essentially disconnect any $U_i$ it is sufficient to show that $\{U_i\}_i$ is a maximizer of $m$. To see this, we notice that $|U_i^j\Delta U_i|\to 0$ as $j\to\infty$ implies
    \[
    m=\lim_{j\to\infty}\sum_{i=1}^k|U_i^j|^s+\sum_{i=k+1}^\infty|U_i^j|^s
    =\sum_{i=1}^k|U_i|^s+\lim_{j\to\infty}\sum_{i=k+1}^\infty|U_i^j|^s\,,
    \]
    so that, letting $k\to\infty$ and exploiting \eqref{partition volume with power s}, we conclude that
    \begin{equation}
      \label{maximal energy}
      m=\sum_{i=1}^\infty|U_i|^s\,.
    \end{equation}
    This completes the proof of the first part of the statement (existence of essential partitions).

    \medskip

    Let now $S$, $S^*$, $\{U_i\}_i$, and $\{U_j^*\}_j$  be as in statement (a) -- that is, $S^*$ is $\H^n$-contained in $S$, $\{U_i\}_i$ is an essential partition of $U$ induced by $S$, and, for every $j$, $\{U_j^*\}_j$ is a Lebesgue partition of $U$ induced by $S^*$ -- and set $Z=\cup_i\pa^*U_i$ and $Z^*=\cup_j\pa^*U_j^*$. Arguing by contradiction with \eqref{essential partitions are monotone}, let us assume $\H^n(Z^*\setminus Z)>0$. By the definition of Lebesgue partition we have that $Z\setminus U^\one$ and $Z^*\setminus U^\one$ are both $\H^n$-equivalent to $\pa^*U$. Therefore we have $\H^n((Z^*\setminus Z)\cap U^\one)>0$. Since $U^\one$ is $\H^n$-equivalent to the union of the $\{U_i^{\one}\cup \pa^* U_i\}_{i\in I}$ we can find $i\in I$ and $j\in J$ such that $\H^n(U_i^{\one}\cap\pa^*U_j^*)>0$. This implies that both $(U_i\cap U_j^*)^\half$ and $(U_i\setminus U_j^*)^\half$ are non-empty, and thus that $\{U_j^*\cap U_i,U_i\setminus U_j^*\}$ is a non-trivial Borel partition of $U_i$. Since
    \[
    U_i^{\one}\cap\pa^e(U_j^*\cap U_i)\shn U^\one\cap \pa^*U_j^*\shn S^*\,,
    \]
    we conclude that $S^*$ is essentially disconnecting $U_i$, against the fact that $S$ is not essentially disconnecting $U_i$ and the fact that $S^*$ is $\H^n$-contained in $S$.

    \medskip

    We finally prove statement (b). Let $\{U_i\}_{i\in I}$, and $\{U_j^*\}_{j\in J}$ be essential partitions of $U$ induced by $S$ and $S^*$ respectively. Given $i\in I$ such that $|U_i|>0$, there is at least one $j\in J$ such that $|U_i\cap U_j^*|>0$. We {\it claim} that it must be $|U_i\setminus U_j^*|=0$. Should this not be the case, $\pa^*U_j^*$ would be essentially disconnecting $U_i$, thus implying that $S^*$ (which contains $\pa^*U_j^*$) is essentially disconnecting $U_i$. Now, either because we are assuming that $S^*$ is $\H^n$-equivalent to $S$, or because we are assuming that $S^*=\RR(S)$ and we have Lemma \ref{lemma rectfiable spanning}, the fact that $S^*$ is essentially disconnecting $U_i$ implies that $S$ is essentially disconnecting $U_i$, a contradiction. Having proved the claim, for each $i\in I$ with $|U_i|>0$ there is a unique $\s(i)\in J$ such that $|U_i\Delta U_{\s(j)}^*|=0$. This completes the proof.
    \end{proof}

    \section{Homotopic spanning on generalized soap films (Theorem \ref{theorem spanning with partition S case})}\label{section spanning and partitions} The goal of this section is proving Theorem \ref{theorem spanning with partition S case}, and, actually, to obtain an even more general result. Let us recall that the objective of Theorem \ref{theorem spanning with partition S case} was to reformulate the homotopic spanning property for a Borel set $S$, in the case when $S$ is locally $\H^n$-finite, in terms of unions of boundaries of induced essential partitions. We shall actually need this kind of characterization also for sets $S$ of the more general form $S=K\cup E^\one$, where $(K,E)\in\K_{\rm B}$. For an illustration of the proposed characterization of homotopic spanning on this type of sets, see
    \begin{figure}
    \input{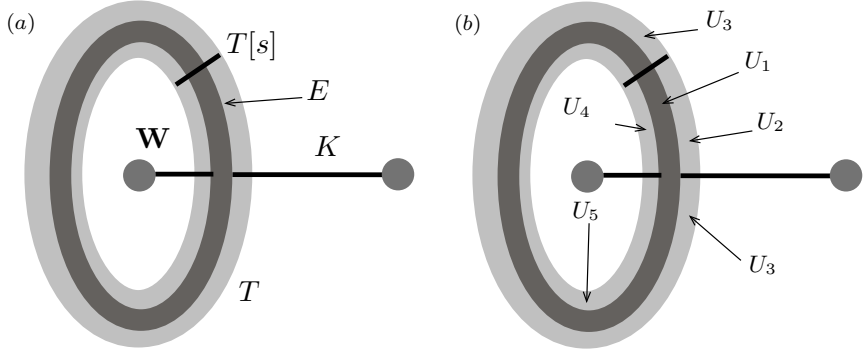}
    \caption{\small{In panel (a) we have depicted a pair $(K,E)$ where $E$ is a tube inside $T$ and $K$ consists of the union of the boundary of $E$ and the {\it non-}spanning set $S$ of Figure \ref{fig Scase}-(a). Notice that $K$ is not $\C$-spanning, if we see things from the point of view of Definition \ref{def homot span closed}, since it misses every loop $\g$ contained in the interior of $E$; while, of course, $K\cup E$ is $\C$-spanning because $E$ has been added. In panel (b) we have depicted the essential partition $\{U_i\}_{i=1}^5$ of $T$ induced by $K\cup T[s]$.  Notice that $E=U_1$, therefore no $\pa^*U_i\cap\pa^*U_j$ $\H^1$-containis $T[s]\cap E$. In particular, $T[s]\cap E$ (which $\H^1$-equivalent to $T[s]\setminus E^\zero$) is not $\H^1$-contained in ${\rm UBEP}(K\cup T[s];T)$, and we see again, this time from the point of view of Definition \ref{def homot span borel} as reformulated in Theorem \ref{theorem spanning with partition S case}, that $K$ is not $\C$-spanning. As stated in Theorem \ref{theorem spanning with partition}, from the viewpoint of Definition \ref{def homot span borel} it is only the $\H^1$-containment of $T[s]\cap E^\zero$ into ${\rm UBEP}(K\cup T[s];T)$ that establishes the $\C$-spanning property of $K\cup E$: and this $\H^1$-containment indeed holds, since $T[s]\cap E^\zero=T[s]\setminus\cl(E)$ is $\H^1$-contained in the union of $\pa^*U_2\cap\pa^*U_3$ and $\pa^*U_4\cap\pa^*U_5$.}}
    \label{fig KEcase}
    \end{figure}
    Figure \ref{fig KEcase}.

    \begin{theorem}[Homotopic spanning for generalized soap films]\label{theorem spanning with partition}
    If $\wire\subset\R^{n+1}$ is a closed set in $\R^{n+1}$, $\C$ is a spanning class for $\wire$, $K$ is a Borel set locally $\H^n$-finite in $\Om$, and $E$ is of locally finite perimeter in $\Om$ such that
    $\Om\cap\pa^*E$ is $\H^n$-contained in $K$, then the set
    \begin{eqnarray}
    \label{S equal to K union E}
    S=\RR(K)\cup E^{\one}
    \end{eqnarray}
    is $\C$-spanning $\wire$ if and only if, for every $(\gamma,\Phi, T)\in \mathcal{T}(\C)$ and $\H^1$-a.e. $s\in\SS^1$,
    \begin{eqnarray}\label{spanning and the S partition equation}
    &&\mbox{$T[s]\cap E^\zero$ is $\H^n$-contained in ${\rm UBEP}(K\cup T[s];T)$}\,.
    \end{eqnarray}
    \end{theorem}

    \begin{remark}\label{remark RRK cup Eone is spanning too}
      {\rm An immediate corollary of Theorem \ref{theorem spanning with partition} is that if $K$ is $\H^n$-finite and $(K,E)\in\KK_{\rm B}$ then $K\cup E^\one$ is $\C$-spanning $\wire$ if and only if $\RR(K)\cup E^\one$ is $\C$-spanning $\wire$. Indeed, $\RR(K\cup T[s])=\RR(K)\cup T[s]$, so that, by \eqref{def of UBEP}, ${\rm UBEP}(K\cup T[s])={\rm UBEP}(\RR(K)\cup T[s])$.}
    \end{remark}

    \begin{proof}
      [Proof of Theorem \ref{theorem spanning with partition S case}] This is Theorem \ref{theorem spanning with partition} with $E=\varnothing$.
    \end{proof}

    \begin{proof}[Proof of Theorem \ref{theorem spanning with partition}] {\it Step one}: We prove the following claim: If $S$ essentially disconnects $G$ into $\{G_1,G_2\}$ and $H\subset G$ satisfies
    \begin{align}\label{nont parti}
        \min\{|H \cap G_1|\,,\,|H \cap G_2|\} > 0\,,
    \end{align}
    then $S$ essentially disconnects $H$ into $H \cap G_1$ and $H \cap G_2$. Indeed, if $x\in H^{\one}$, then $x\in \partial^e(H \cap G_i)$ if and only if $x\in \partial^e G_i$ ($i=1,2$). Hence $H^{\one} \cap \partial^e (G_1 \cap H) \subset H^{\one} \cap \partial^e G_1 \subset G^{\one} \cap \partial^e G_1$, which, by \eqref{nont parti} and our assumption on $S$ and $G$, gives the desired conclusion.

    \medskip

    \noindent {\it Step two}: Taking from now on $S$, $K$ and $E$ as in the statement we preliminary notice that if $(\g,\Phi,T)\in\T(\C)$, $s\in\SS^1$, and $\{U_i\}_i$ is the essential partition of $T$ induced by $(\RR(K)\cup T[s])$, then
    \begin{align}\label{partial E containment}
        T\cap\partial^* E  \shn T\cap\bigcup_{i} \partial^* U_i\,.
    \end{align}
    Indeed, since $\Om\cap\pa^*E$ is $\H^n$-contained in $\RR(K)$, if a Borel set $G$ is such that $|G\cap E|\,|G\setminus E|>0$ then, by step one, $\RR(K)$ essentially disconnects $G$. In particular, since, for each $i$, $\RR(K)\cup T[s]$ does not essentially disconnect $U_i$, we find that, for each $i$,
    \begin{equation}\label{in or out of E prelude}
    \mbox{either $U_i^{\one}\subset E^\zero$}\qquad\mbox{or $U_i^{\one}\subset E^{\one}$}\,.
    \end{equation}
    Clearly, \eqref{in or out of E prelude} immediately implies  \eqref{partial E containment}.

    \medskip

    \noindent {\it Step three}: We prove the ``only if'' part of the statement, that is, given $(\g,\Phi,T)\in\T(\C)$ and $s\in \mathbb{S}^1$, we assume that
     \begin{eqnarray}\label{spanning borel x}
          &&\mbox{for $\H^n$-a.e. $x\in T[s]$}\,,
          \\\nonumber
          &&\mbox{$\exists$ a partition $\{T_1,T_2\}$ of $T$ with $x\in\partial^e T_1 \cap \partial^e T_2$}\,,
          \\ \nonumber
          &&\mbox{and s.t. $\RR(K)\cup E^{\one}\cup  T[s]$ essentially disconnects $T$ into $\{T_1,T_2\}$}\,,
    \end{eqnarray}
    and then prove that
    \begin{eqnarray}\label{spanning and the S partition equation x}
    &&\mbox{$T[s]\cap E^\zero$ is $\H^n$-contained in $\bigcup_i\pa^*U_i$}\,,
    \end{eqnarray}
    where $\{U_i\}_i$ is the essential partition of $T$ induced by $\RR(K)\cup T[s]$. To this end, arguing by contradiction, we suppose that for some $s\in \mathbb{S}^1$, there is $G \subset T[s] \cap E^\zero$ with $\H^n(G)>0$ and such that $G\cap\cup_i\pa^*U_i=\varnothing$. In particular, there is an index $i$ such that $\H^n(G \cap U_i^{\one})>0$, which, combined with \eqref{in or out of E prelude} and $G \subset E^\zero$, implies
    \begin{equation}
      \label{bibn}
      U_i^{\one}\subset E^\zero\,.
    \end{equation}
    Now by \eqref{spanning borel x} and $\H^n(G \cap U^{\one}_i)>0$, we can choose $x\in G \cap U_i^{\one}$ such that $\RR(K)\cup E^{\one} \cup T[s]$ essentially disconnects $T$ into some $\{T_1,T_2\}$ such that $x\in \partial^e T_1 \cap \partial^e T_2$. Then, $\{U_i \cap T_1, U_i \cap T_2\}$ is a non-trivial partition of $U_i$, so that, by step one and \eqref{bibn}, $\RR(K) \cup T[s]$ essentially disconnects $U_i$ into $\{U_i \cap T_1,U_i \cap T_2\}$. This contradicts the defining property \eqref{essential partition} of essential partitions, and concludes the proof.

    \medskip

    \noindent {\it Step four}: We prove the ``if'' part of the statement. More precisely, given $(\g,\Phi,T)\in\T(\C)$ and $s\in \mathbb{S}^1$, we assume that \eqref{spanning and the S partition equation x} holds at $s$, and then proceed to prove that \eqref{spanning borel x} holds at $s$. We first notice that, since $\{E^{\one},E^\zero,\pa^*E\}$ is a partition of $\Om$ modulo $\H^n$, it is enough to prove \eqref{spanning borel x} for $\H^n$-a.e. $x\in T[s]\cap(E^{\one}\cup E^\zero\cup\pa^*E)$.

    \medskip

    If $x\in T[s] \cap \partial^* E$, then by letting $T_1 = T \cap E$ and $T_2 = T \setminus E$ we obtain a partition of $T$ such that $x\in T\cap\pa^*E=T\cap\pa^*T_1\cap\pa^*T_2\subset\pa^eT_1\cap\pa^eT_2$, and such that $\pa^*E$ essentially disconnects $T$ into $\{T_1,T_2\}$. Since $\Om\cap\pa^*E$ is $\H^n$-contained in $\RR(K)$, we deduce \eqref{spanning borel x}.

    \medskip

    If $x\in T[s] \cap E^\zero$, then, thanks to \eqref{spanning and the S partition equation x} and denoting by $\{U_i\}_i$ the essential partition of $T$ induced by $(\RR(K)\cup T[s])$, there is an index $i$ such that $x\in T \cap \partial^* U_i$. Setting $T_1=U_i$ and $T_2=T\setminus U_i$, we have that $T \cap \partial^* U_i$ (which contains $x$) is in turn contained into $\partial^e T_1 \cap \partial^e T_2 \cap T$. Since the latter set is non-empty, $\{T_1,T_2\}$ is a non-trivial partition of $T$. Moreover, by definition of essential partition,
    \begin{align}\notag
        T^{\one}\cap \partial^e T_1 \cap \partial^e T_2 = T \cap \partial^e U_i \shn \RR(K) \cup T[s]\,,
    \end{align}
    so that $\RR(K) \cup T[s]$ essentially disconnects $T$, and \eqref{spanning borel x} holds.

    \medskip

    Finally, if $x\in T[s] \cap E^{\one}$, we let $s_1=s$, pick $s_2\neq s$, denote by $\{I_1,I_2\}$ the partition of $\SS^1$ defined by $\{s_1,s_2\}$, and set
    \[
    T_1=\Phi(I_1\times B_1^n)\cap E\,,\qquad T_2=\Phi(I_2\times B_1^n)\cup\,\Big(\Phi(I_1\times B_1^n)\setminus E\Big)\,.
    \]
    This is a Borel partition of $T$, and using the fact that $x\in E^{\one}$, we compute
    \[
    |T_1\cap B_r(x)|=|\Phi(I_1\times B_1^n)\cap E\cap B_r(x)|=|\Phi(I_1\times B_1^n)\cap B_r(x)|+{\rm o}(r^{n+1})=\frac{|B_r(x)|}2+{\rm o}(r^{n+1})\,.
    \]
    Therfore $x\in \partial^e T_1 \cap \partial^e T_2$, and by standard facts about reduced boundaries \cite[Chapter 16]{maggiBOOK},
    \[
    \partial^e T_1 \cap \partial^e T_2 \cap T^{\one} \shn \partial^* T_1  \cap T^{\one} \shn \big( \partial^* E \cup ((T[s_1]\cup T[s_2]) \cap E^{\one}) \big)\cap T^{\one}\,.
    \]
    Since $\Om\cap\pa^*E$ is $\H^n$-contained in $\RR(K)$, we have shown \eqref{spanning borel x}.
    \end{proof}

    \section{The fundamental closure theorem for homotopic spanning conditions}\label{section closure theorem basic} In Theorem \ref{theorem spanning with partition S case} and Theorem \ref{theorem spanning with partition} we have presented two reformulations of the homotopic spanning condition in terms of $\H^n$-containment into union of boundaries of essential partitions. The goal of this section is discussing the closure of such reformulations, and provide a statement (Theorem \ref{theorem basic closure for homotopic} below) which will lie at the heart of the closure theorems proved in Section \ref{section closure theorems sharp}.

    \begin{theorem}[Basic closure theorem for homotopic spanning]\label{theorem basic closure for homotopic}
      Let $\wire\subset\R^{n+1}$ be closed and let $\C$ be a spanning class for $\wire$. Let us assume that:

      \medskip

      \noindent {\bf (a):} $K_j$ are $\H^n$-finite Borel subsets of $\Om$ with $\H^n\mres K_j\weakstar\mu$ as Radon measures in $\Om$;

      \medskip

      \noindent {\bf (b):} $(\g,\Phi,T)\in\T(\C)$, $\{s_j\}_j$ is a sequence in $\SS^1$ with $s_j\to s_0$ as $j\to\infty$;

      \medskip

      \noindent {\bf (c):} if $\{U_i^j\}_i$ denotes the essential partition  of $T$ induced by $K_j\cup T[s_j]$, then there is a limit partition $\{U_i\}_i$ of $\{U_i^j\}_i$ in the sense of \eqref{partition convergence} in Lemma \ref{lemma partition compactness};

      \medskip

      Under these assumptions, if $\mu(T[s_0])=0$, $F_j,F\subset\Om$ are sets of finite perimeter with $F_j\to F$ as $j\to\infty$ and such that, for every $j$, $\Om\cap\pa^*F_j$ is $\H^n$-contained in $K_j$ and
      \begin{equation}\label{Tsj Fj in Kjstar}
        \mbox{$T[s_j]\cap F_j^\zero$ is $\H^n$-contained in $K^*_j$}\,,
      \end{equation}
      then
      \begin{equation}\label{Ts0 F in Kstar}
        \mbox{$T[s_0]\cap F^\zero$ is $\H^n$-contained in $K^*$}\,,
      \end{equation}
      where we have set
      \begin{equation}
        \label{def of Kjstar and Kstar}
          K_j^*={\rm UBEP}(K_j\cup T[s_j];T)=T\cap\bigcup_i\pa^*U_i^j\,,\qquad K^*=T\cap\bigcup_i\pa^*U_i\,.
      \end{equation}
    \end{theorem}

    \begin{remark}
      {\rm Notice that $\{U_i\}_i$ may fail to be the essential partition of $T$ induced by $K^*$ (which is the ``optimal'' choice of a Borel set potentially inducing $\{U_i\}_i$ on $T$): indeed, some of the sets $U_i$ may fail to be essentially connected, even though $U_i^j\to U_i$ as $j\to\infty$ and every $U_i^j$, as an element of an essential partition, is necessarily essentially connected; see
    \begin{figure}\input{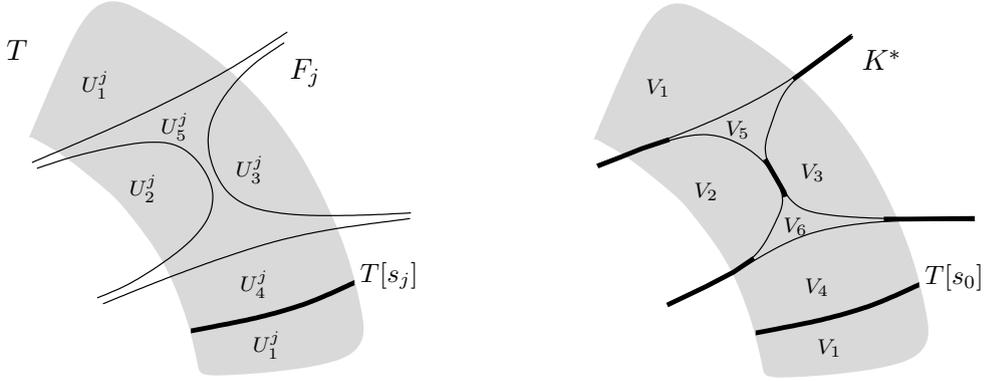}
    \caption{\small{The situation in the proof of Theorem \ref{theorem basic closure for homotopic} in the basic case when $K_j=\Om\cap\pa^*F_j$. The essential partition of $T$ induced by $K_j\cup T[s_j]$ is denoted by $\{U_i^j\}_i$. The limit partition $\{U_i\}_i$ of $\{U_i^j\}_i$ may fail to be the essential partition of $T$ induced by $K^*=T\cap\cup_i\pa^*U_i$, since some of the $U_i$ may be essentially disconnected. In the picture, denoting by $\{V_k\}_k$ the essential partition of $T$ induced by $K^*$, we have  $U_5=V_5\cup V_6=T\cap F$. We also notice, in reference to the notation set in \eqref{def of Xj}, that $X_1^j=\{5\}$ and $X_0^j=\{1,2,3,4\}$.}}\label{fig convergence}\end{figure}
    Figure \ref{fig convergence}.}
    \end{remark}

    \begin{proof}[Proof of Theorem \ref{theorem basic closure for homotopic}] {\it Step one}:  We start by showing that, for each $j$ and $i$ such that $|U_i^j|>0$, we have
    \begin{eqnarray}
    \label{doesnt cross partial Ej}
    \mbox{either}\quad (U_i^j)^{\one} \subset F_j^{\one}\,,\qquad\mbox{or}\quad(U_i^j)^{\one} \subset F_j^\zero\,,
    \end{eqnarray}
    and for each $i$ such that $|U_i|>0$,
    \begin{equation}
    \label{doesnt cross partial E}
    \mbox{either}\quad U_i^{\one} \subset F^{\one}\,,\qquad\mbox{or}\quad U_i^{\one} \subset F^\zero\,.
    \end{equation}
    Postponing for the moment the proof of \eqref{doesnt cross partial Ej} and \eqref{doesnt cross partial E}, let us record several consequences of these inclusions. First, if we set
    \begin{align}\label{def of Xj}
     &X^j_1=\big\{i:|U_i^j|>0\,,\, (U_i^j)^{\one} \subset F_j^{\one}\big\}\,,\qquad X^j_0=\big\{i:|U_i^j|>0\,,\,(U_i^j)^{\one} \subset F_j^\zero\big\}\,,
     \\
     \label{def of X}
     &X_1=\big\{i:|U_i|>0\,,\,U_i^{\one} \subset F^{\one}\big\}\,,\qquad\hspace{0.5cm} X_0=\big\{i:|U_i|>0\,,\,U_i^{\one} \subset F^\zero\big\}\,,
    \end{align}
    then, thanks to \eqref{doesnt cross partial Ej} and \eqref{doesnt cross partial E}, we have
    \begin{eqnarray}
      \label{Xj deco}
      X^j:=\{i:|U_i^j|>0\}=X_0^j\cup X_1^j\,, \qquad X:=\{i:|U_i|>0\}=X_0\cup X_1\,.
    \end{eqnarray}
    Combining \eqref{doesnt cross partial Ej} and \eqref{doesnt cross partial E} with $F_j\to F$ and $U_i^j\to U_i$, we find that for every $i\in X$, there is $J_i\in \mathbb{N}$ such that, for every $m\in\{0,1\}$,
    \begin{align}\label{index agreement}
      \mbox{if $i\in X_m$, then $i \in X_m^j$ for all $j\geq J_i$.}
    \end{align}
    Lastly, $\{U_i^j\}_{i\in X_1^j}$ is a Lebesgue partition of $T \cap F_j$, and thus, by Federer's theorem \eqref{federer theorem},
    \begin{eqnarray}\label{partition of Ej}
       T \cap  F_j^\one\,\, \shn \,\,\bigcup_{i\in X_1^j} (U_i^j)^\one\cup \pa^*U_i^j\,,
       \qquad  T \cap \partial^* F_j \,\,\shn \,\,T \cap \bigcup_{i\in X_1^j} \partial^* U_i^j \,\,\subset\,\, T \cap K_j^*\,.
    \end{eqnarray}
    {\it To prove \eqref{doesnt cross partial Ej} and \eqref{doesnt cross partial E}}: Since $\{U_i^j\}_i$ is the essential partition of $T$ induced by $K_j \cup T[s_j]$ and $K_j^*={\rm UBEP}(K_j\cup T[s_j];T)$, we have
    \begin{eqnarray}
      \label{Uij ess part union}
      &&\mbox{$K^*_j$ is $\H^n$-contained in  $K_j \cup T[s_j]$}\,,\qquad\forall j\,,
      \\
      \label{Uij ess part disco}
      &&
      \mbox{$K_j\cup T[s_j]$ does not essentially disconnect $U_i^j$}\,,\qquad\forall i,j\,.
    \end{eqnarray}
    Since $\Om\cap\pa^*F_j$ is $\H^n$-contained in $K_j\cup T[s_j]$, the combination of \eqref{Uij ess part disco} with Federer's theorem \eqref{federer theorem} gives \eqref{doesnt cross partial Ej}. The combination of $|U_i^j\Delta U_i|\to 0$ as $j\to\infty$ with \eqref{doesnt cross partial Ej} gives \eqref{doesnt cross partial E}.

    \medskip

    \noindent {\it Step two}: We reduce the proof of \eqref{Ts0 F in Kstar} to that of
    \begin{eqnarray}\label{luk}
      \H^n(U_i^\one\cap T[s_0])=0\,,\qquad\forall i\in X_0\,.
    \end{eqnarray}
    Indeed, $\{U_i^\one:i\in X_0\}\cup\{F^\zero\cap\pa^* U_i:i\in X_0\}$ is an $\H^n$-partition of $T\cap F^\zero$. In particular, $T\cap F^\zero$ is $\H^n$-contained in $\cup_{i\in X_0} U_i^{\one}\cup \partial^* U_i$, so that, should \eqref{luk} hold,  then $T[s_0]\cap F^\zero$ would be $\H^n$-contained in $\cup_{i\in X_0}\pa^*U_i$, and thus in $K^*$, thus proving \eqref{Ts0 F in Kstar}.

    \medskip

    \noindent {\it Step three}: We change variables from $T$ to\footnote{Here we identify $\SS^1$ with $\R/(2\pi\Z)$ and, with a slight abuse of notation, denote by $\L^{n+1}$ the ``Lebesgue measure on $\SS^1\times B_1^n$'', which we use to define sets of finite perimeter and points of density in $\SS^1\times B_1^n$.} $Y=\Phi^{-1}(T)=\SS^1\times B_1^n$. We set $Y[s]=\Phi^{-1}(T[s])=\{s\}\times B_1^n$ for the $s$-slice of $Y$, and
    \begin{equation}
      \label{def of Yi interface}
    Y_i=\Phi^{-1}(U_i)\,,\qquad Y_i^j=\Phi^{-1}(U_i^j)\,,\qquad W_i=Y\setminus Y_i\,,\qquad W_i^j=Y\setminus Y_i^j\,,
    \end{equation}
    Since $\Phi$ is a diffeomorphism, by \cite[Lemma A.1]{KingMaggiStuvard} and the area formula we have that
    \begin{equation}
      \label{diffeo boundaries}
      \pa^*\Phi^{-1}(H)=\Phi^{-1}(\pa^*H)\,,\qquad (\Phi^{-1}(H))^{\scriptscriptstyle{(m)}}=\Phi^{-1}(H^{\scriptscriptstyle{(m)}})\,, m\in\{0,1\}\,,
    \end{equation}
    for every set of finite perimeter $H\subset T$; in particular, setting
    \[
    M_j=\Phi^{-1}(F_j\cap T)\,,\qquad M=\Phi^{-1}(F\cap T)\,,
    \]
    by Federer's theorem \eqref{federer theorem}, we see that \eqref{Tsj Fj in Kjstar} is equivalent
    \begin{equation}
      \label{basic closure assumption Y}
      \mbox{$Y[s_j]$ is $\H^n$-contained in $\bigcup_i\pa^*Y_i^j \cup M_j^\one \cup \partial^* M_j $}\,,
    \end{equation}
    By \eqref{partition of Ej} and \eqref{diffeo boundaries}, we may rewrite \eqref{basic closure assumption Y} as
    \begin{equation}\label{alternative basic closure assumption Y}
        \mbox{$Y[s_j]$ is $\H^n$-contained in $\bigcup_{i\in \mathbb{N}}\pa^*Y_i^j \cup \bigcup_{i\in X_1^j} (Y_i^j)^\one $}\,.
    \end{equation}
    Similarly, $Y_i^\one=\Phi^{-1}(U_i^\one)$ for every $i$, and thus \eqref{luk} is equivalent to
    \begin{equation}
    \label{basic closure conclusion Y}
         \H^n(Y_i^{\one} \cap Y[s_0])=0\,,\qquad\forall i\in X_0\,.
    \end{equation}
    We are thus left to prove that \eqref{alternative basic closure assumption Y} implies \eqref{basic closure conclusion Y}.

    \medskip

    To this end, let us denote by $\pp$ the projection of $Y=\SS^1\times B_1^n$ onto $B_1^n$, and consider the sets
    \[
    G_i= \pp\big( Y_i^{\one} \cap Y[s_0]\big)\,,\qquad G_i^*=G^*\cap G_i\,,
    \]
    corresponding to the set $G^*\subset B_1^n$ with $\H^n(B_1^n\setminus G^*)=0$ defined as follows:

    \medskip

    \noindent (i) denoting by $H_y=\{s\in \mathbb{S}^1 : (s,y)\in H\}$ the ``circular slice of $H\subset Y$ above $y$'', if $y\in G^*$, $j\in\N$, $k$ is an index for the partitions $\{Y_k\}_k$ and $\{Y_k^j\}$, and $H\in\{Y_k,W_k,Y_k^j,W_k^j\}$, then $H_y$ is a set of finite perimeter in $\mathbb{S}^1$ with
    \begin{equation}
    \label{equivalence of reduced boundaries}
    H_y \overset{\H^1}{=} (H_y)^{{\one}_{\mathbb{S}^1}}\,, \qquad
    \partial^*_{\mathbb{S}^1}(H_y) \overset{\H^0}{=} (\partial^* H)_y\,,
    \end{equation}
    (and thus with $\partial^*_{\mathbb{S}^1}(H_y)=(\partial^* H)_y$); this is a standard consequence of the slicing theory for sets of finite perimeter, see, e.g., \cite[Theorem 2.4]{BarCagFus13} or \cite[Remark 18.13]{maggiBOOK};

    \medskip

    \noindent (ii) for every $y\in G^*$ and $j\in \N$,
    \begin{equation}
      \label{ii}
      (s_j,y)\in\bigcup_{k\in \mathbb{N}}\pa^*Y_k^j \cup \bigcup_{k\in X_1^j} (Y_k^j)^\one\,;
    \end{equation}
    this is immediate from \eqref{alternative basic closure assumption Y};

    \medskip

    \noindent (iii) for every $y\in G^*$, and $k$ an index for the partitions $\{Y_k\}_k$ and $\{Y_k^j\}$,
    \begin{equation}
      \label{immediate}
      \lim_{j\to\infty}\H^1((Y_k)_y\Delta (Y_k^j)_y)=0\,;
    \end{equation}
    this is immediate from Fubini's theorem and $Y_k^j\to Y_k$ as $j\to\infty$;

    \medskip

    \noindent (iv) for every $y\in G^*$,
    \begin{align}\label{property iv of Gstar}
    \sum_k \H^0((\partial^* Y_k^j)_y ) < \infty;
    \end{align}
    indeed, by applying in the order the coarea formula, the area formula and \eqref{essential partition perimeter bound} we find
    \begin{eqnarray*}
    \sum_k \int_{B_1^n}\H^0((\partial^* Y_k^j)_y )\,d\H^n &\leq& \sum_k P(Y_k^j;Y) \leq (\Lip\Phi^{-1})^n\,\sum_k P(U_k^j;T)
    \\
    &\leq&2\,(\Lip\Phi^{-1})^n\,\H^n(K_j \cup T[s_j])\,.
    \end{eqnarray*}

    \medskip

    Now, let us pick $y\in G_i^*$. Since $y\in G_i$ implies $(s_0,y)\in Y_i^\one$, and $Y_i^\one\cap\pa^*Y_i=\varnothing$, we find $(s_0,y)\not\in\pa^*Y_i$, i.e. $s_0\not\in(\pa^*Y_i)_y$. By $y\in G^*$, we have $(\pa^*Y_i)_y=\partial^*_{\mathbb{S}^1}(Y_i)_y$, so that
    \begin{equation}
      \label{heyhey}
      s_0\not \in \partial^*_{\mathbb{S}^1}(Y_i)_y\,.
    \end{equation}
    Since $(Y_i)_y$ has finite perimeter, $\partial^*_{\mathbb{S}^1}(Y_i)_y$ is a finite set, and so \eqref{heyhey} implies the existence of an open interval $\mathcal{A}_y\subset \mathbb{S}^1$, containing $s_0$, $\H^1$-contained either in $(Y_i)_y$ or in $(W_i)_y$, and such that
    \begin{equation}
      \label{heyhey2}
      \partial_{\mathbb{S}^1}\mathcal{A}_y \subset (\partial^* Y_i)_y=\pa^*_{\SS^1}(W_i)_y\,.
    \end{equation}
    We claim that there is $G_i^{**}\subset G_i^*$, with full $\H^n$-measure in $G_i^*$ (and thus in $G_i$), such that
    \begin{equation}
      \label{bellali}
      \mbox{$\mathcal{A}_y$ is $\H^1$-contained in $(Y_i)_y$}\,,\qquad\forall y\in G_i^{**}\,.
    \end{equation}
    Indeed, let us consider the countable decomposition $\{G_{i,m}^*\}_{m=1}^\infty$ of $G_i^*$ given by
    \[
    G_{i,m}^*=\Big\{y\in G_i^*:\dist\big(\{s_0\},\pa_{\SS^1}\mathcal{A}_y\big)\in\big[1\big/(m+1),1\big/m\big)\Big\}\subset B_1^n\,,
    \]
    and let
    \[
    Z_{i,m}=\big\{y\in G_{i,m}^*: \mbox{$\mathcal{A}_y$ is $\H^1$-contained in $(W_i)_y$}\big\}\,.
    \]
    If $\H^n(Z_{i,m})>0$, then there is $y^*\in Z_{i,m}^\one$, so that $\H^n(Z_{i,m}\cap B_r^n(y^*))=\om_n\,r^n+{\rm o}(r^n)$. Therefore, if $r<1/(m+1)$ and $B_r^1(s_0)$ denotes the open interval of center $s_0$ and radius $r$ inside $\SS^1$, then
    \begin{eqnarray*}
    &&\L^{n+1}\big(Y_i\cap\big(B_r^1(s_0)\times B_r^n(y^*)\big)\big)=\int_{B_r^n(y^*)}\H^1(B_r^1(s_0)\cap (Y_i)_y)\,d\H^n_y
    \\
    &=&\int_{Z_{i,m}\cap B_r^n(y^*)}\H^1(B_r^1(s_0)\cap (Y_i)_y)\,d\H^n_y+{\rm o}(r^{n+1}) ={\rm o}(r^{n+1})
    \end{eqnarray*}
    where in the last identity we have used the facts that $y\in Z_{i,m}\cap B_r^n(y^*)$, $s_0\in\mathcal{A}_y$, and $r<1/(m+1)$ to conclude that $B_r^1(s_0)$ is $\H^1$-contained in $(W_i)_y$; in particular, $(s_0,y^*)\in Y_i^\zero$, against the fact that $Z_{i,m}\subset G_i(=\pp(Y[s_0]\cap Y_i^\one))$. We have thus proved that each $Z_{i,m}$ is $\H^n$-negligible, and therefore that there is $G_i^{**}\subset G_i^*$ and $\H^n$-equivalent to $G_i^*$, such that \eqref{bellali} holds true.

    \medskip

    Having proved \eqref{bellali}, we now notice that, by \eqref{ii}, $y\in G_i^*$ implies
    \begin{equation}
      \label{combined with sj}
      s_j\in\bigcup_{k\in \mathbb{N}}(\pa^*Y_k^j)_y \cup \bigcup_{k\in X_1^j} \big((Y_k^j)^\one\big)_y=\bigcup_{k}\pa^*_{\SS^1}(Y_k^j)_y \cup \bigcup_{k\in X_1^j} \big((Y_k^j)_y \big)^{\one_{\mathbb{S}^1}}\,.
    \end{equation}
    If \eqref{combined with sj} holds because $s_j\in\pa^*_{\SS^1}(Y_k^j)_y$ for some $k$, then, thanks to \eqref{property iv of Gstar} there must $k'\ne k$ such that $s_j\in\pa^*_{\SS^1}(Y_{k'}^j)_y$ too; since either $k$ or $k'$ must be different from $i$, we conclude that $s_i\in\pa^*_{\SS^1}(Y_{k(i)}^j)_y$ for some $k(i)\ne i$; if, instead, \eqref{combined with sj} holds because $s_j\in\big((Y_k^j)_y \big)^{\one_{\mathbb{S}^1}}$ for some $k\in X_1^j$, then we can recall that, thanks to \eqref{index agreement}, $i\in X_0^j$ for every $j\ge J_i$, and thus $i\ne k$; in summary, for each $y\in G_i^*$,
    \begin{equation}
    \label{spanning yay}
    \mbox{if $j\ge J_i$, then $\exists k(j)\ne i$ s.t.}\,\,s_j\in \partial^*_{\mathbb{S}^1} (Y_{k(j)}^j)_y \cup \big((Y_{k(j)}^j)_y\big)^{\one_{\mathbb{S}^1}}\,.
    \end{equation}
    With the goal of obtaining a lower bound on the relative perimeters of the sets $Y_i^j$ in a neighborhood of $G_i$ (see \eqref{bad multiplicity take 2} below), we now consider $y\in G_i^{**}$, and pick $r>0$ such that $\cl B_r^1(s_0) \subset \mathcal{A}_y$. Correspondingly, since $s_j\to s_0$ and \eqref{spanning yay} holds, we can find $J^*=J^*(i,y,r)\ge J_i$ such that, for $j\ge J^*$,
    \begin{equation}
      \label{bellali2}
     s_j\in B_r^1(s_0) \cap   \big[\partial^*_{\mathbb{S}^1} (Y_{k(j)}^j)_y\cup \big((Y_{k(j)}^j)_y\big)^{\one_{\mathbb{S}^1}}\big]\subset\mathcal{A}_y\cap   \big[\partial^*_{\mathbb{S}^1} (Y_{k(j)}^j)_y\cup \big((Y_{k(j)}^j)_y\big)^{\one_{\mathbb{S}^1}}\big]\,.
    \end{equation}
    Now, by \eqref{immediate}, $k(j)\neq i$, and $\mathcal{A}_y \overset{\H^1}{\subset} (Y_i)_y$, we have
    \begin{equation}
      \label{bellali3}
      \lim_{j\to\infty}\H^1(\mathcal{A}_y \cap (Y_{k(j)}^j)_y ) =0\,.
    \end{equation}
    Since, by \eqref{equivalence of reduced boundaries}, $(Y_{k(j)}^j)_y$ is $\H^1$-equivalent to a finite union of intervals, \eqref{bellali2} implies the existence of an open interval $\mathcal{I}_y^j$ such that
    \begin{eqnarray}\label{last}
    s_j\in \cl_{\SS^1}\mathcal{I}_y^j\,,\qquad \mathcal{I}_y^j\overset{\H^1}{\subset}  (Y_{k(j)}^j)_y\,,
    \qquad
    \pa_{\SS^1}\mathcal{I}_y^j\subset(\pa^*Y_{k(j)}^j)_y\subset(\pa^*W_i^j)_y\,,
    \end{eqnarray}
    which, due to \eqref{bellali2} and \eqref{bellali3}, must satisfy
    \[
    \lim_{j\to\infty}\diam \big(\mathcal{I}_y^j \big)=0\,.
    \]
    In particular,
    \[
    \pa_{\SS^1}\,\mathcal{I}_y^j\subset B_r^1(s_0)\,,\qquad\forall j\ge J^*\,,
    \]
    and thus, by the last inclusion in \eqref{last},
    \[
    \H^0\big(B_r^1(s_0)\cap\partial^*_{\SS^1}(W_i^j)_y\big)
    \ge
    \H^0(B_r^1(s_0)\cap \pa_{\SS^1}\mathcal{I}_j^y)\ge 2\,,
    \]
    whenever $j\ge J^*$. Since $y\in G_i^{**}$ and $r>0$ were arbitrary, by the coarea formula and Fatou's lemma,
    \begin{eqnarray}
      \nonumber
      \liminf_{j\to\infty}P(W_i^j;B_r^1(s_0)\times G_i^{**})&\ge&\liminf_{j\to\infty}\int_{G_i^{**}}\H^0\big(B_r^1(s_0)\cap\partial^*_{\SS^1}(W_i^j)_y\big)\,d\H^n_y
      \\
      \label{bad multiplicity take 2}
      &\ge&2\,\H^n(G_i^{**})=2\,\H^n(G_i)\,.
    \end{eqnarray}
    Now, since $\pa^*W_i^j=\pa^* Y_i^j=\Phi^{-1}(\pa^*U_i^j)$, by \eqref{Uij ess part union} we have
    \[
    \mbox{$Y\cap\bigcup_i\pa^*W_i^j$ is $\H^n$-contained in $Y[s_j]\cup\Phi^{-1}\big(T\cap K_j\big)$}\,,
    \]
    which implies, for every $j$ large enough to have $s_j\in B_r^1(s_0)$,
    \begin{eqnarray}
      \nonumber
      &&P(W_i^j;B_r^1(s_0)\times G_i^{**})
      \\\nonumber
    &&\le
    \H^n\big(Y[s_j]\cap( B_r^1(s_0)\times G_1^{**})\big)+\H^n\big( \Phi^{-1}(T\cap K_j)\cap (B_r^1(s_0)\times B_1^n)\big)
    \\\nonumber
    &&=
    \H^n(G_i^{**})+\H^n\big( \Phi^{-1}(T\cap K_j)\cap (B_r^1(s_0)\times B_1^n)\big)
    \\\label{gelato}
    &&\le
    \H^n(G_i)+\Lip(\Phi^{-1})^n\,\H^n\big(K_j\cap \Phi(B_r^1(s_0)\times B_1^n)\big)\,.
    \end{eqnarray}
    By combining \eqref{bad multiplicity take 2} with \eqref{gelato} we conclude that for every $r>0$
    \begin{equation}
      \label{in}
      \H^n(G_i)\le\Lip(\Phi^{-1})^n\,\mu\big(\Phi(\cl(B_r^1(s_0))\times B_1^n)\big)\,,
    \end{equation}
    By $\mu(T[s_0])=0$, if we let $r\to 0^+$ in \eqref{in}, we conclude that $\H^n(G_i)=0$. Now, since $G_i= \pp\big( Y_i^{\one} \cap Y[s_0]\big)$, we have
    \begin{equation}
      \label{soft}
      \H^n\big( Y_i^{\one} \cap Y[s_0]\big)=\H^n(G_i)\,,
    \end{equation}
    thus proving \eqref{basic closure conclusion Y}, and hence the theorem.
    \end{proof}

    \section{Direct Method on generalized soap films (Theorem \ref{theorem first closure theorem intro})}\label{section closure theorems sharp} In Section \ref{section first closure theorem} we prove Theorem \ref{theorem first closure theorem intro}, while in Section \ref{section second closure theorem} we notice the changes to that argument that are needed to prove a different closure theorem that will be crucial in the companion papers \cite{MNR2,MNR3}. In particular, Section \ref{section second closure theorem} will not be needed for the other main results of this paper (although it is included here since it is definitely easier to understand in this context).

    \subsection{Proof of Theorem \ref{theorem first closure theorem intro}}\label{section first closure theorem} Let us first of all recall the setting of the theorem. We are given a closed set $\wire$ in $\mathbb{R}^{n+1}$, a spanning class $\C$ for $\wire$, and a sequence $\{(K_j,E_j)\}_j$ in $\K_{\rm B}$ such that
    \begin{equation}
      \label{energy bound capillarity 2}
      \sup_j\,\H^n(K_j)<\infty\,,
    \end{equation}
    and, for some Borel set $E$ and Radon measures $\mu_{\rm bk}$ and $\mu_{\rm bd}$ in $\Om$, it holds that $E_j\toloc E$ and
    \begin{eqnarray}
    \label{def of mub}
    &&\H^n\mres (\Om\cap\pa^*E_j) + 2\,\H^n \mres (\mathcal{R}(K_j) \cap E_j^\zero) \weakstar \mu_{\rm bk}\,,
    \\
    \label{def of mu}
    &&\H^n\mres (\Om\cap\pa^*E_j) + 2\,\H^n \mres (\mathcal{R}(K_j) \setminus \pa^* E_j) \weakstar \mu_{\rm bd}\,,
    \end{eqnarray}
    as $j\to\infty$. In this setting we want to prove that the sets
    \begin{eqnarray}
      \label{def of Kbk}
      K_{\rm bk}\!\! &:=&\!\!\big(\Om\cap\partial^* E\big) \cup \Big\{x\in \Omega \cap E^\zero : \theta^n_*(\mu_{\rm bk})(x)\geq 2 \Big\}\,,
      \\
      \label{def of Kbd}
      K_{\rm bd}\!\!&:=&\!\!\big(\Om\cap\partial^* E\big) \cup \Big\{x\in \Omega \setminus\pa^*E : \theta^n_*(\mu_{\rm bd})(x)\geq 2 \Big\}\,,
    \end{eqnarray}
    are such that $(K_{\rm bk},E),(K_{\rm bd},E)\in\K_{\rm B}$ and
    \begin{eqnarray}
      \label{mub limit lb}
      \mu_{\rm bk}\!\!&\ge&\!\! \H^n\mres (\Om\cap\pa^*E) + 2\,\H^n \mres (K_{\rm bk} \cap E^\zero)\,,
      \\
      \label{mu limit lb}
      \mu_{\rm bd}\!\!&\ge&\!\!\H^n\mres (\Om\cap\pa^*E) + 2\,\H^n \mres (K_{\rm bd} \setminus\pa^*E)\,,
    \end{eqnarray}
    with
    \begin{equation}
      \label{lsc Fb}
      \liminf_{j\to\infty}\F_{\rm bk}(K_j,E_j)\ge\F_{\rm bk}(K_{\rm bk},E)\,,\qquad
        \liminf_{j\to\infty}\F_{\rm bd}(K_j,E_j)\ge\F_{\rm bd}(K_{\rm bd},E)\,;
    \end{equation}
    and that the closure statements
    \begin{eqnarray}\label{hp bulk spanning}
      &&\mbox{if $K_j\cup E_j^\one$ is $\C$-spanning $\wire$ for every $j$,}
      \\
      \label{limit bulk spanning}
      &&\mbox{then $K_{\rm bk}\cup E^\one$ is $\C$-spanning $\wire$}\,,
    \end{eqnarray}
    and
    \begin{eqnarray}
      \label{hp interface spanning}
      &&\mbox{if $K_j$ is $\C$-spanning $\wire$ for every $j$,}
      \\
      \label{limit interface spanning}
      &&\mbox{then $K_{\rm bd}$ is $\C$-spanning $\wire$}\,,
    \end{eqnarray}
    hold true.

    \begin{proof}[Proof of Theorem \ref{theorem first closure theorem intro}] By $\Om\cap\pa^*E\subset K_{\rm bk}\cap K_{\rm bd}$ we have $(K_{\rm bk},E),(K_{\rm bd},E)\in\K_{\rm B}$. By \cite[Theorem 6.4]{maggiBOOK}, $\tnl(\mu_{\rm bk}) \geq 2$ on $K_{\rm bk}\cap E^\zero$ implies $\mu_{\rm bk} \mres  (K_{\rm bk} \cap E^\zero) \geq 2\H^n \mres (K_{\rm bk} \cap E^\zero)$, and, similarly, we have $\mu_{\rm bd}\mres (K_{\rm bd}\setminus\pa^*E)\ge 2\,\H^n\mres(K_{\rm bd}\setminus\pa^*E)$. Since,  by the lower semicontinuity of distributional perimeter, we have $\min\{\mu_{\rm bk},\mu_{\rm bd}\}\geq \H^n \mres (\partial^* E\cap \Omega)$,  \eqref{mub limit lb}, \eqref{mu limit lb} and \eqref{lsc Fb} follow. We are thus left to prove that if either \eqref{hp bulk spanning} or \eqref{hp interface spanning} holds, then \eqref{limit bulk spanning} or \eqref{limit interface spanning} holds respectively. We divide the proof into three parts, numbered by Roman numerals.

    \medskip

    \noindent {\bf I. Set up of the proof:} Fixing from now on a choice of $(\gamma,\Phi,T)\in \mathcal{T}(\C)$ against which we want to test the $\C$-spanning properties \eqref{limit bulk spanning} and \eqref{limit interface spanning}, we introducing several key objects related to $(\gamma,\Phi,T)$.

    \medskip

    \noindent {\it Introducing $s_0$}: Up to extracting subsequences, let $\mu$ be the weak-star limit of $\H^n\mres\,K_j$, and set
    \begin{align}\label{def of J with mustar}
    J=\{s\in \SS^1:\mu(T[s])=0\}\,,
    \end{align}
    so that $\H^1(\SS^1\setminus J)=0$. We fix $s_0\in J$.

    \medskip

    \noindent {\it Introducing $s_j$, $\{U_i^j\}_i$, and $K_j^*$}: For $\H^1$-a.e. $s\in\SS^1$ it holds that $\H^n(K_j\cap T[s])=0$ for every $j$ and (thanks to Theorem \ref{theorem spanning with partition S case}/Theorem \ref{theorem spanning with partition}) the essential partition $\{U_i^j[s]\}_i$ induced on $T$ by $K_j\cup T[s]$ is such that
    \begin{eqnarray*}
      \mbox{$T[s]\cap E_j^\zero$ is $\H^n$-contained in ${\rm UBEP}(K_j\cup T[s];T)$}\,,&&\qquad\mbox{(if \eqref{hp bulk spanning} holds)}\,,
      \\
      \mbox{$T[s]$ is $\H^n$-contained in ${\rm UBEP}(K_j\cup T[s];T)$}\,,&&\qquad\mbox{(if \eqref{hp interface spanning} holds)}\,.
    \end{eqnarray*}
    Therefore we can find a sequence $s_j\to s_0$ as $j\to\infty$ such that
    \begin{equation}
      \label{def of sj}
      \H^n(K_j\cap T[s_j])=0\qquad\forall j\,,
    \end{equation}
    and, denoting by $\{U_i^j\}_i$ the essential partition of $T$ induced by $K_j\cup T[s_j]$ (i.e. $U_i^j=U_i^j[s_j]$), and setting for brevity
    \begin{equation}
      \label{def of Kj star}
    K_j^*={\rm UBEP}(K_j\cup T[s_j];T)=T\cap\bigcup_i\pa^*U_i^j\,,
    \end{equation}
    we have
    \begin{eqnarray}
    \label{whatabout Kj 2}
    \mbox{$T[s_j]\cap E_j^\zero$ is $\H^n$-contained in $K_j^*$}\,,&&\qquad\mbox{(if \eqref{hp bulk spanning} holds)}\,,
    \\
    \label{whatabout Kj 2 interface}
    \mbox{$T[s_j]$ is $\H^n$-contained in $K_j^*$}\,,&&\qquad\mbox{(if \eqref{hp interface spanning} holds)}\,.
    \end{eqnarray}

    \medskip

    \noindent {\it Introducing $\{U_i\}_i$ and $K^*$:} By \eqref{energy bound capillarity 2}, Lemma \ref{lemma partition compactness}, and up to extract a subsequence we can find a Lebesgue partition $\{U_i\}_i$ of $T$ such that,
    \begin{equation}
      \label{def of Ui}
      \mbox{$\{U_i\}_i$ is the limit of $\{\{U_i^j\}_i\}_j$ in the sense specified by \eqref{partition convergence}}\,.
    \end{equation}
    Correspondingly we set
    \begin{equation}
      \label{def of Kstar}
      K^*=T\cap\bigcup_i\pa^*U_i\,.
    \end{equation}

    \medskip

    Having introduced $s_0$, $s_j$, $\{U_i^j\}_i$, $K_j^*$, $\{U_i\}_i$, and $K^*$,  we notice that if \eqref{hp bulk spanning} holds, then we can apply Theorem \ref{theorem basic closure for homotopic} with $F_j=E_j$ and find that
    \begin{eqnarray}
    \label{whatabout K^*}
    \mbox{$T[s_0]\cap E^\zero$ is $\H^n$-contained in $K^*$}\,,&&\mbox{(if  \eqref{hp bulk spanning} holds)}\,;
    \end{eqnarray}
    if, instead, \eqref{hp interface spanning} holds, then Theorem \ref{theorem basic closure for homotopic} can be applied with $F_j=F=\varnothing$ to deduce
    \begin{eqnarray}
    \label{whatabout K^* int}
    \mbox{$T[s_0]$ is $\H^n$-contained in $K^*$}\,,&&\mbox{(if \eqref{hp interface spanning}  holds)}\,.
    \end{eqnarray}
    We now make the following claim:

    \medskip

    \noindent {\bf Claim:} We have
    \begin{eqnarray}\label{Kstar inclusion}
        &&\mbox{$K^*\setminus (T[s_0]\cup E^\one)$ is $\H^n$-contained in $K_{\rm bk}$}\,,
        \\
        \label{Kstar inclusion int}
        &&\mbox{$K^*\setminus T[s_0]$ is $\H^n$-contained in $K_{\rm bd}$}\,.
    \end{eqnarray}
    The rest of the proof of the theorem is then divided in two parts: the conclusion follows from the claim, and the proof of the claim.

    \medskip

    \noindent {\bf II. Conclusion of the proof from the claim:} {\it Proof that \eqref{hp interface spanning} implies \eqref{limit interface spanning}}: By $\H^1(\SS^1\setminus J)=0$, the arbitrariness of $s_0\in J$, and that of $(\g,\Phi,T)\in\T(\C)$, thanks to Theorem \ref{theorem spanning with partition S case} we can conclude that $K_{\rm bd}$ is $\C$-spanning $\wire$ by showing that
    \begin{equation}
      \label{chiaro1}
      \mbox{$T[s_0]$ is $\H^n$-contained in ${\rm UBEP}(K_{\rm bd}\cup T[s_0];T)$.}
    \end{equation}
    Now, since $\{U_i\}_i$ is a Lebesgue partition of $T$ induced by $K^*$ (in the very tautological sense that $K^*$ is defined as $T\cap\cup_i\pa^*U_i$!) and, by \eqref{Kstar inclusion int} in claim one, $K^*$ is $\H^n$-contained in $K_{\rm bd}\cup T[s_0]$, by Theorem \ref{theorem decomposition}-(a) we have that if $\{Z_i\}_i$ is the essential partition of $T$ induced by $K_{\rm bd}\cup T[s_0]$, then $\cup_i\pa^*U_i$ is $\H^n$-contained in $\cup_i\pa^*Z_i$: therefore, by definition of $K^*$ and by definition of ${\rm UBEP}$, we have that
    \begin{equation}
      \label{f1}
      \mbox{$K^*$ is $\H^n$-contained in ${\rm UBEP}\big(K_{\rm bd}\cup T[s_0];T\big)$}\,.
    \end{equation}
    By combining \eqref{f1} with \eqref{whatabout K^* int} we immediately deduce \eqref{chiaro1} and conclude.

    \medskip

    \noindent {\it Proof that \eqref{hp bulk spanning} implies \eqref{limit bulk spanning}}: Thanks to Theorem \ref{theorem spanning with partition} it suffices to prove that
    \begin{equation}
      \label{chiaro2}
      \mbox{$T[s_0]\cap E^\zero$ is $\H^n$-contained in ${\rm UBEP}(K_{\rm bk}\cup T[s_0];T)$}\,.
    \end{equation}
    By \eqref{whatabout K^*}, the proof of \eqref{chiaro2} can be reduced to that of
    \begin{equation}
      \label{chiaro3}
      \mbox{$K^*\cap E^\zero$ is $\H^n$-contained in ${\rm UBEP}(K_{\rm bk}\cup T[s_0];T)$}\,.
    \end{equation}
    Now, let us consider the Lebesgue partition of $T$ defined by $\{V_k\}_k=\{U_i\setminus E\}_i\cup\{T\cap E\}$. By \cite[Theorem 16.3]{maggiBOOK} we easily see that for each $i$
    \begin{equation}
      \label{chiaro5}
      E^\zero\cap\pa^*U_i\,\shn\, \pa^*(U_i\setminus E)\,\shn \,\big(E^\zero\cap\pa^*U_i\big)\cup\pa^*E\,,
    \end{equation}
    which combined with $T\cap\pa^*(T\cap E)=T\cap\pa^*E\subset K_{\rm bk}$ and with \eqref{Kstar inclusion} in claim one, gives
    \begin{eqnarray}
    \nonumber
    T\cap\bigcup_k\pa^*V_k&=&(T\cap\pa^*E)\cup \Big\{T\cap\bigcup_i\pa^*(U_i\setminus E)\Big\}
    \,\shn\,(T\cap\pa^*E)\cup\,\big(E^\zero\cap K^*\big)
    \\
    \label{chiaro4}
    &\shn&(T\cap\pa^*E)\cup\,\big(K^*\setminus E^\one)\,\shn\,K_{\rm bk}\cup T[s_0]\,.
    \end{eqnarray}
    By \eqref{chiaro4} we can exploit Theorem \ref{theorem decomposition}-(a) to conclude that
    \begin{equation}
      \label{chiaro6}
      \mbox{$T\cap\bigcup_k\pa^*V_k$ is $\H^n$-contained in ${\rm UBEP}(K_{\rm bk}\cup T[s_0];T)$}\,.
    \end{equation}
    By the first inclusion in \eqref{chiaro5}, $E^\zero\cap K^*$ is $\H^n$-contained in $T\cap\bigcup_k\pa^*V_k$, therefore \eqref{chiaro6} implies \eqref{chiaro3}, as required. We are thus left to prove the two claims.

    \medskip

    \noindent {\bf III. Proof of the claim:} We finally prove that $K^*\setminus (T[s_0]\cup E^\one)$ is $\H^n$-contained in $K_{\rm bk}$ (that is \eqref{Kstar inclusion}), and that $K^*\setminus T[s_0]$ is $\H^n$-contained in $K_{\rm bd}$ (that is \eqref{Kstar inclusion int}).

    \medskip

    To this end, repeating the argument in the proof of Theorem \ref{theorem basic closure for homotopic} with $F_j=E_j$ and $F=E$ we see that, if we set $X^j_m=\{i:(U_i^j)^\one\subset E_j^{\scriptscriptstyle{(m)}}\}$ and $X_m=\{i:U_i^\one\subset E^{\scriptscriptstyle{(m)}}\}$ for $m\in\{0,1\}$ (see \eqref{def of Xj} and \eqref{def of X}), then
    \begin{eqnarray}
      \label{Xj deco next}
      X^j:=\{i:|U_i^j|>0\}=X_0^j\cup X_1^j\,, \qquad X:=\{i:|U_i|>0\}=X_0\cup X_1\,;
    \end{eqnarray}
    and, moreover, for every $i$ there is $j(i)$ such that $i\in X_m$ implies $i\in X_m^j$ for every $j\ge j(i)$. Thanks to \eqref{Xj deco next} we easily see that $K_j^*=T\cap\cup_i\pa^*U_i^j$ can be decomposed as
    \begin{equation}
      \label{manca 0}
      K_j^*\ehn\bigcup_{(i,k)\in X_0^j\times X_0^j\,, i\ne j}M_{ik}^j
      \cup\bigcup_{(i,k)\in X_1^j\times X_1^j\,, i\ne j}M_{ik}^j
      \cup\bigcup_{(i,k)\in X_0^j\times X_1^j}M_{ik}^j\,,
    \end{equation}
    where $M_{ik}^j=T\cap\pa^*U_i^j\cap\pa^*U_k^j$ (an analogous decomposition of $K^*$ holds as well, and will be used in the following, but is not explicitly written for the sake of brevity). We now prove that
    \begin{eqnarray}
      \label{Mikj zero}
      &&M_{ik}^j \subset E_j^\zero\,,\qquad\forall i,k\in X^j_0\,, i\ne k\,,
      \\
      \label{Mikj bordo}
      &&M_{ik}^j\subset\pa^e E_j\,,\qquad\forall i\in X^j_0\,,k\in X^j_1\,,
      \\
      \label{Mikj one}
      &&M_{ik}^j \subset E_j^{\one}\,,\qquad\forall i,k\in X^j_1\,, i\ne k\,.
    \end{eqnarray}

    \medskip

    \noindent {\it To prove \eqref{Mikj zero} and \eqref{Mikj one}}: if $i\ne k$, $i,k\in X^j_0$, and $x\in M_{ik}^j$, then (by $|U_i^j\cap U_k^j|=0$) $U_i^j$ and $U_k^j$ blow-up two complementary half-spaces at $x$, an information that combined with the $\L^{n+1}$-inclusion of $U_i^j\cup U_k^j$ in $\R^{n+1}\setminus E_j$ implies
    \begin{eqnarray*}
      |B_r(x)|+{\rm o}(r^{n+1})=|B_r(x)\cap U_i^j|+|B_r(x)\cap U_k^j|\le|B_r(x)\setminus E_j|\,,
    \end{eqnarray*}
    that is, $x\in E_j^\zero$, thus proving \eqref{Mikj zero}; the proof of \eqref{Mikj one} is analogous.

    \medskip

    \noindent {\it To prove \eqref{Mikj bordo}}: if $i\in X^j_0$, $k\in X^j_1$, and $x\in M_{ik}^j$, then
    \begin{eqnarray*}
    &&|B_r(x)\cap E_j|\ge|B_r(x)\cap U_k^j|=\frac{|B_r(x)|}2+{\rm o}(r^{n+1})\,,
    \\
    &&|B_r(x)\setminus E_j|\ge|B_r(x)\cap U_i^j|=\frac{|B_r(x)|}2+{\rm o}(r^{n+1})\,,
    \end{eqnarray*}
    so that $x\not\in E_j^\zero$ and $x\not\in E_j^{\one}$, i.e. $x\in\pa^eE_j$, that is \eqref{Mikj bordo}.

    \medskip

    With \eqref{Mikj zero}--\eqref{Mikj one} at hand, we now prove that
    \begin{eqnarray}
    \label{Mikj bordo decompo}
      T\cap\pa^*E_j\ehn \bigcup_{(i,k)\in X_0^j\times X_1^j}M_{ik}^j\,,
      \\
    \label{Mikj zero decompo}
      K_j^*\cap E_j^\zero\ehn \bigcup_{(i,k)\in X_0^j\times X_0^j\,,k\ne i}M_{ik}^j\,.
    \end{eqnarray}
    (Analogous relations hold with $K^*$ and $E$ in place of $K^*_j$ and $E_j$.)

    \medskip

    \noindent {\it To prove \eqref{Mikj bordo decompo}}: By $\pa^*E_j\subset\pa^eE_j$ and \eqref{doesnt cross partial Ej} we find $\pa^*E_j\cap (U_i^j)^{\one}=\varnothing$ for every $i,j$; hence, since $\{(U_i^j)^{\one}\}_i\cup\{\pa^*U_i^j\}_i$ is an $\H^n$-partition of $T$, and by repeatedly applying \eqref{Mikj zero}, \eqref{Mikj bordo} and \eqref{Mikj one}, we find
    \begin{eqnarray*}
    \bigcup_{(i,k)\in X_0^j\times X_1^j}M_{ik}^j&\shn& T\cap\pa^*E_j\ehn\bigcup_i (T\cap\pa^*E_j\cap\pa^*U_i^j)\ehn\bigcup_{i,k} M_{ik}^j\cap\pa^*E_j
    \\
    &\ehn&\bigcup_{(i,k)\in X_0^j\times X_1^j} M_{ik}^j\cap\pa^*E_j\,,
    \end{eqnarray*}
    which gives \eqref{Mikj bordo decompo}.

    \medskip

    \noindent {\it To prove \eqref{Mikj zero decompo}}: By \eqref{Mikj zero}, \eqref{Mikj bordo}, and \eqref{Mikj one}, $M_{ik}^j$ has empty intersection with $E_j^\zero$ unless $i,k\in X_0^j$, in which case $M_{ik}^j$ is $\H^n$-contained in $E_j^\zero$: hence,
    \[
    \bigcup_{(i,k)\in X_0^j\times X_0^j\,,k\ne i}M_{ik}^j\shn K_j^*\cap E_j^\zero=\bigcup_{(i,k)\in X_0^j\times X_0^j\,, k\ne i}E_j^\zero\cap M_{ik}^j\,,
    \]
    that is  \eqref{Mikj zero decompo}.

    \medskip

    With \eqref{Mikj bordo decompo} and \eqref{Mikj zero decompo} at hand, we now prove the following perimeter formulas: for every open set $A\subset T$ and every $j$,
    \begin{eqnarray}
      \label{perimeter Uij on Xj0}
      \sum_{i\in X^j_0}P(U_i^j;A)=\H^n\big(A\cap\pa^* E_j\big)+2\,\H^n\big( A\cap K^*_j\cap E_j^\zero\big)\,,
      \\
      \label{perimeter Uij on Xj1}
      \sum_{i\in X^j_1}P(U_i^j;A)=\H^n\big(A\cap\pa^* E_j\big)+2\,\H^n\big( A\cap K^*_j\cap E_j^\one\big)\,.
    \end{eqnarray}
    Analogously, for $\a=0,1$,
    \begin{equation}\label{perimeter Ui on Xalfa}
      \sum_{i\in X_\a}P(U_i;A)=\H^n\big(A\cap\pa^* E\big)+2\,\H^n\big( A\cap K^*\cap E^{\scriptscriptstyle{(\a)}}\big)\,.
    \end{equation}

    \medskip

    \noindent {\it To prove \eqref{perimeter Uij on Xj0} and \eqref{perimeter Uij on Xj1}}: Indeed, by \eqref{Mikj bordo decompo} and \eqref{Mikj zero decompo},
    \begin{eqnarray*}
      \sum_{i\in X_0^j}P(U_i^j;A)
      &=&\sum_{(i,k)\in X_0^j\times X_1^j}\H^n(A\cap  M_{ik}^j)
      +\sum_{i  \in X_0^j}\sum_{k\in X_0^j\setminus\{i\}}\H^n(A\cap M_{ik}^j)
      \\
      &=&\H^n\Big(\bigcup_{(i,k)\in X_0^j\times X_1^j} A\cap M_{ik}^j\Big)+2\,\H^n\Big(\bigcup_{(i,k)\in X_0^j\times X_0^j\,,i\ne k} A\cap M_{ik}^j\Big)
      \\
      &=&\H^n(A\cap\pa^*E)+2\,\H^n\big(A\cap K_j^*\cap E_j^\zero\big)\,,
    \end{eqnarray*}
    that is \eqref{perimeter Uij on Xj0}. The proof of \eqref{perimeter Uij on Xj1} is analogous (since \eqref{perimeter Uij on Xj1} is \eqref{perimeter Uij on Xj0} applied to the complements of the $E_j$'s -- recall indeed that $\Om\cap\pa^*E_j=\Om\cap\pa^*(\Om\setminus E_j)$).

    \medskip

    \noindent {\it Conclusion of the proof of \eqref{Kstar inclusion} in the claim}: We want to prove that $K^*\setminus (T[s_0]\cup E^\one)$ is $\H^n$-contained in $K_{\rm bk}$. Since $\{E^\zero,E^\one,\pa^*E\}$ is an $\H^n$-partition of $\Om$, and $\Om\cap\pa^*E$ is contained in $K_{\rm bk}$, looking back at the definition \eqref{def of Kbk} of $K_{\rm bk}$ it is enough to show that
    \begin{equation}
    \label{Kstar inclusion proof}
        \mbox{$\theta^n_*(\mu_{\rm bk})(x)\ge 2$ for $\H^n$-a.e. $x\in (K^*\cap E^\zero)\setminus T[s_0]$}\,.
    \end{equation}
    To this end, we begin noticing that, if $Y_0$ is an arbitrary finite subset of $X_0$, then there is $j(Y_0)$ such that $Y_0\subset X_0^j$ for every $j\ge j(Y_0)$; correspondingly,
    \[
    \sum_{i\in Y_0}P(U_i;A)\le\liminf_{j\to\infty}\sum_{i\in Y_0}P(U_i^j;A)\le
    \liminf_{j\to\infty}\sum_{i\in X_0^j}P(U_i^j;A)\,.
    \]
    By arbitrariness of $Y_0$, \eqref{perimeter Ui on Xalfa} with $\a=0$, \eqref{perimeter Uij on Xj0}, and \eqref{Uij ess part union} (notice that the $\H^n$-containment of the $\H^n$-rectifiable set $K^*_j$ into $K_j\cup T[s_0]$ is equivalent to its $\H^n$-containment in $\RR(K_j\cup T[s_j])=\RR(K_j)\cup T[s_j]$) we conclude that, if $A\subset T$ is open and such that $\cl(A)\cap T[s_0]=\varnothing$, so that $A\cap T[s_j]=\varnothing$ for $j$ large enough, then
    \begin{eqnarray}\nonumber
    &&\H^n\big(A\cap\pa^* E\big)+2\,\H^n\big( A\cap K^*\cap E^\zero\big)
    \\\nonumber
    &&=\sum_{i\in X_0}P(U_i;A)\le\liminf_{j\to\infty}\sum_{i\in X_0^j}P(U_i^j;A)
    \\\nonumber
    &&=\liminf_{j\to\infty}\H^n\big(A\cap\pa^* E_j\big)+2\,\H^n\big( A\cap K^*_j\cap E_j^\zero\big)
    \\
    \nonumber
    &&\le\liminf_{j\to\infty}\H^n\big(A\cap\pa^* E_j\big)+2\,\H^n\big( A\cap \big(\RR(K_j)\cup T[s_j]\big)\cap E_j^\zero\big)
    \\
    \label{nm}
    &&=\liminf_{j\to\infty}\H^n\big(A\cap\pa^* E_j\big)+2\,\H^n\big( A\cap \RR(K_j)\cap E_j^\zero\big)\le\mu_{\rm bk}(\cl(A))\,,
    \end{eqnarray}
    where we have used the definition \eqref{def of mub} of $\mu_{\rm bk}$. Now, if $x\in (K^*\cap E^\zero)\setminus T[s_0]$, then we we can apply \eqref{nm} with $A=B_s(x)$ and $s>0$ such that $\cl(B_s(x))\cap T[s_0]=\varnothing$, together with the fact that $x\in E^\zero$ implies $\H^n(B_s(x)\cap\pa^*E)={\rm o}(s^n)$ as $s\to 0^+$, to conclude that
    \begin{align}\label{nm applied}
      \mu_{\rm bk}(\cl(B_s(x)))&\ge 2\,\H^n\big(B_s(x)\cap K^*\cap E^\zero\big)+{\rm o}(s^n)\,,\qquad\mbox{as $s\to0^+$}\,.
    \end{align}
    Since $K^*\cap E^\zero$ is an $\H^n$-rectifiable set, and thus $\H^n\big(B_s(x)\cap K^*\cap E^\zero\big)=\om_n\,s^n+{\rm o}(s^n)$ for $\H^n$-a.e. $x\in K^*\cap E^\zero$, we deduce \eqref{Kstar inclusion proof} from \eqref{nm applied}.

    \medskip

    \noindent {\it Conclusion of the proof of \eqref{Kstar inclusion int} in the claim}: We want to prove the $\H^n$-containment of $K^*\setminus T[s_0]$ in $K_{\rm bd}$. As in the proof of \eqref{Kstar inclusion}, combining Federer's theorem \eqref{federer theorem} with the definition \eqref{def of Kbd} of $K_{\rm bd}$, we are left to prove that
    \begin{equation}
    \label{Kstar inclusion int proof}
        \mbox{$\theta_*^n(\mu_{\rm bd})(x)\ge 2$ for $\H^n$-a.e. $x\in K^*\setminus(T[s_0]\cup\pa^*E)$}\,.
    \end{equation}
    As proved in \eqref{nm}, if $A\subset T$ is open and such that $\cl(A)\cap T[s_0]=\varnothing$, then by exploiting \eqref{perimeter Uij on Xj0} and \eqref{perimeter Ui on Xalfa} with $\a=0$ we have
    \begin{eqnarray}\label{nm again}
    &&\H^n\big(A\cap\pa^* E\big)+2\,\H^n\big( A\cap K^*\cap E^\zero\big)
    \\\nonumber
    &&\le\liminf_{j\to\infty}\H^n\big(A\cap\pa^* E_j\big)+2\,\H^n\big( A\cap \RR(K_j)\cap E_j^\zero\big)\,;
    \end{eqnarray}
    the same argument, this time based on \eqref{perimeter Uij on Xj1} and \eqref{perimeter Ui on Xalfa} with $\a=1$, also gives
    \begin{eqnarray}\label{nm again one}
    &&\H^n\big(A\cap\pa^* E\big)+2\,\H^n\big( A\cap K^*\cap E^\one\big)
    \\\nonumber
    &&\le\liminf_{j\to\infty}\H^n\big(A\cap\pa^* E_j\big)+2\,\H^n\big( A\cap \RR(K_j)\cap E_j^\one\big)\,;
    \end{eqnarray}
    and, finally, since $\Om\setminus\pa^*E$ is $\H^n$-equivalent to $\Om\cap(E^\zero\cup E^\one)$, the combination of \eqref{nm again} and \eqref{nm again one} gives
    \begin{eqnarray}\label{nm again bordo}
    &&\H^n\big(A\cap\pa^* E\big)+2\,\H^n\big( A\cap K^*\setminus\pa^* E\big)
    \\\nonumber
    &&\le\liminf_{j\to\infty}\H^n\big(A\cap\pa^* E_j\big)+2\,\H^n\big( A\cap \RR(K_j)\setminus\pa^* E_j\big)\le\mu_{\rm bd}(\cl(A))\,,
    \end{eqnarray}
    where we have used the definition \eqref{def of mu} of $\mu_{\rm bd}$. Now, for $\H^n$-a.e. $x\in K^*\setminus(T[s_0]\cup\pa^*E)$ we have $\H^n(B_r(x)\cap\pa^*E)={\rm o}(r^n)$ and $\H^n(B_r(x)\cap K^*\setminus \pa^*E)=\om_n\,r^n+{\rm o}(r^n)$ as $r\to 0^+$, as well as  $\cl(B_r(x))\cap T[s_0]=\varnothing$ for $r$ small enough, so that \eqref{nm again bordo} with $A=B_r(x)$ readily implies \eqref{Kstar inclusion int proof}. The proof of the claim, and thus of the theorem, is now complete.
    \end{proof}

    \subsection{A second closure theorem}\label{section second closure theorem} We now present a variant of the main arguments presented in this section and alternative closure theorem to Theorem \ref{theorem first closure theorem intro}. As already noticed, this second closure theorem, Theorem \ref{theorem second closure theorem} below, will play a role only in the companion paper \cite{MNR2}, where Plateau's laws will be studied in the relation to the Allen--Cahn equation, so that this section can be omitted on a first reading focused on Gauss' capillarity theory alone.

    \medskip

    To introduce Theorem \ref{theorem first closure theorem intro}, let us consider the following question: given an $\H^n$-finite set $S$ which is $\C$-spanning $\wire$, {\it what parts of $S$ are essential to its $\C$-spanning property}? We already know from Lemma \ref{lemma rectfiable spanning} that the unrectifiable part of $S$ is not necessary, since $\RR(S)$ is also $\C$-spanning. However, some parts of $\RR(S)$ could be discarded too -- indeed rectifiable sets can be ``porous at every scale'', and thus completely useless from the point of view of achieving $\C$-spanning. To make an example, consider the rectifiable set $P\subset\R^2$ obtained by removing from $[0,1]$ all the intervals $(q_i-\e_i,q_i+\e_i)$ where $\{q_i\}_i$ are the rational numbers in $[0,1]$ and $2\,\sum_i\e_i=\e$ for some given $\e\in(0,1)$: it is easily seen that $P$ is a rectifiable set with positive $\H^1$-measure in $\R^2$, contained in $\R\times\{0\}$, which fails to essentially disconnect any stripe of the form $(a,b)\times\R$ with $(a,b)\cc(0,1)$. Intuitively, if a set like $P$ stands as an isolated portion of $S$, then $\RR(S)\setminus P$ should still be $\C$-spanning.

    \medskip

    We can formalize this idea as follows. Denoting as usual $\Om=\R^{n+1}\setminus\wire$, we consider the open covering $\{\Om_k\}_k$ of $\Om$ defined by
    \begin{equation}
      \label{def of Omega i}
      \{\Om_k\}_k=\{B_{r_{mh}}(x_m)\}_{m,h}\,,
    \end{equation}
    where $\{x_m\}_m=\Q^{n+1}\cap\Om$ and $\{r_{mh}\}_h=\Q\cap(0,\dist(x_m,\pa\Om))$. For every $\H^n$-finite set $S$ we define the {\bf essential spanning part of $S$ in $\Om$} as the Borel set
    \[
    {\rm ESP}(S)=\bigcup_k\,{\rm UBEP}(S;\Om_k)=\bigcup_k\,\Big\{\Om_k\cap\bigcup_i\pa^*U_i[\Om_k]\Big\}\,,
    \]
    if $\{U_i[\Om_k]\}_i$ denotes the essential partition of $\Om_k$ induced by $S$. Since each ${\rm UBEP}(S;\Om_k)$ is a countable union of reduced boundaries and is $\H^n$-contained in the $\H^n$-finite set $S$, we see that ${\rm ESP}(S)$ is always $\H^n$-rectifiable. The idea is that by following the unions of boundaries of essential partitions induced by $S$ over smaller and smaller balls we are capturing all the parts of $S$ that may potentially contribute to achieve a spanning condition with respect to $\wire$. Thinking about Figure \ref{fig esspart}: the tendrils of $S$ appearing in panel (a) and not captured by ${\rm UBEP}(S;U)$, will eventually be included into ${\rm ESP}(S)$ by considering ${\rm UBEP}$'s of $S$ relative to suitable subsets of $U$. Another way to visualize the construction of ${\rm ESP}(S)$ is noticing that if $B_r(x)\subset B_s(x)\subset\Om$, then
    \[
    B_r(x)\cap {\rm UBEP}(S;B_s(x))\subset {\rm UBEP}(S;B_r(x))\,,
    \]
    which points to the monotonicity property behind the construction of ${\rm ESP}(S)$. Intuitively, we expect that
    \begin{equation}
      \label{intro S 1}
      \mbox{if $S$ is $\C$-spanning $\wire$, then ${\rm ESP}(S)$ is $\C$-spanning $\wire$}
    \end{equation}
    (where $\C$ is an arbitrary spanning class for $\wire$). This fact will proved in a moment as a particular case of Theorem \ref{theorem second closure theorem} below.

    \medskip

    Next, we introduce the notion of convergence behind our second closure theorem. Consider a sequence $\{S_j\}_j$ of Borel subsets of $\Om$ such that $\sup_j\H^n(S_j)<\infty$. If we denote by $\{U_i^j[\Om_k]\}_i$ the essential partition induced on $\Om_k$ by $S_j$, then a diagonal argument based on Lemma \ref{lemma partition compactness} shows the existence of a (not relabeled) subsequence in $j$, and, for each $k$, of a Borel partition $\{U_i[\Om_k]\}_i$ of $\Om_k$ such that  $\{U_i^j[\Om_k]\}_i$ converges to $\{U_i[\Om_k]\}_i$ as $j\to\infty$ in the sense specified by \eqref{partition convergence}. Since ${\rm UBEP}(S_j;\Om_k)=\Om_k\cap\bigcup_i\pa^*U_i^j[\Om_k]$, we call any set $S$ of the form\footnote{The limit partition $\{U_i[\Om_k]\}_i$ appearing in \eqref{spl def} may not be the essential partition induced by $S$ on $\Om_k$ since the individual $U_i[\Om_k]$, arising as $L^1$-limits, may fail to be essentially connected. This said, $\{U_i[\Om_k]\}_i$ is automatically a partition of $\Om_k$ induced by $S_0$.}
    \begin{equation}
      \label{spl def}
      S=\bigcup_k\,\Big\{\Om_k\cap\bigcup_i\pa^*U_i[\Om_k]\Big\}\,,
    \end{equation}
    a {\bf subsequential partition limit of $\{S_j\}_j$ in $\Om$}. Having in mind \eqref{intro S 1}, it is natural to ask if the following property holds:
    \begin{eqnarray}\nonumber
      &&\mbox{if $S_j$ is $\C$-spanning $\wire$ for each $j$}\,,
      \\
      \nonumber
      &&\mbox{and $S$ is a subsequential partition limit of $\{S_j\}_j$ in $\Om$}\,,
      \\\label{intro S 2}
      &&\mbox{then $S$ is $\C$-spanning $\wire$}\,.
    \end{eqnarray}
    Our next theorem implies both \eqref{intro S 1} and \eqref{intro S 2} as particular cases (corresponding to be taking $E_j=\varnothing$ and, respectively, $K_j=S$ and $K_j=S_j$ for every $j$).

    \begin{theorem}
      [Closure theorem for subsequential partition limits]\label{theorem second closure theorem} Let $\wire$ be a closed set in $\mathbb{R}^{n+1}$, $\C$ a spanning class for $\wire$, and $\{(K_j,E_j)\}_j$ a sequence in $\K_{\rm B}$ such that $\sup_j\H^n(K_j)<\infty$ and $K_j\cup E_j^\one$ is $\C$-spanning $\wire$ for every $j$.

      \medskip

      If $S_0$ and $E_0$ are, respectively, a subsequential partition limit of $\{K_j\}_j$ in $\Om$ and an $L^1$-subsequential limit of $\{E_j\}_j$ (corresponding to a same not relabeled subsequence in $j$), and we set
      \[
      K_0=(\Om\cap\pa^*E_0)\cup S_0\,,
      \]
      then $(K_0,E_0)\in\KK_{\rm B}$ and $K_0\cup E_0^\one$ is $\C$-spanning $\wire$.
    \end{theorem}

    \begin{proof} Since $\Om\cap\pa^*E_0\subset K_0$ by definition of $K_0$ we trivially have $(K_0,E_0)\in\KK_{\rm B}$. Aiming to prove that $K_0\cup E_0^\one$ is $\C$-spanning $\wire$, we fix $(\g,\Phi,T)\in\T(\C)$, and define $s_0$, $s_j$, $\{U_i^j\}_i$ and $\{U_i\}_i$ exactly as in part I of the proof of Theorem \ref{theorem first closure theorem intro}. Thanks to Theorem \ref{theorem basic closure for homotopic} and by arguing as in part II of the proof of Theorem \ref{theorem first closure theorem intro}, we have reduced to prove that
    \begin{equation}
      \label{upon}
      \mbox{$K^*\setminus (T[s_0]\cup E^\one)$ is $\H^n$-contained in $K_0$}\,.
    \end{equation}
    By Federer's theorem \eqref{federer theorem} and since $\Om\cap\pa^*E\subset K_0$ it is enough to prove
    \[
      \mbox{$(K^*\cap E^\zero)\setminus T[s_0]$ is $\H^n$-contained in $S_0$}\,,
    \]
    and, thanks to the construction of $S_0$, we shall actually be able to prove
    \begin{equation}
      \label{upon3}
      \mbox{$K^* \setminus T[s_0]$ is $\H^n$-contained in $S_0$}\,.
    \end{equation}
    To this end let us pick $k$ such that $\Om_k\cc T$ and $\Om_k\cap T[s_0]=\emptyset$. Then, for $j\ge j(k)$, we have $\Om_k\cap T[s_j]=\varnothing$, so that
    \[
    \Om_k\cap {\rm UBEP}\big(K_j\cup T[s_j];T\big)\subset {\rm UBEP}\big(K_j\cup T[s_j];\Om_k\big)={\rm UBEP}\big(K_j;\Om_k\big)\,.
    \]
    Since $\{U_i^j\}_i$ is the essential partition of $T$ induced by $K_j\cup T[s_j]$, if $\{U_m^j[\Om_k]\}_m$ is the essential partition of $\Om_k$ induced by $K_j$, we have just claimed that, for every $i$ and $j\ge j(k)$,
    \begin{equation}
      \label{localize it}
      \Om_k\cap\pa^*U_i^j\subset\Om_k\cap\bigcup_m \pa^*U_m^j[\Om_k]\,.
    \end{equation}
    Since $\{U_m^j[\Om_k]\}_m$ is a Lebesgue partition of $\Om_k$ into essentially connected sets, by \eqref{localize it} the indecomposable components of $\Om_k\cap U_i^j$ must belong to $\{U_m^j[\Om_k]\}_m$. In other words, for each $i$ and each $j\ge j(k)$ there is $M(k,i,j)$ such that
    \[
    \Om_k\cap U_i^j=\bigcup_{m\in M(k,i,j)}U_m^j[\Om_k]\,.
    \]
    As a consequence of $U_i^j\to U_i$ and of $U_m^j[\Om_k]\to U_m[\Om_k]$ as $j\to\infty$ we find that, for a set of indexes $M(k,i)$, it must be
    \[
    \Om_k\cap U_i=\bigcup_{m\in M(k,i)}U_m[\Om_k]\,,
    \]
    and therefore
    \[
    \Om_k\cap\pa^*U_i\,\,\shn\,\,\bigcup_{m\in M(k,i)}\pa^*U_m[\Om_k]\subset S_0\,.
    \]
    Since we have proved this inclusion for every $i$ and for every $k$ such that $\Om_k\cc T$ with $\Om_k\cap T[s_0]=\emptyset$, it follows that $K^*\setminus T[s_0]$ is $\H^n$-contained in $S_0$, that is \eqref{upon3}.
    \end{proof}

    \section{Existence of minimizers and convergence to Plateau's problem (Theorem \ref{thm existence EL for bulk})}\label{section existence theorems} In this section we prove two main results: the first one (Theorem \ref{corollary existence plateau}) concerns the equivalence of Harrison--Pugh Plateau's problem $\ell$ with its measure theoretic reformulation $\ell_{\rm B}$ (see \eqref{def of ellB}); the second (Theorem \ref{thm existence EL for bulk}) is a very refined version of Theorem \ref{thm existence EL for bulk}.

    \begin{theorem}[Existence for $\ell_{\rm B}$ and $\ell=\ell_{\rm B}$]\label{corollary existence plateau}
      If $\wire\subset\R^{n+1}$ is closed, $\C$ is a spanning class for $\wire$, and the Harrison--Pugh formulation of the Plateau problem
      \begin{align}\notag
        \ell= \inf\big\{\H^n(S) : \mbox{$S$ is a closed subset $\Om$, $S$ is $\C$-spanning $\wire$}\big\}
      \end{align}
      is finite, then the problem
      \begin{align}\notag
        \ell_{\rm B}= \inf\big\{\H^n(S) : \mbox{$S$ is a Borel subset $\Om$, $S$ is $\C$-spanning $\wire$}\big\}
      \end{align}
      admits minimizers, and given any minimizer $S$ for $\ell_{\rm B}$, there exists relatively closed $S^*$ which is $\H^n$-equivalent to $S$ and a minimizer for $\ell$. In particular, $\ell=\ell_{\rm B}$.
    \end{theorem}

    \begin{theorem}[Theorem \ref{thm existence EL for bulk} refined]\label{thm existence EL for bulk section}
    If $\wire$ is a compact set in $\mathbb{R}^{n+1}$ and $\C$ is a spanning class for $\wire$ such that $\ell<\infty$, then for every $v>0$ there exist minimizers $(K,E)$ of $\Psi_{\rm bk}(v)$. Moreover,

    \medskip

    \noindent {\bf (i):} if $(K_*,E_*)$ is a minimizer of $\Psi_{\rm bk}(v)$, then there is $(K,E)\in\KK$ such that $K$ is $\H^n$-equivalent to $K^*$, $E$ is Lebesgue equivalent to $E_*$, $(K,E)$ is a minimizer of $\Psi_{\rm bk}(v)$, both $E$ and $K$ are bounded, $K\cup E$ is $\C$-spanning $\wire$, $K\cap E^\one=\varnothing$, and there is $\lambda\in \mathbb{R}$ such that
    \begin{eqnarray}
    \label{first var cap theorem section}
    \lambda \int_{\Om\cap\partial^* E} X \cdot \nu_{E} \,d\H^n = \int_{\Om\cap\partial^* E} \mathrm{div}^{K}\,X\,d\H^n + 2\int_{K \cap E^\zero} \mathrm{div}^{K}\, X\,d\H^n\,,&&
    \\\nonumber
    \hspace{2.3cm}\forall X\in C^1_c(\mathbb{R}^{n+1};\mathbb{R}^{n+1})\quad\mbox{with $X\cdot \nu_\Omega =0$ on $\partial \Omega$}\,,&&
    \end{eqnarray}
    and there are positive constants $c=c(n)$ and $r_1=r_1(K,E)$ such that
    \begin{equation}
     \label{upper density estimates proof}
      |E\cap B_\rho(y)|\le (1-c)\,\om_{n+1}\,\rho^{n+1}\,,
    \end{equation}
    for every $y\in\Om\cap\pa E$ and $\rho<\min\{r_1,\dist(y,\wire)\}$; under the further assumption that $\partial \wire$ is $C^2$, then there is positive $r_0=r_0(n,\wire,|\l|)$ such that
    \begin{equation}
     \label{lower density estimates proof}
      \H^n(K\cap B_r(x))\ge c\,r^n
    \end{equation}
    for every $x\in\cl(K)$ and $r<r_0$;

    \medskip

    \noindent {\bf (ii):} if $(K_j,E_j)$ is a sequence of minimizers for $\Psi_{\rm bk}(v_j)$ with $v_j\to 0^+$, then there exists a minimizer $S$ of $\ell$ such  that, up to extracting subsequences, as Radon measures in $\Om$,
    \begin{align}
    \label{convergence as measures section}
            \H^n\mres(\Om\cap\partial^* E_j) + 2\,\H^n\mres (K_j \cap E_j^{\zero}) \weakstar 2\H^n\mres S\,,
    \end{align}
    as $j\to\infty$. In particular, $\Psi_{\rm bk}(v)\to 2\,\ell=\Psi_{\rm bk}(0)$ as $v\to 0^+$.
    \end{theorem}

    \begin{proof}[Proof of Theorem \ref{corollary existence plateau}] By Theorem \ref{theorem definitions equivalence}, if  $\ell<\infty$, then $\ell_{\rm B}<\infty$. Let now $\{S_j\}_j$ be a minimizing sequence for $\ell_{\rm B}$, then $\{(S_j,\varnothing)\}_j$ is a sequence in $\KK_{\rm B}$ satisfying \eqref{energy bound capillarity 2}. By Theorem \ref{theorem first closure theorem intro}, we find a Borel set $S$ which is $\C$-spanning $\wire$ and is such that
    \[
    2\,\liminf_{j\to\infty}\H^n(S_j)=\liminf_{j\to\infty}\F_{\rm bk}(S_j,\varnothing)\ge\F_{\rm bk}(S,\varnothing)=2\,\H^n(S)\,.
    \]
    This shows that $S$ is a minimizer of $\ell_{\rm B}$. By Lemma \ref{lemma rectfiable spanning}, $S$ is $\H^n$-rectifiable, for, otherwise, $\mathcal{R}(S)$ would be admissible for $\ell_{\rm B}$ and have strictly less area than $S$. We conclude the proof by showing that up to modifications on a $\H^n$-null set, $S$ is relatively closed in $\Omega$ (and thus is a minimizer of $\ell$ too). Indeed the property of being $\C$-spanning $\wire$ is preserved under diffeomorphism $f$ with $\{f\ne\id\}\cc\Omega$. In particular, $\H^n(S) \leq \H^n(f(S))$ for every such $f$, so that the multiplicity one rectifiable varifold $V_S=\var(S,1)$ associated to $S$ is stationary. By a standard application of the monotonicity formula, we can find $S^*$ $\H^n$-equivalent to $S$ such that $S^*$ is relative closed in $\Om$. Since $\H^n(S)=\H^n(S^*)$ and $\C$-spanning is preserved under $\H^n$-null modifications, we conclude the proof.
    \end{proof}

    \begin{proof}[Proof of Theorem \ref{thm existence EL for bulk section}] {\it Step one}: We prove conclusion (i). To this end, let $(K_*,E_*)\in\KK_{\rm B}$ be a minimizer of $\Psi_{\rm bk}(v)$. Clearly, $(\RR(K_*),E_*)\in\KK_{\rm B}$ is such that $\RR(K_*)\cup E^\one$ is $\C$-spanning $\wire$ (thanks to Theorem \ref{theorem spanning with partition}/Remark \ref{remark RRK cup Eone is spanning too}) and $\F_{\rm bk}(\RR(K_*),E_*)\le\F_{\rm bk}(K_*,E_*)$. In particular, $(\RR(K_*),E_*)$ is a minimizer of $\Psi_{\rm bk}(v)$, and energy comparison between $(\RR(K_*),E_*)$ and $(\RR(K_*)\setminus E_*^\one,E_*)$ (which is also a competitor for $\Psi_{\rm bk}(v)$) proves that
    \begin{equation}
      \label{obe 1}
      \H^n(\RR(K_*)\cap E_*^\one)=0\,.
    \end{equation}
    Since ``$\C$-spanning $\wire$'' is preserved under diffeomorphisms, by a standard first variation argument (see, e.g. \cite[Appendix C]{KingMaggiStuvard}) wee see that $(\RR(K_*),E_*)$ satisfies \eqref{first var cap theorem section} for some $\l\in\R$. In particular, the integer $n$-varifold $V={\rm var}(\RR(K_*),\theta)$, with multiplicity function $\theta=2$ on $\RR(K_*)\cap E_*^\zero$ and $\theta=1$ on $\Om\cap\pa^*E_*$, has bounded mean curvature in $\Om$, and thus satisfies $\|V\|(B_r(x))\ge c(n)\,r^n$ for every $x\in K$ and $r<\min\{r_0,\dist(x,\wire)\}$, where $r_0=r_0(n,|\l|)$ and, by definition,
    \[
    K:=\Om\cap\spt V\,.
    \]
    In particular, since \eqref{obe 1} implies $\|V\|\le 2\,\H^n\mres \RR(K^*)$ , we conclude (e.g. by \cite[Corollary 6.4]{maggiBOOK}) that $K$ is $\H^n$-equivalent to $\RR(K_*)$, and is thus $\H^n$-rectifiable and relatively closed  in $\Om$. Now let
    \[
    E=\big\{x\in\Om:\mbox{$\exists\,\, r<\dist(x,\wire)$ s.t. $|E_*\cap B_r(x)|=|B_r(x)|$}\big\}\,,
    \]
    so that, trivially, $E$ is an open subset of $\Om$ with $E\subset E_*^\one$. By applying \eqref{sofp1} to $E_*$, and by noticing that if $x\in\Om\setminus E$ then $|E_*\cap B_r(x)|<|B_r(x)|$ for every $r>0$, and that if $x\in\Om\cap\cl(E)$ then $|E_*\cap B_r(x)|>0$ for every $r>0$, we see that
    \begin{equation}
      \label{obe 2}
      \Om\cap\pa E\,\,\subset\,\, \big\{x\in\Om:0<|E_*\cap B_r(x)|<|B_r(x)|\,\,\forall r>0\big\}\,\,=\,\,\Om\cap\cl(\pa^*E_*)\,.
    \end{equation}
    Since $\|V\|\ge\H^n\mres(\Om\cap\pa^*E_*)$ and $\H^n(B_r(x)\cap\pa^*E)=\om_n\,r^n+{\rm o}(r^n)$ as $r\to 0^+$ for every $x\in\Om\cap\pa^*E$, we see that $\Om\cap\pa^*E_*\subset\Om\cap\spt\|V\|=K$, and since $K$ is relatively closed in $\Om$, we have $\Om\cap\cl(\pa^*E_*)\subset K$, and so $\Om\cap\pa E\subset K$. In particular, $E$ is of finite perimeter, and thus by applying \eqref{sofp1} to $E$,
    \begin{equation}
      \label{obe 3}
      \Om\cap\cl(\pa^*E)\,\,=\,\,\big\{x\in\Om:0<|E\cap B_r(x)|<|B_r(x)|\,\,\forall r>0\big\}\,\,\subset\,\, \Omega \cap \partial E\,.
    \end{equation}
    Finally, if there is $x\in (\Om\cap E_*^\one)\setminus E$, then it must be $0<|E_*\cap B_r(x)|<|B_r(x)|$ for every $r>0$, and thus $x\in\Om\cap\cl(\pa^*E_*)\subset K$. However, we {\it claim} that for every $x\in \Om\cap\cl(\pa^*E_*)$ and $r<\min\{r_*,\dist(x,\wire)\}$ (with $r_*=r_*(K_*,E_*)$) it holds
    \begin{equation}
      \label{obe 4}
      |B_r(x)\cap E_*|\le (1-c)\,\om_{n+1}\,r^{n+1}\,,
    \end{equation}
    in contradiction with $x\in E^\one$; this proves that $\Om\cap E_*^\one\subset E$, and thus that $E_*$ and $E$ are Lebesgue equivalent. Combining the latter information with \eqref{obe 2} and \eqref{obe 3} we conclude that $\Om\cap\cl(\pa^*E)=\Om\cap\pa E \subset K$ and conclude the proof of $(K,E)\in\KK$ -- conditional to proving \eqref{obe 4}.

    \medskip

    To prove \eqref{obe 4}, let us fix $x\in \Om\cap\cl(\pa^*E_*)$ and set $u(r)=|B_r(x)\setminus E_*|$, so that, for a.e. $r>0$ we have
    \begin{eqnarray}\label{obe 5}
    u'(r)=\H^n(E_*^\zero\cap\pa B_r(x))\,,\qquad P(B_r(x)\setminus E_*)=u'(r)+P(E_*;B_r(x))\,.
    \end{eqnarray}
    Since $|E_*|=v>0$, we have $\H^n(\Om\cap\pa^*E_*)>0$, therefore there must be $y_1,y_2\in\Om\cap\pa^*E_*$ with $|y_1-y_2|> 4 r_*$ for some $r_*$ depending on $E_*$. In particular there is $i\in\{1,2\}$ such that $B_{r_*}(x)\cap B_{r_*}(y_i)=\varnothing$, and we set $y=y_i$. Since $y_i\in\Om\cap\pa^*E_*$, there is $w_*>0$ and smooth maps $\Phi:\Om\times(-w_*,w_*)\to\Om$ such that $\Phi(\cdot,w)$ is a diffeomorphism of $\Om$ with $\{\Phi(\cdot,w)\ne\Id\}\cc B_{r_*}(y)$, and
    \begin{equation}
      \label{obe 6}
      |\Phi(E_*,w)|=|E_*|-w\,,\qquad P(\Phi(E_*,w);B_{r_*}(y))\le P(E_*,B_{r_*}(y))(1+ 2\,|\l|\,|w|)\,,
    \end{equation}
    for every $|w|<w_*$. We then consider $r_1$ such that $|B_{r_1}|<w_*$, so that for every $r<\min\{r_1,\dist(x,\wire)\}$ we have $0\le u(r)<w_*$, and thus we can define
    \[
    (K_r,E_r)=\Big(\Phi^{u(r)}\big(K\cup\pa B_r(x)\big),\Phi^{u(r)}\big(E_*\cup B_r(x)\big)\Big)\,.
    \]
    Since $\Phi^{u(r)}$ is a diffeomorphism, we have  $\Om\cap\pa^*E_r\subset K_r$, and by the first relation in \eqref{obe 6} and $\Phi^{u(r)}=\Id$ on $\Om\setminus B_{r_*}(y)$, we get
    \[
    |E_r|-|E|=|B_r(x)|-|B_r(x)\cap E_*|+|\Phi^{u(r)}(E_*)\cap B_{r_*}(y)|-|E_*\cap B_{r_*}(y)|=u(r)-u(r)=0\,.
    \]
    Hence $\F_{\rm bk}(K_*,E_*)\le\F_{\rm bk}(K_r,E_r)$, from which we deduce
    \begin{eqnarray*}
      &&P(E;B_r(x))+P(E;B_{r_*}(y))+2\,\H^n(K_*\cap E_*^\zero\cap B_r(x))
      \\
      &&\le \H^n(B_r(x)\cap E^\zero)+P(\Phi^{u(r)}(E_*);B_{r_*}(y))\le u'(r)+P(E_*,B_{r_*}(y))(1+ 2\,|\l|\,|w|)\,;
    \end{eqnarray*}
    where we have used \eqref{obe 5} and \eqref{obe 6}; by adding up $u'(r)$ on both sides of the inequality, and using \eqref{obe 5} again, we find that
    \[
    c(n)\,u(r)^{n/(n+1)}\le P(B_r(x)\setminus E_*)\le 2\,u'(r)+2\,|\l|\,\Psi_{\rm bk}(v)\,u(r)\,,
    \]
    for a.e. $r<\min\{r_1,\dist(x,\wire)\}$; since, by \eqref{obe 2}, $x\in \Om\cap \cl(\pa^*E_*)$ implies $u(r)>0$ for every $r>0$, we can apply a standard ODE argument to conclude that \eqref{obe 4} holds true.

    \medskip

    We now prove the remaining assertions in statement (i). First of all,  when $\partial \wire$ is $C^2$, we can argue similarly to \cite[Theorem 4.1]{KMS3} to deduce from the modified monotonicity formula of Kagaya and Tonegawa \cite{kagayatone} that the area lower bound in \eqref{lower density estimates proof} holds for every $x\in\cl(K)$ and every $r<r_0$. The validity of the volume upper bound in \eqref{upper density estimates proof} is immediate from \eqref{obe 4} and the Lebesgue equivalence of $E_*$ and $E$. The monotonicity formula for $V$ combined with $\H^n(\Om\cap K)<\infty$ implies of course that $V$ has bounded support. Having proved that $K$ is bounded, $|E|<\infty$ and $\Om\cap\pa E\subset K$ imply that $E$ is bounded too. Since $\RR(K_*)$ and $K$ are $\H^n$-equivalent, we have that $K\cup E_*^\one$ is $\C$-spanning $\wire$. It turns out that $K\cup E^\one$ is $\C$-spanning $\wire$ too, since $E$ and $E_*$ are Lebesgue equivalent {\it and} of finite perimeter, therefore such that $E^\one$ and $E_*^\one$ are $\H^n$-equivalent. In fact, on noticing that $\Om\cap(E^\one\setminus E)\subset\Om\cap\pa E\subset K$, we see that $K\cup E^\one=K\cup E$, so that $K\cup E$ is $\C$-spanning $\wire$, as claimed.

    \medskip

    Finally, we prove that $K\cap E^\one=\varnothing$. We first notice that, since $E\subset\Om$ is open and $K=\Om\cap\spt\,V$ with $\|V\|\le 2\,\H^n\mres\RR(K^*)$, if $K\cap E\ne\emptyset$, then $\H^n(\RR(K_*)\cap E)>0$; and since $E\subset E_*^\one$ by construction, we arrive at a contradiction with \eqref{obe 1}. Hence, $K\cap E=\varnothing$. Now, if $x\in K\cap E^\one$, then, by \eqref{upper density estimates proof}, $x\not\in\Om\cap\pa E$; combining this with $K\cap E=\varnothing$, we find $K\cap E^\one\subset\Om\setminus\cl(E)\subset E^\zero$, and thus $K\cap E^\one=\varnothing$.

    \medskip

    \noindent {\it Step two}: For every $v_1\ge0$ and $v_2>0$ we have
    \begin{equation}
      \label{psiv1 v2}
      \Psi_{\rm bk}(v_1+v_2)\le\Psi_{\rm bk}(v_1)+(n+1)\,\om_{n+1}^{1/(n+1)}\,v_2^{n/(n+1)}\,.
    \end{equation}
    Since $\Psi_{\rm bk}(0)=2\,\ell<\infty$, \eqref{psiv1 v2} implies in particular that $\Psi_{\rm bk}(v)<\infty$ for every $v>0$ (just take $v_1=0$ and $v_2=v$).

    \medskip

    Indeed, let $(K_1,E_1)$ be a competitor in $\Psi_{\rm bk}(v_1)$ and let $\{B_{r_j}(x_j)\}_j$ be a sequence of balls with $|x_j|\to\infty$ and $|E_1\cup B_{r_j}(x_j)|=v_1+v_2$ for every $j$. Setting for the sake of brevity $B_j=B_{r_j}(x_j)$, sine $\pa^*(E_1\cup B_j)$ is $\H^n$-contained in $(\pa^*E_1)\cup\pa B_j$ we have that $(K_2,E_2)$, with $K_2=K_1\cup\pa B_j$ and $E_2=E_1\cup B_j$, is a competitor of $\Psi_{\rm bk}(v_1+v_2)$. Since $\pa B_j\cap E_2^\zero=\varnothing$ implies $E_2^\zero\subset E_1^\zero\setminus\pa B_j$, we find that
    \begin{eqnarray*}
      \Psi_{\rm bk}(v_1+v_2)&\le&2\,\H^n\big(K_2\cap E_2^\zero)+\H^n(\Om\cap\pa^*E_2)
      \\
      &\le& 2\,\H^n(K_1\cap E_1^\zero\setminus\pa B_j)+\H^n(\Om\cap\pa^*E_1)+\H^n(\pa B_j)
      \\
      &\le& \F_{\rm bk}(K_1,E_1)+(n+1)\,\om_{n+1}^{1/(n+1)}\,|B_j|^{n/(n+1)}\,.
    \end{eqnarray*}
    Since $|x_j|\to\infty$, $|E_1|=v_1$, and $|E_1\cup B_{r_j}(x_j)|=v_1+v_2$ imply $|B_j|\to v_2$, we conclude by arbitrariness of $(K_1,E_1)$.

    \medskip

    \noindent {\it Step three}: Now let $\{(K_j,E_j)\}_j$ be a minimizing sequence  for $\Psi_{\rm bk}(v)$. Since $\Psi_{\rm bk}(v)<\infty$, assumption \eqref{energy bound capillarity 2} of Theorem \ref{theorem first closure theorem intro} holds. Therefore there is $(K,E)\in\KK_{\rm B}$ with $K \cup E^{\one}$ is $\C$-spanning $\wire$ and such that, up to extracting subsequences,
    \begin{equation}
      \label{lsc Fb section}
    \lim_{j\to\infty}|(E_j\Delta E)\cap B_R|=0\quad\forall R>0\,,\qquad  \liminf_{j\to\infty}\F_{\rm bk}(K_j,E_j)\ge\F_{\rm bk}(K,E)\,;
    \end{equation}
    actually, to be more precise, if $\mu$ denotes the weak-star limit of $\H^n\mres (\Om\cap\pa^*E_j) + 2\,\H^n \mres (\mathcal{R}(K_j) \cap E_j^\zero)$ in $\Om$, then
    \begin{equation}
      \label{mu larger than}
         \mu\ge 2\,\H^n \mres (K \cap E^\zero)+\H^n\mres (\Om\cap\partial^* E)\,.
    \end{equation}
    We {\it claim} that
    \[
    \mbox{$(K,E)$ is a minimizer of $\Psi_{\rm bk}(|E|)$}\,.
    \]
    (Notice that, at this stage of the argument, we are not excluding that $v^*:=v-|E|$ is positive, nor that $|E|=0$.) Taking into account \eqref{psiv1 v2}, to prove the claim it suffices to show that
    \begin{align}\label{lower bound}
        \Psi_{\rm bk}(v) \geq \F_{\rm bk}(K,E) + (n+1)\,\omega_{n+1}^{1/(n+1)}\,(v^*)^{n/(n+1)}\,.
    \end{align}
    To see this, we start noticing that, given any sequence $\{r_j\}_j$ with $r_j\to\infty$, by \eqref{lsc Fb section} and \eqref{mu larger than} we have that
    \begin{eqnarray}
    \label{right volume inside}
    &&E_j \cap B_{r_j}\toloc E\,,\qquad |E_j \setminus B_{r_j}|\to v^*\,,\qquad\mbox{as $j\to\infty$}\,,
    \\
    \label{right energy inside}
    &&\liminf_{j\to \infty}\,2\,\H^n\big(\RR(K_j)\cap E_j^\zero\cap B_{r_j}\big)+\H^n(B_{r_j}\cap\pa^*E_j)\ge
    \F_{\rm bk}(K,E)\,,
    \end{eqnarray}
    Moreover, since $|E_j|<\infty$, we can choose $r_j\to \infty$ so that $\H^n(E_j^\one\cap\pa B_{r_j})\to 0$, while, taking into account that $P(E_j\setminus B_{r_j})=\H^n(E_j^\one\cap \pa B_{r_j})+\H^n((\pa^*E_j)\setminus B_{r_j})$, we have
    \begin{eqnarray*}
    \F_{\rm bk}(K_j,E_j)&\ge&2\,\H^n\big(\RR(K_j)\cap E_j^\zero\cap B_{r_j}\big)+\H^n(B_{r_j}\cap\pa^*E_j)
    \\
    &&+P(E_j\setminus B_{r_j})-\H^n(E_j^\one\cap\pa B_{r_j})\,.
    \end{eqnarray*}
    By combining these facts with \eqref{right volume inside}, \eqref{right energy inside}, and the Euclidean isoperimetric inequality, we conclude that
    \begin{eqnarray*}
      \Psi_{\rm bk}(v)=\lim_{j\to\infty}\F_{\rm bk}(K_j,E_j)\ge  \F_{\rm bk}(K,E)+(n+1)\,\om_{n+1}^{1/(n+1)}\,\lim_{j\to\infty}|E_j\setminus B_{r_j}|^{n/(n+1)}\,,
    \end{eqnarray*}
    that is \eqref{lower bound}.

    \medskip

    \noindent {\it Step four}: We prove the existence of minimizers in $\Psi_{\rm bk}(v)$, $v>0$. By step three, there is $(K,E)\in\KK_{\rm B}$ such that $K\cup E^\one$ is $\C$-spanning $\wire$, $(K,E)$ is a minimizer of $\Psi_{\rm bk}(|E|)$ and, combining \eqref{psiv1 v2} and \eqref{lower bound},
    \begin{equation}
      \label{lower bound contra}
      \Psi_{\rm bk}(v)=\Psi_{\rm bk}(|E|)+(n+1)\,\omega_{n+1}^{1/(n+1)}\,(v-|E|)^{n/(n+1)}\,.
    \end{equation}
    Since $(K,E)$ is a minimizer in $\Psi_{\rm bk}(|E|)$, by step one we can assume that $K$ is $\H^n$-rectifiable and that both $K$ and $E$ are bounded. We can thus find $B_r(x_0)\cc\Om$ such that $|B_r(x_0)|=v-|E|$, $|B_r(x_0)\cap E|=0$, and $\H^n(K\cap B_r(x_0))=0$. In this way $(K_*,E_*)=(K\cup\pa B_r(x_0),E\cup B_r(x_0))\in\KK_{\rm B}$ is trivially $\C$-spanning $\wire$ and such that $|E_*|=v$, and thus is a competitor for $\Psi_{\rm bk}(v)$. At the same time,
    \[
    \F_{\rm bk}(K_*,E_*)=\F_{\rm bk}(K,E)+(n+1)\,\omega_{n+1}^{1/(n+1)}\,(v-|E|)^{n/(n+1)}
    \]
    so that, by \eqref{lower bound contra}, $(K_*,E_*)$ is a minimizer of $\Psi_{\rm bk}(v)$. Having proved that minimizers of $\Psi_{\rm bk}(v)$ do indeed exist, a further application of step one completes the proof of statement (i).

    \medskip

    \noindent {\it Step five}: We finally prove statement (ii). Let us consider a sequence $v_j\to 0^+$ and corresponding minimizers $(K_j,E_j)$ of $\Psi_{\rm bk}(v_j)$. By \eqref{psiv1 v2} with $v_1=0$ and $v_2=v_j$ we see that $\{(K_j,E_j)\}_j$ satisfies the assumptions of Theorem \ref{theorem first closure theorem intro}. Since $|E_j|=v_j\to 0$, setting $\mu_j= \H^n\mres (\Om\cap\pa^*E_j) + 2\,\H^n \mres (\mathcal{R}(K_j) \cap E_j^\zero)$, the conclusion of Theorem \ref{theorem first closure theorem intro} is that there are a Radon measure $\mu$ in $\Om$ and a Borel set $K$ such that $K$ is $\C$-spanning $\wire$ and $\mu_j\weakstar\mu$ for a Radon measure $\mu$ in $\Om$ such that $\mu\ge 2\,\H^n\mres K$. Thanks to \eqref{psiv1 v2} we thus have
    \begin{eqnarray*}
      2\,\ell&=&\lim_{j\to\infty}\Psi_{\rm bk}(0)+(n+1)\,\om_{n+1}^{1/(n+1)}\,v_j^{n/(n+1)}
      \ge \liminf_{j\to\infty}\Psi_{\rm bk}(v_j)
      \\
      &=&\liminf_{j\to\infty}\F_{\rm bk}(K_j,E_j)
      \ge \F_{\rm bk}(K,\emptyset)= 2\,\H^n(K)\ge 2\,\ell\,.
    \end{eqnarray*}
    We conclude that $\Psi_{\rm bk}(v_j)\to 2\,\ell$, $K$ is a minimizer of $\ell$, and $\mu=2\,\H^n\mres K$, thus completing the proof of the theorem.
    \end{proof}

    \begin{proof}
      [Proof of Theorem \ref{thm existence EL for bulk}] The identity \eqref{ellb equals ell} is proved in Theorem \ref{corollary existence plateau}. Conclusions (i), (ii), and (iii) are proved in Theorem \ref{thm existence EL for bulk section}.
    \end{proof}

\section{Equilibrium across transition lines in wet soap films (Theorem \ref{thm regularity for bulk})}\label{section regularity transition regions} We finally prove Theorem \ref{thm regularity for bulk}. We shall need two preliminary lemmas:

    \begin{lemma}[Representation of $\mathcal{F}_{\rm bk}$ via induced partitions]\label{lemma breakdown of F via open components}
    If $U\subset \Omega$ is a set of finite perimeter, $(K,E)\in \mathcal{K}_{\rm B}$ is such that $\mathcal{F}_{\rm bk}(K,E) <\infty$, and $\{U_i\}_{i}$ is a Lebesgue partition of $U\setminus E$ induced by $K$, then each $U_i$ has finite perimeter, and, setting $K^*=\bigcup_i\pa^*U_i$, we have
    \begin{align}\label{ke decomp formula}
        \F_{\rm bk}(K,E;U^\one)= \sum_{i} \H^n(U^\one\cap\pa^*U_i) + 2\,\mathcal{H}^n\big( U^\one\cap (K\setminus K^*) \cap E^\zero\big)\,;
    \end{align}
    see
    \begin{figure}
    \input{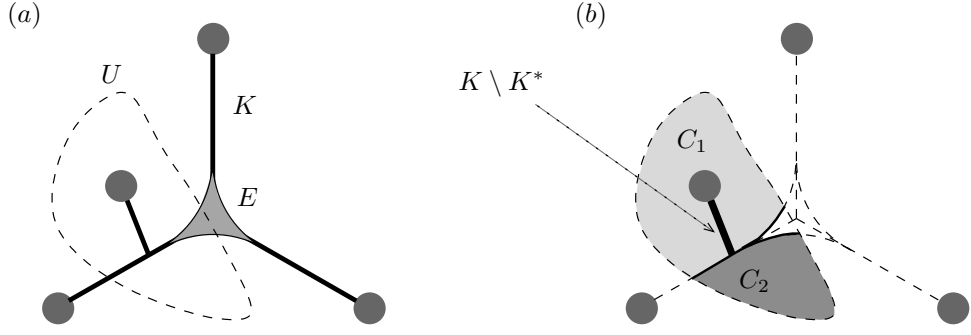}
    \caption{\small{The situation in Lemma \ref{lemma breakdown of F via open components}: (a) a depiction of the left hand side of \eqref{ke decomp formula}, where $K\setminus\pa^*E$ is drawn with a bold line to indicate that, in the computation of $\F_{\rm bk}(K,E;U^\one)=\H^n(U^\one\cap\pa^*E)+2\,\H^n(U^\one\cap K\setminus\pa^*E)$, it is counted with multiplicity $2$; (b) a depiction of the right hand side of \eqref{ke decomp formula}, where $K\setminus K^*$ is drawn with a bold line to indicate that it has to be counted with multiplicity $2$.}}
    \label{fig c1c2}
    \end{figure}
    Figure \ref{fig c1c2}.
    \end{lemma}

    \begin{proof} For each $i$, $\pa^eU_i$ is contained in $(\pa^eU)\cup(\pa^eE)\cup(U\setminus E)^\one$, where both $\pa^eU$ and $\pa^eE$ are $\H^n$-finite being $U$ and $E$ of finite perimeter, and where $(U\setminus E)^\one\cap\pa^eU_i$ is $\H^n$-contained in $K$ by assumption. Now, $(U\setminus E)^\one\subset\R^{n+1}\setminus E^\one$, so that
    \[
    \H^n\big((U\setminus E)^\one\cap\pa^eU_i\big)\le\H^n(K\setminus E^\one)\le \F_{\rm bk}(K,E)<\infty\,.
    \]
    This shows that, for each $i$, $U_i$ is a set of finite perimeter. As a consequence $\{U\cap E\}\cup\{U_i\}_i$ is a Caccioppoli partition of $U$, so that, by \eqref{caccioppoli partitions},
    \begin{equation}
      \label{h3}
      2\,\H^n\Big(U^\one\cap\Big[\pa^*(U\cap E)\cup K^*\Big]\Big)=
    \H^n\big(U^\one\cap\pa^*(U\cap E)\big)+\sum_i\H^n(U^\one\cap\pa^*U_i)\,,
    \end{equation}
    with $K^*=\bigcup_i\pa^*U_i$. Now, thanks to \eqref{E cap F}, \eqref{E minus F}, and the inclusion in \eqref{caccioppoli partitions}, we have
    \[
    U^\one\cap\pa^*(U\cap E)\ehn U^\one\cap\pa^*E \shn U^\one\cap K^*\,,
    \]
    which combined with \eqref{h3} gives
    \begin{equation}
      \label{cacciop partition}
      2\,\H^n(U^\one\cap K^*)=
    \H^n\big(U^\one\cap\pa^* E\big)+\sum_i\H^n(U^\one\cap\pa^*U_i)\,.
    \end{equation}
    Therefore, using in order
    \[
    U^\one\cap\pa^*E\shn U^\one\cap K^*\,,\qquad K^*\shn K\,,\qquad \H^n(K^*\cap E^\one)=0\,,
    \]
    and Federer's theorem \eqref{federer theorem}, we obtain
    \begin{eqnarray*}
      \F_{\rm bk}(K,E;U^\one)&=&\H^n(U^\one\cap\pa^*E)+2\,\H^n(U^\one\cap K\cap E^\zero)
      \\
      &=&2\,\H^n(U^\one\cap K^*\cap \pa^*E)-\H^n(U^\one\cap\pa^*E)
      \\
      &&+2\,\H^n(U^\one\cap K^*\cap E^\zero)+2\,\H^n(U^\one\cap (K\setminus K^*)\cap E^\zero)
      \\
      &=&2\,\H^n(U^\one\cap K^*)-\H^n(U^\one\cap\pa^*E)+2\,\H^n(U^\one\cap (K\setminus K^*)\cap E^\zero)
      \\
      &=&\sum_i\H^n(U^\one\cap\pa^*U_i)+2\,\H^n(U^\one\cap (K\setminus K^*)\cap E^\zero)\,,
    \end{eqnarray*}
    where in the last identity we have used \eqref{cacciop partition}.
    \end{proof}

    The next lemma is a slight reformulation of \cite[Lemma 10]{DLGM} and \cite[Lemma 4.1]{delederosaghira}.

    \begin{lemma}\label{dldrg lemma}
        If $\wire$ is closed, $\mathcal{C}$ is a spanning class for $\wire$, $S$ is relatively closed in $\Om$ and $\mathcal{C}$-spanning $\wire$, and $B\subset\Om$ is an open ball, then for any $\gamma \in \mathcal{C}$ we either have $\gamma(\SS^1) \cap (S\setminus B)\neq \varnothing$, or $\gamma(\SS^1)$ has non-empty intersection with at least two connected components of $B \setminus S$. In particular, it intersects the boundaries of both components.
    \end{lemma}

    We are now ready for the proof of Theorem \ref{thm regularity for bulk}.

    \begin{proof}[Proof of Theorem \ref{thm regularity for bulk}] The opening part of the statement of Theorem \ref{thm regularity for bulk} is Theorem \ref{thm existence EL for bulk section}-(i), therefore we can directly consider a minimizer $(K,E)\in\KK$ of $\Psi_{\rm bk}(v)$ such that both $E$ and $K$ are bounded, $K\cup E$ is $\C$-spanning $\wire$, and
    \begin{equation}
      \label{no collapsing proof}
      K\cap E^\one=\varnothing\,,
    \end{equation}
    and begin by proving the existence of a closed set $\Sigma\subset K$ closed such that (i): $\Sigma=\varnothing$ if $1\leq n \leq 6$, $\Sigma$ is locally finite in $\Omega$ if $n=7$, and $\mathcal{H}^{s}(\Sigma)=0$ for every $s>n-7$ if $n\geq 8$; (ii): $(\pa^*E)\setminus\Sigma$ is a smooth hypersurface with constant mean curvature; (iii) $K\setminus(\cl(E)\cup\Sigma)$ is a smooth minimal hypersurface; (iv)$_\a$: if $x\in[\Om\cap(\pa E\setminus\pa^*E)]\setminus\Sigma$, then there are $r>0$, $\nu\in\SS^n$, $u_1,u_2\in C^{1,\a}(\DD_r^\nu(x);(-r/4,r/4))$ ($\a\in(0,1/2)$ arbitrary) such that $u_1(x)=u_2(x)=0$, $u_1\le u_2$ on $\DD_r^\nu(x)$, $\{u_1<u_2\}$ and ${\rm int}\{u_1=u_2\}$ are both non-empty, and
    \begin{eqnarray}\label{grafico K proof}
    \CC_r^\nu(x)\cap K&=&\cup_{i=1,2}\big\{y+u_i(y)\,\nu:y\in\DD_r^\nu(x)\big\}\,,
    \\\label{grafico pa E proof}
    \CC_r^\nu(x)\cap \pa^*E&=&\cup_{i=1,2}\big\{y+u_i(y)\nu:y\in\{u_1<u_2\}\big\}\,,
    \\\label{grafico E proof}
    \CC_r^\nu(x)\cap E&=&\big\{y+t\,\nu:y\in\{u_1<u_2\}\,,u_1(x)<t<u_2(x)\big\}\,.
    \end{eqnarray}
    (The sharp version of conclusion (iv), that is conclusion (iv)$_\a$ with $\a=1$, and conclusion (v), will be proved in the final step five of this proof.) The key step to prove conclusions (i)--(iv)$_\a$ is showing the validity of the following claim.

    \medskip

    \noindent {\it Claim}: There exist positive constants $\Lambda$ and $r_0$ such that if $B_{2r}(x)\cc \Omega$, then, denoting by $\{U_j\}_j$ the open connected components of $B_{2r}(x)\setminus (E \cup K)$,
    \begin{eqnarray}
    \label{K char}
    &&B_r(x)\cap K = B_r(x)\cap\cup_j\pa U_j\,,
    \\
    \label{finitely many in BR}
    &&\#\big\{i:B_r(x)\cap U_j\ne\varnothing\}<\infty\,,
    \\
    \label{normalization}
    &&B_{2\,r}(x)\cap \cl (\pa^* U_j) =B_{2\,r}(x)\cap \pa U_j\,,
    \\
    \label{lambda minimality ineq for C}
    && P(U_j;B_r(x))\leq P(V_j;B_r(x)) + \Lambda\,|U_j\Delta V_j|\,,
    \end{eqnarray}
    whenever $V_j$ satisfies $V_j\Delta U_j \cc B_{r}(x)$ and $\diam(U_j\Delta V_j)<r_0$.

    \medskip

    \noindent {\it Deduction of (i)-(iv) from the claim}: Let $\{B_{2r_i}(x_i)\}_{i\in \mathbb{N}}$ be a countable family of balls, locally finite in $\Omega$, such that $B_{2r_i}(x_i)\cc \Omega$ and $\Omega = \cup_i B_{r_i}(x_i)$. Setting for brevity
    \[
    \Om_i=B_{r_i}(x_i)\,,
    \]
    by \eqref{finitely many in BR} there are finitely many connected components $\{U_j^i\}_{j=1}^{J_i}$ of $B_{2r_i}(x_i)\setminus (E \cup K)$ such that $U_j^i \cap \Om_i\neq \varnothing$. Thanks to \eqref{lambda minimality ineq for C}, we deduce from \cite[Theorem 28.1]{maggiBOOK} that, if we set $\Sigma_j^i=\Om_i\cap(\pa U_j^i \setminus \pa^* U_j^i)$, then $\Om_i\cap\pa^*U_j^i$ is a $C^{1,\a}$-hypersurface for every $\a\in(0,1/2)$, and $\Sigma_j^i$ is a closed set that satisfies the dimensional estimates listed in conclusion (i). In particular, if we set
    \begin{equation}\label{def of singular set}
    \Sigma=\cup_{i\in \mathbb{N}} \cup_{j=1}^{J_i} \Sigma^i_j\,,
    \end{equation}
    then $\Sigma\subset K$ thanks to $\Sigma_j^i\subset \Om_i\cap \pa U_j^i $ and to \eqref{K char}, and conclusion (i) holds by the local finiteness of the covering $\{B_{2r_i}(x_i)\}_i$ of $\Om$ and from $J_i<\infty$ for every $i$. Before moving to prove the remaining conclusions, we first notice that \eqref{K char} gives
    \begin{eqnarray}\nonumber
    \Om_i\cap K\setminus\Sigma&=&\Om_i\cap\cup_{j=1}^{J_i}\pa U_j^i\setminus\Sigma
    \\
    &\subset&\Om_i\cap\cup_{j=1}^{J_i} (\pa U_j^i\setminus\Sigma^i_j)
    \,\,=\,\,\Om_i\cap\cup_{j=1}^{J_i} \pa^* U_j^i\,;
    \label{K breakdown}
    \end{eqnarray}
    second, we notice that, since $K$ is $\H^n$-finite,
    \begin{equation}
      \label{is cacc}
      \mbox{$\{E\cap\Om_i,U_i^j\cap\Om_i\}_{j=1}^{J_i}$ is a Caccioppoli partition of $\Om_i$}\,;
    \end{equation}
    finally, we recall that, by \eqref{euler lagrange equation bulk}, for every $X\in C^1_c(\Om;\mathbb{R}^{n+1})$ it holds
    \begin{align}\label{euler lagrange equation bulk 2}
        \lambda\, \int_{\partial^* E} X \cdot \nu_{E} \,d\H^n = \int_{\partial^* E} \mathrm{div}^{K}\,X\,d\H^n
         + 2\,\int_{K \cap E^{\zero}} \mathrm{div}^{K}\, X\,d\H^n\,.
    \end{align}

    \noindent {\it To prove conclusion (ii)}: Given $x\in\Om\cap\pa^*E\setminus\Sigma$, there is $i\in\N$ such that $x\in \Om_i\cap\pa^*E$. By $\Om\cap\pa^*E\subset K$ and by \eqref{K breakdown} there is $j(x)\in\{1,...,J_i\}$ such that $x\in\pa^*U_{j(x)}^i$. By \eqref{is cacc}, we can use \eqref{caccioppoli exactly two} and $x\in\Om\cap\pa^*E\cap\pa^*U_{j(x)}^i$ to deduce that
    \begin{equation}
      \label{not in}
      x\not\in\cup_{j\ne j(x)}\pa^*U_j^i\,.
    \end{equation}
    Let $r>0$ be such that $B_r(x)\cap\pa^*U_{j(x)}^i$ is a $C^1$-hypersurface. Since $\Sigma$ contains $\cup_j\pa U_j^i$ and \eqref{normalization} holds, \eqref{not in} implies that there is $r>0$ such that
    \begin{equation}
      \label{taking into}
      B_r(x)\cc\Om_i\setminus\Sigma\,,\qquad B_r(x)\cap \cup_j\pa U_j^i=B_r(x)\cap\pa U_{j(x)}^i=B_r(x)\cap\pa^* U_{j(x)}^i\,.
    \end{equation}
    Since $B_r(x)\cap \cup_{j\ne j(x)}\pa U_j^i=\varnothing$ and $B_r(x)\cap U_{j(x)}^i\ne\varnothing$, we also have that
    \[
    B_r(x)\cap \cup_j U_j^i=B_r(x)\cap U_{j(x)}^i\,,
    \]
    and thus, by \eqref{is cacc}, that $\{E\cap B_r(x),U_{j(x)}^i\cap B_r(x)\}$ is an $\H^n$-partition of $B_r(x)$. In particular, $B_r(x)\cap\pa^*E=B_r(x)\cap\pa^*U_{j(x)}^i$: intersecting with $B_r(x)$ in \eqref{K breakdown} and taking into account \eqref{taking into}, we conclude that
    \begin{eqnarray*}
    B_r(x)\cap K&=&B_r(x)\cap[\Om_i\cap K\setminus\Sigma]\,\,\subset\,\, B_r(x)\cap[\Om_i\cap\cup_{j=1}^{J_i} \pa^* U_j^i]\,\,=\,\,B_r(x)\cap\pa^*U_{j(x)}^i
    \\
    &=&B_r(x)\cap\pa^*E\,,
    \end{eqnarray*}
    and \eqref{euler lagrange equation bulk 2} implies that, for every $X\in C^1_c(B_r(x);\mathbb{R}^{n+1})$,
    \begin{align}\label{euler lagrange equation bulk 3}
        \lambda \int_{\partial^* E} X \cdot \nu_{E} \,d\H^n = \int_{\partial^* E} \mathrm{div}^{K}\,X\,d\H^n\,.
    \end{align}
    Hence, $\pa^*E$ can be represented, locally in $B_r(x)$, as the graph of distributional solutions of class $C^{1,\a}$ to the constant mean curvature equation. By Schauder's theory, $B_r(x)\cap\pa^*E$ is a smooth hypersurface whose mean curvature with respect to $\nu_E$ is equal to $\l$ thanks to \eqref{euler lagrange equation bulk 3}.

    \medskip

    \noindent {\it To prove conclusions (iii) and (iv)}: Let us now pick $x\in K\setminus(\Sigma\cup\pa^*E)$ and let $i\in\N$ be such that $x\in\Om_i\cap K$. Let $i\in\N$ be such that $x\in\Om_i$. By \eqref{K breakdown} there is $j(x)\in\{1,...,J_i\}$ such that $x\in\pa^*U_{j(x)}^i$. By \eqref{is cacc} and by \eqref{caccioppoli exactly two}, either $x\in\pa^*E$ (which is excluded from the onset), or there is $k(x)\ne j(x)$ such that $x\in\pa^*U_{k(x)}^i$. We have thus proved that
    \begin{equation}
      \label{jx kx}
      x\in\pa^*U_{j(x)}^i\cap \pa^*U_{k(x)}^i\,,\qquad x\not\in\cup_{j\ne j(x),k(x)}\pa^*U_j^i\,.
    \end{equation}
    To prove conclusion (iii) we notice that if we are in the case when $x\in K\setminus(\Sigma\cup\pa E)=K\setminus(\Sigma\cup\cl(E))$ (thanks to $K\cap E=\varnothing$), then $x\not\in\cl(E)$ implies that, for some $r>0$, $B_r(x)\cap(\Sigma\cup\cl(E))=\emptyset$. In particular, by \eqref{is cacc} and \eqref{jx kx}, $\{B_r(x)\cap U_{j(x)}^i,B_r(x)\cap U_{k(x)}^i\}$ is an $\H^n$-partition of $B_r(x)$, and by \eqref{K breakdown}
    \[
    B_r(x)\cap K=B_r(x)\cap \pa^*U_{j(x)}^i=B_r(x)\cap\pa^*U_{k(x)}^i\,,
    \]
    is a $C^{1,\a}$-hypersurface. Under these conditions, \eqref{euler lagrange equation bulk 2} boils down to
    \begin{align}\label{euler lagrange equation bulk 5}
    \int_{K} \mathrm{div}^{K}\, X\,d\H^n=0\,,\qquad\forall X\in C^1_c(B_r(x);\mathbb{R}^{n+1})\,,
    \end{align}
    so that $K$ can be represented, locally in $B_r(x)$, as the graph of distributional solutions to the minimal surfaces equation of class $C^{1,\a}$. By Schauder's theory, $B_r(x)\cap K$ is a smooth minimal surface.

    To finally prove conclusion (iv), let us assume that $x\in\Om\cap(\pa E\setminus\pa^*E)\setminus\Sigma$. In this case \eqref{is cacc} and \eqref{jx kx} do not imply that $\{B_r(x)\cap U_{j(x)}^i,B_r(x)\cap U_{k(x)}^i\}$ is an $\H^n$-partition of $B_r(x)$; actually, by $\Om\cap\pa E=\Om\cap\cl(\pa^*E)$, the fact that $x\in\pa E$ implies that $B_s(x)\cap\pa^*E\ne\emptyset$ for every $s>0$, so that $|B_s(x)\cap E|>0$ for every $s>0$, and the situation is such that, for every $s<r$,
    \begin{equation}
      \label{again1}
      \mbox{$\{B_s(x)\cap U_{j(x)}^i,B_s(x)\cap U_{k(x)}^i, B_s(x)\cap E\}$ is an $\H^n$-partition of $B_s(x)$}
    \end{equation}
    with all three sets in the partition having positive measure.

    Now, by the first inclusion in \eqref{jx kx}, there exists $\nu\in\SS^n$ such that, up to further decrease the value of $r$ and for some $u_1,u_2\in C^{1,\a}(\DD_r^\nu(x);(-r/4,r/4))$ with $u_1(x)=u_2(x)=0$ and $\nabla u_1(x)=\nabla u_2(x)=0$ it must hold
    \begin{eqnarray*}
    \CC_r^\nu(x)\cap U_{j(x)}^i=\big\{y+t\,\nu:y\in\DD_r^\nu(x)\,,t> u_2(y)\big\}\,,
    \\
    \CC_r^\nu(x)\cap U_{k(x)}^i=\big\{y+t\,\nu:y\in\DD_r^\nu(x)\,,t< u_1(y)\big\}\,.
    \end{eqnarray*}
    By $U_{j(x)}^i\cap U_{k(x)}^i=\varnothing$ we have $u_1\le u_2$ on $\DD_r^\nu(x)$, so that \eqref{again1} gives
    \[
    \CC_r^\nu(x)\cap E=\big\{y+t\,\nu:y\in\{u_1<u_2\}\,,u_1(y)<t< u_2(y)\big\}\,,
    \]
    and $\{u_1<u_2\}$ is non-empty. Again by \eqref{jx kx} and \eqref{K breakdown} we also have that
    \begin{eqnarray*}
    \CC_r^\nu(x)\cap K&=&\cup_{k=1}^2\,\big\{y+u_k(y)\,\nu:y\in\DD_r^\nu(x)\big\}\,,
    \\
    \CC_r^\nu(x)\cap\pa^*U_{j(x)}^i\cap\pa^*U_{k(x)}^i&=&\big\{y+u_1(y)\,\nu:y\in\DD_r^\nu(x)\cap\{u_1=u_2\}\big\}\,,
    \\
    \CC_r^\nu(x)\cap\pa^*E&=&\cup_{k=1}^2\,\big\{y+u_k(y)\,\nu:y\in\DD_r^\nu(x)\cap\{u_1<u_2\}\big\}\,.
    \end{eqnarray*}
    This completes the proof of conclusion (iv)$_\a$.
    %
    \medskip

    \noindent{\it Proof of the claim}: Assuming without loss of generality that $x=0$, we want to find $\Lambda$ and $r_0$ positive such that if $B_{2r}\cc \Omega$, then, denoting by $\{U_j\}_j$ the open connected components of $B_{2r}\setminus (E \cup K)$, we have
    \begin{eqnarray}
      \label{K char2}
      &&B_r\cap K = B_r\cap\cup_j\pa U_j\,,
      \\
      \label{finitely many in BR 2}
      &&\#\big\{j:B_r\cap U_j\ne\varnothing\big\}<\infty\,,
      \\
      \label{normalization 2}
      &&B_{2\,r}\cap \cl (\pa^* U_j) =B_{2\,r}\cap \pa U_j\,,
    \end{eqnarray}
    and that $P(U_j;B_r)\leq P(V_j;B_r) + \Lambda\,|U_j\Delta V_j|$ whenever $V_j$ satisfies $V_j\Delta U_j \cc B_r$ and $\diam(U_j\Delta V_j)<r_0$.

    \medskip

    \noindent {\it Step one}: We prove that
    \begin{eqnarray}
    \label{K doesnt see them and theyre equivalent}
    K \cap \mathrm{int}\, U_j^\one = \varnothing\,,\qquad \mathrm{int}\, U_j^\one = U_j\quad \forall j\,.
    \end{eqnarray}
    To this end, we begin by noticing that, for every $j$,
    \begin{eqnarray}
    \label{one inclusion}
    B_{2\,r}\cap \pa U_j  &\subset& B_{2\,r}\cap K\,,
    \\
    \label{easy inclusions}
    U_j \,\,\subset\,\, {\rm int}(U_j^\one) &\subset& B_{2\,r} \cap \cl U_j \,\,\subset\,\, B_{2\,r}\cap (U_j \cup K)\,,
    \\
    \label{more easy inclusions}
    B_{2\,r}\cap \pa[{\rm int}(U_j^\one)]&\subset& B_{2\,r}\cap K\,.
    \end{eqnarray}
    Indeed, for every $k$ and $j$, $U_k\cap U_j=\varnothing$ with $U_k$ and $U_j$ open gives $U_k\cap\pa U_j=\varnothing$, so that $B_{2r}\cap\pa U_j\subset B_{2r}\setminus\cup_kU_k=B_{2\,r}\cap(E\cup K)=B_{2\,r}\cap K$ thanks to the fact that $E\cap\pa U_j=\varnothing$ (as $U_j\cap E=\varnothing$). Having proved \eqref{one inclusion}, one easily deduces the third inclusion in \eqref{easy inclusions}, while the first two are evident. Finally, from \eqref{easy inclusions}, and since $K$ is closed, we find
    \[
    B_{2\,r}\cap\cl\big({\rm int}(U_j^\one)\big)\subset B_{2\,r}\cap(\cl(U_j)\cup K)\,,
    \]
    so that subtracting ${\rm int}(U_j^\one)$, and recalling that $U_j\subset {\rm int}(U_j^\one)$ we find
    \[
    B_{2\,r}\cap \pa[{\rm int}(U_j^\one)]\subset B_{2\,r}\cap(K\cup\pa U_j)
    \]
    and deduce \eqref{more easy inclusions} from \eqref{one inclusion}.

    \medskip

    Next, we claim that,
    \begin{equation}
    \label{kprime E}
    \mbox{if $K_* = K \setminus \bigcup_j \intci$, then $(K_*,E)\in\KK$ and $K_*\cup E$ is $\C$-spanning}\,.
    \end{equation}
    {\it To prove that $(K_*,E) \in \mathcal{K}$}, the only assertion that is not immediate is the inclusion $\Omega \cap \pa E \subset K_*$. To prove it we notice that if $z\in \intci$, then $B_s(z)\subset \intci$ for some $s>0$, so that $U_j\cap E=\varnothing$ gives $|E\cap B_s(z)|=0$. Since $E$ is open this implies $B_s(z) \cap E = \varnothing$, hence $z\notin \pa E$.

    \medskip

    \noindent{\it To prove that $E \cup K_*$ is $\mathcal{C}$-spanning}: Since $E\cup K_*$ is relatively closed in $\Omega$, it suffices to verify that for arbitrary $\gamma \in \mathcal{C}$, $(K_* \cup E) \cap \gamma \neq \varnothing$. Since $K \setminus B_{2r} = K_* \setminus B_{2r}$, we directly assume that $(K \cup E) \cap (\gamma \setminus B_{2r})=\varnothing$. Since $K \cup E$ is $\mathcal{C}$-spanning $\wire$, by Lemma \ref{dldrg lemma}, there are two distinct connected components $U_j$ and $U_k$ of $B_{2r}\setminus(K\cup E)$ such that there is  $\gamma(\SS^1)\cap B_{2\,r}\cap (\pa U_j)\cap(\pa U_k)\ne\varnothing$. We conclude by showing that
    \begin{equation}
      \label{southwest}
      B_{2\,r}\cap(\pa U_j)\cap(\pa U_k)\subset K_*\,,\qquad\forall j\ne k\,.
    \end{equation}
    Indeed any point in $B_{2r}\cap(\pa U_j)\cap(\pa U_k)$ is an accumulation point for both $U_j$ and $U_k$, and thus, by \eqref{easy inclusions}, for both ${\rm int}U_j^\one$ and ${\rm int}U_k^\one$. Since $U_j\cap U_k=\emptyset$ implies $({\rm int}U_j^\one)\cap({\rm int}U_k^\one)=\emptyset$, an accumulation point for both ${\rm int}U_j^\one$ and ${\rm int}U_k^\one$ must lie in $[\pa({\rm int}U_j^\one)]\cap[\pa({\rm int}U_k^\one)]$. We thus deduce \eqref{southwest} from \eqref{more easy inclusions}, and complete the proof of \eqref{kprime E}.

    \medskip

    \noindent{\it To deduce \eqref{K doesnt see them and theyre equivalent} from \eqref{kprime E}, and complete step one}: By \eqref{kprime E}, $(K_*,E)$ is admissible in $\Psi_{\rm bk}(v)$. Since $(K,E)$ is a minimizer of $\Psi_{\rm bk}(v)$, we conclude that $\H^n(K\setminus K_*)=0$. Would there be $z\in{\rm int}(U_j^\one)\cap K$ for some $j$, then by \eqref{lower density estimates proof}, and with $\rho>0$ such that $B_\rho(z)\subset {\rm int}(U_j^\one)$, we would find
    \[
    c\,\rho^n\le\H^n(K\cap B_\rho(z))\le\H^n(K\cap {\rm int}(U_j^\one))\le\H^n(K\setminus K_*)=0\,.
    \]
    This shows that $K\cap{\rm int}(U_j^\one)=\varnothing$. Using this last fact in combination with  ${\rm int}(U_j^\one)\subset B_{2\,r}\cap (U_j\cap K)$ from \eqref{easy inclusions} we conclude that ${\rm int}(U_j^\one)\subset U_j$, and thus that ${\rm int}(U_j^\one)= U_j$ by the first inclusion in \eqref{easy inclusions}.

    \medskip

    \noindent{\it Step two}: We prove \eqref{normalization 2}, i.e. $B_{2\,r}\cap\cl(\pa^*U_j)=B_{2\,r}\cap\pa U_j$. The $\subset$ inclusion is a general fact, see \eqref{sofp1}. To prove the reverse inclusion we recall, again from \eqref{sofp1}, that $z\in B_{2\,r}\cap\cl(\pa^*U_j)$ if and only if $0<|B_\rho(z)\cap U_j|<|B_\rho|$ for every $\rho>0$. Now, if $z\in B_{2\,r}\cap\pa U_j$, then clearly, being $U_j$ open, we have $|U_j\cap B_\rho(z)|>0$ for every $\rho>0$; moreover, should $|B_\rho(z)\cap U_j|=|B_\rho|$ hold for some $\rho$, then we would have $z\in{\rm int}(U_j^\one)$, and thus $z\in U_j$ by \eqref{K doesnt see them and theyre equivalent}, a contradiction.

    \medskip

    \noindent{\it Step three}: We prove, for each $j$, the $\H^n$-equivalence of $\pa^*U_j$ and $\pa U_j$, that is
    \begin{align}\label{boundaries agree up to null}
    \H^n(B_{2\,r}\cap\pa U_j \setminus \pa^* U_j)=0\,.
    \end{align}
    By a standard argument \cite[Theorem 21.11]{maggiBOOK} it will suffice to prove the existence of $r_0>0$ and $\a,\b\in(0,1/2)$ (depending on $n$) such that, for each $j$ and each $z\in B_{2\,r}\cap \pa U_j$, it holds
    \begin{align}\label{upper lower volume bounds}
    \alpha\, |B_\rho| \leq |B_\rho(z) \cap U_j | \leq (1-\b) |B_\rho|\,,
    \end{align}
    for every $\rho<\min\{r_0,\dist(z,\pa B_{2\,r})\}$.

    \medskip

    \noindent{\it Proof of the lower bound in \eqref{upper lower volume bounds}}: Since diffeomorphic images of $\C$-spanning sets are $\C$-spanning, a standard argument using diffeomorphic volume fixing variations shows the existence of positive constants $\Lambda$ and $r_0$ such that if $(K',E') \in \mathcal{K}_{\rm B}$, $K' \cup (E')^\one$ is $\mathcal{C}$-spanning $\wire$, and $(K'\Delta K) \cup (E'\Delta E)\cc B_\rho(z)$ for some $\rho<r_0$ and $B_\rho(z)\cc B_{2\,r}$, then
    \begin{align}\label{lambda minimality ineq for F}
        \mathcal{F}_{\rm bk}(K,E) \leq \mathcal{F}_{\rm bk}(K',E') + \Lambda\,|E \Delta E'|\,.
    \end{align}
    We claim that we can apply \eqref{lambda minimality ineq for F} with
    \begin{align}\label{Eprime Kprime}
     E' = E \cup \big(B_\rho(z) \cap \cl U_j\big)\,,\quad K' = \big(K \cup (U_j^\one \cap \pa B_\rho(z)\big) \setminus (E')^\one\,,
    \end{align}
    where $\rho<r_0$, $B_\rho(z)\cc B_{2\,r}$, and
    \begin{equation}
      \label{choice of rho}
      \H^n\big(\pa B_\rho(z)\cap[\pa^*E\cup\pa^*U_j]\big)=\mathcal{H}^n(K \cap \pa B_\rho(z))=0\,.
    \end{equation}
    Indeed, $K'\cup (E')^\one$ contains $K\cup E^\one$, thus $K\cup E$ being $E$ open, and is thus $\C$-spanning. To check that $(K',E')\in\KK_{\rm B}$, we argue as follows. First, we notice that $\H^n(\{\nu_E=\nu_{B_\rho(z)\cap\cl(U_j)}\})=0$, since it is $\H^n$-contained in the union of $\pa B_\rho(z)\cap\pa^*E$ and $\{\nu_E=\nu_{\cl(U_j)}\}$, that are $\H^n$-negligible by \eqref{choice of rho} and by the fact that $\nu_E=-\nu_{\cl(U_j)}$ $\H^n$-a.e. on $\pa^*E\cap\pa^*\cl(U_j)$ thanks to $|E\cap\cl(U_j)|=0$. By $\H^n(\{\nu_E=\nu_{B_\rho(z)\cap\cl(U_j)}\})=0$ and \eqref{E cup F} we thus have
    \begin{equation}
      \label{t1}
      \Om\cap\pa^*E'\ehn\Om\cap\big\{\big[E^\zero \cap \pa^*\big(B_\rho(z) \cap \cl U_j\big)\big]\cup \big[\big(B_\rho(z) \cap \cl U_j\big)^\zero\cap \pa^*E\big]\big\}\,.
    \end{equation}
    Since $U_j$ is Lebesgue equivalent to $\cl(U_j)$ (indeed, $B_{2\,r}\cap\pa U_j\subset K$), we have $U_j^\one=[\cl(U_j)]^\one$ and $\pa^*[\cl(U_j)]=\pa^* U_j$, so that \eqref{E cap F}  and \eqref{choice of rho} give
    \begin{eqnarray}\nonumber
    &&\pa^*\big(B_\rho(z)\cap\cl(U_j)\big)\ehn\big\{[\cl(U_j)]^\one\cap\pa B_\rho(z)\big\}\cup \big\{B_\rho(x)\cap\pa^*[\cl(U_j)]\big\}\,,
      \\\label{t half}
      &&=\big(U_j^\one\cap\pa B_\rho(z)\big)\cup \big(B_\rho(x)\cap\pa^* U_j\big)\subset \big(U_j^\one\cap\pa B_\rho(z)\big)\cup K\,,
    \end{eqnarray}
    by $B_{2\,r}\cap\pa U_j\subset K$. By \eqref{t1} and $\H^n((E')^\one\cap\pa^*E')=0$ we thus find that
    \begin{equation}
      \label{t2}
      \Om\cap\pa^*E'\cap\pa^*\big(B_\rho(z)\cap\cl(U_j)\big)\shn K'\,.
    \end{equation}
    Moreover, by $\Om\cap\pa^*E\subset\Om\cap\pa E\subset K$ and
    \[
    (\pa^*E)\cap\big(B_\rho(z) \cap \cl U_j\big)^\zero\subset E^\half\cap \big(B_\rho(z) \cap \cl U_j\big)^\zero\subset \R^{n+1}\setminus (E')^\one\,,
    \]
    we find $(\pa^*E)\cap\big(B_\rho(z) \cap \cl U_j\big)^\zero\subset K\setminus (E')^\one\subset K'$, which combined with \eqref{t2} finally proves the $\H^n$-containment of $\Om\cap\pa^*E'$ into $K'$, and thus $(K',E')\in\KK_{\rm B}$. We have thus proved that $(K',E')$ as in \eqref{Eprime Kprime} is admissible into \eqref{lambda minimality ineq for F}. Since $\F_{\rm bk}(K,E;\pa B_\rho(z))=0$ by \eqref{choice of rho} and $\F_{\rm bk}(K,E;A)=\F_{\rm bk}(K',E';A)$ if $A=\Om\setminus\cl(B_\rho(z))$, we deduce from \eqref{lambda minimality ineq for F} that
    \begin{align}\label{lambda minimality ineq for F2}
        \mathcal{F}_{\rm bk}(K,E;B_\rho(z)) \leq \mathcal{F}_{\rm bk}(K',E';\cl(B_\rho(z))) + \Lambda\,|E \Delta E'|\,.
    \end{align}
    To exploit \eqref{lambda minimality ineq for F2}, we first notice that $\{B_\rho(z)\cap U_k\}_k$ is a Lebesgue partition of $B_\rho(z)\setminus E$ with $B_\rho(z)^\one\cap\pa^*(B_\rho(z)\cap U_k)=B_\rho(z)\cap\pa^*U_k$ for every $k$, so that, by Lemma \ref{lemma breakdown of F via open components},
    \begin{equation}
      \label{t3}
      \F_{\rm bk}(K,E;B_\rho(z))=2\,\H^n\Big(B_\rho(z)\cap E^\zero\cap \Big(K\setminus\bigcup_k\pa^*U_k\Big)\Big)+\sum_kP(U_k;B_\rho(z))\,.
    \end{equation}
    Similarly, $\{B_\rho(z)\cap U_k\}_{k\ne j}$ is a Lebesgue partition of $B_\rho(z)\setminus E'$, so that again by Lemma \ref{lemma breakdown of F via open components} we find
    \begin{eqnarray}\nonumber
      &&\F_{\rm bk}(K',E';B_\rho(z))=2\,\H^n\Big(B_\rho(z)\cap (E')^\zero\cap \Big(K'\setminus\bigcup_{k\ne j}\pa^*U_k\Big)\Big)+\sum_{k\ne j}P(U_k;B_\rho(z))
      \\\label{t4}
      &&=
      2\,\H^n\Big(B_\rho(z)\cap (E')^\zero\cap \Big(K\setminus\bigcup_k\pa^*U_k\Big)\Big)+\sum_{k\ne j}P(U_k;B_\rho(z))
    \end{eqnarray}
    where in the last identity we have used that, by \eqref{Eprime Kprime}, we have $B_\rho(z)\cap(E')^\zero\cap\pa^*U_j=0$ and $B_\rho(z)\cap K' \cap (E')^\zero = B_\rho(z)\cap K\cap (E')^\zero$. Combining \eqref{lambda minimality ineq for F2}, \eqref{t3}, \eqref{t4} and the fact that $(E')^\zero\subset E^\zero$, we find that
    \begin{eqnarray}\label{t5}
      P(U_j;B_\rho(z))\le \mathcal{F}_{\rm bk}\big(K',E';\pa B_\rho(z)\big)+ \Lambda\,|B_\rho(z)\cap U_j|\,.
    \end{eqnarray}
    The first term in $\mathcal{F}_{\rm bk}\big(K',E';\pa B_\rho(z)\big)$ is $P(E';\pa B_\rho(z))$: taking into account $\H^n(\pa^*E\cap\pa B_\rho(z))=0$, by \eqref{t1} and the second identity in \eqref{t half} we find
    \begin{eqnarray*}
     P(E';\pa B_\rho(z))&=&\H^n\big(\pa B_\rho(z)\cap E^\zero \cap \pa^*\big(B_\rho(z) \cap \cl U_j\big)\big)
     \\
     &=&\H^n(E^\zero\cap U_j^\one\cap\pa B_\rho(z))=\H^n(U_j^\one\cap\pa B_\rho(z))\,,
    \end{eqnarray*}
    while for the second term in $\mathcal{F}_{\rm bk}\big(K',E';\pa B_\rho(z)\big)$, by $\H^n(K\cap\pa B_\rho(z))=0$,
    \[
    \H^n(K'\cap (E')^\zero\cap\pa B_\rho(z))=\H^n((E')^\zero\cap U_j^\one\cap\pa B_\rho(z))=0
    \]
    since $(E')^\zero\subset(B_\rho(z)\cap\cl(U_j))^\zero$ and $B_\rho(z)\cap\cl(U_j)$ has positive Lebesgue density at points in $U_j^\one\cap\pa B_\rho(z)$. Having thus proved that $\mathcal{F}_{\rm bk}\big(K',E';\pa B_\rho(z)\big)=\H^n(U_j^\one\cap\pa B_\rho(z))$, we conclude from \eqref{t5} that
    \[
    P(U_j;B_\rho(z))\le \H^n(U_j^\one\cap\pa B_\rho(z))+ \Lambda\,|B_\rho(z)\cap U_j|\,,
    \]
    for a.e. $\rho<r_0$. Since $z\in B_{2\,r}\cap\pa U_j=B_{2\,r}\cap\cl(\pa^*U_j)$ and \eqref{sofp1} imply that $|B_\rho(z)\cap U_j|>0$ for every $\rho>0$, a standard argument (see, e.g. \cite[Theorem 21.11]{maggiBOOK}) implies that, up to further decrease the value of $r_0$ depending on $\Lambda$, and for some constant $\a=\a(n)\in(0,1/2)$, the lower bound in \eqref{upper lower volume bounds} holds true.

    \medskip

    \noindent{\it Proof of the upper bound in \eqref{upper lower volume bounds}}: We argue by contradiction that, no matter how small $\beta\in(0,1/2)$ is, we can find $j$, $z\in B_{2\,r}\cap\pa U_j$, and $\rho<\min\{r_0,\dist(z,\pa B_{2\,r})\}$, such that
    \begin{align}\label{contradiction upper density}
        |B_\rho(z) \cap U_j| > (1-\b)\,|B_\rho|\,.
    \end{align}
    We first notice that for every $k\ne j$ it must be $B_{\rho/2}(z)\cap\pa U_k=\varnothing$: indeed if $w\in B_{\rho/2}(z)\cap\pa U_k$ for some $k\ne j$, then by the lower bound in \eqref{upper density estimates proof} and by \eqref{contradiction upper density} we find
    \[
    \a\,|B_{\rho/2}|\le|U_k\cap B_{\rho/2}(w)|\le |B_\rho(z)\setminus U_j|<\b\,|B_\rho|
    \]
    which gives a contradiction if $\b<\a/2^{n+1}$. By $B_{\rho/2}(z)\cap\pa U_k=\varnothing$ it follows that
    \begin{equation}
      \label{theball}
      B_{\rho/2}(z)\subset\cl(U_j)\cup\cl(E)\,.
    \end{equation}
    Let us now set
    \begin{equation}
      \label{nuovo comp}
      E'=E\setminus B_{\rho/2}(z)\,,\qquad K'=\big(K\setminus B_{\rho/2}(z)\big)\cup\big(E^\one\cap\pa B_{\rho/2}(z)\big)\,.
    \end{equation}
    By \eqref{E minus F}, if $\H^n(\pa^*E\cap\pa B_{\rho/2})=0$, then $(K',E')\in\KK$, since $(\Om\setminus B_{\rho/2}(z))\cap\pa^*E\subset K\setminus B_{\rho/2}(z)\subset K'$ implies
    \[
    \Om\cap\pa^*E'\ehn\Om\cap\big\{\big((\pa^*E)\setminus B_{\rho/2}(z)\big)\cup\big(E^\one\cap\pa B_{\rho/2}(z)\big)\big\}\subset K'\,.
    \]
    Moreover $K'\cup(E^\one)'$ is $\C$-spanning $\wire$ since it contains $(K\cup E)\setminus B_{\rho/2}(z)$, and
    \begin{equation}\label{uv}
    \mbox{$(K \cup E)\setminus B_{\rho/2}(z)$ is $\mathcal{C}$-spanning $\wire$}\,.
    \end{equation}
    Indeed, if $\g\in\C$ and $\g(\SS^1)\cap (K\cup E)\setminus B_{\rho/2}(z)=\emptyset$, then by applying Lemma \ref{dldrg lemma} to $S=K\cup E$ and $B=B_{2\,r}$ we see that either $\g(\SS^1)\cap (K\cup E)\setminus B_{2\,r}\ne\varnothing$ (and thus $\g(\SS^1)\cap(K \cup E)\setminus B_{\rho/2}(z)\ne\varnothing$ by $B_{\rho/2}(z)\subset B_r$), or there are $k\ne h$ such that $\g(\SS^1)\cap \pa U_k\ne\varnothing$ and $\g(\SS^1)\cap \pa U_h\ne\varnothing$. Up to possibly switch $k$ and $h$, we have that $k\ne j$, so that \eqref{theball} implies that $\varnothing\ne\g(\SS^1)\cap \pa U_k=\g(\SS^1)\cap \pa U_k\setminus B_{\rho/2}(z)$, where the latter set is contained in $K\setminus B_{\rho/2}(z)$ by \eqref{K char2} and $B_{\rho/2}(z)\subset B_r$. This proves \eqref{uv}.

    \medskip

    We can thus plug the competitor $(K',E')$ defined in \eqref{nuovo comp} into \eqref{lambda minimality ineq for F2}, and find
    \[
    \F_{\rm bk}(K,E;B_{\rho/2}(z))\le\F_{\rm bk}\big(K',E';\cl(B_{\rho/2}(z))\big)+\Lambda\,|E\cap B_{\rho/2}(z)|\,,
    \]
    for every $\rho<\min\{r_0,\dist(z,\pa B_{2\,r})\}$ such that $\H^n(K\cap\pa B_{\rho/2}(z))=0$. Now, by Lemma \ref{lemma breakdown of F via open components} and by \eqref{theball} we have
    \[
    \F_{\rm bk}(K,E;B_{\rho/2}(z))\ge P(U_j;B_{\rho/2}(z))=P(E;B_{\rho/2}(z))\,,
    \]
    while \eqref{E cap F} gives
    \[
    \cl(B_{\rho/2}/z)\cap K'\ehn\cl(B_{\rho/2}/z)\cap\pa^*E'\ehn E^\one\cap\pa B_{\rho/2}(z)\,,
    \]
    thus proving that, for a.e. $\rho<\min\{r_0,\dist(z,\pa B_{2\,r})\}$,
    \[
    P(E;B_{\rho/2}(z))\le\H^n(E^\one\cap B_{\rho/2}(z))+\Lambda\,|E\cap B_{\rho/2}(z)|\,.
    \]
    Since $z\in B_{2\,r}\cap\pa U_j$ and $B_{\rho/2}(z)\cap\pa^*U_j=B_{\rho/2}(z)\cap\pa^*E$, by \eqref{sofp1} we see that $|E\cap B_{\rho/2}(z)|>0$ for every $\rho<\min\{r_0,\dist(z,\pa B_{2\,r})\}$. By a standard argument, up to further decrease the value of $r_0$, we find that for some $\a'=\a'(n)$ it holds
    \[
    |E\cap B_{\rho/2}(z)|\ge \a'\,|B_{\rho/2}|\,,\qquad\forall  \rho<\min\{r_0,\dist(z,\pa B_{2\,r})\}\,,
    \]
    and since $|E\cap B_{\rho/2}(z)|=|B_{\rho/2}(z)\setminus U_j|$ this give a contradiction with \eqref{contradiction upper density} up to further decrease the value of $\beta$.

    \medskip

    \noindent{\it Step three}: We prove \eqref{K char2} and \eqref{finitely many in BR 2}. The lower bound in \eqref{upper lower volume bounds} implies \eqref{finitely many in BR 2}, i.e., $J=\#\{j:U_j\cap B_r\ne\varnothing\}<\infty$. Next, by $B_{2\,r}\cap \pa U_j\subset K$ (last inclusion in \eqref{easy inclusions}), to prove \eqref{K char2} it suffices to show that
    \begin{equation}\label{two inclusions}
        K\cap B_r\subset \cup_{j=1}^J \pa U_j\,.
    \end{equation}
    Now, if $z\in K \cap B_r$, then by $K\cap E=\varnothing$ we have either $z\in K\setminus\cl(E)$ or $z\in B_r\cap\pa E$, and, in the latter case, $|E\cap B_\rho(z)|\le(1-c)\,|B_\rho|$ for every $\rho<\min\{r_0,\dist(z,\pa\wire)\}$ thanks to \eqref{upper density estimates proof}. Therefore, in both cases, $z$ is an accumulation point for $(\cup_{j=1}^J U_j)^\one \cap B_{r}$. Since $J$ is finite, there must be at least one $j$ such that $z\in\cl(U_j)$ -- hence $z\in\pa U_j$ thanks to $K\cap U_j=\varnothing$.

    \medskip

    Before moving to the next step, we also notice that
    \begin{equation}
      \label{energy is perimeters!}
      \mathcal{F}_{\rm bk}(K,E;B_r) =\sum_{j=1}^J P(U_j;B_r)\,.
    \end{equation}
    Indeed, by \eqref{K char2}, \eqref{finitely many in BR 2}, and \eqref{boundaries agree up to null} we have
    \begin{equation}
    \label{K is reduced boundaries}
        K \cap B_r= B_r\cap \cup_{j=1}^J \pa U_j \ehn B_r\cap\cup_{j=1}^J \pa^* U_j\,,
    \end{equation}
    so that, in the application of Lemma \ref{lemma breakdown of F via open components}, i.e. in \eqref{t3}, the multiplicity $2$ terms vanishes, and we find \eqref{energy is perimeters!}.

    \medskip

    \noindent{\it Step four}: In this step we consider a set of finite perimeter $V_1$ such that, for some $B:=B_\rho(z)\subset B_{r}$ with $\rho<r_0$ and $\H^n(K\cap\pa B)=0$, we have
    \begin{eqnarray}\label{condition 2 on perturbations}
         U_1 \Delta V_1 \cc B\,.
    \end{eqnarray}
    We then define a pair of Borel sets $(K',E')$ as
    \begin{eqnarray}\label{good competitors 1}
      E'&=&\big(E\setminus B\big)\,\cup\,\big[B\cap\big(V_1\Delta (E\cup U_1)\big)\big]\,,
      \\ \label{good competitors 2}
      K'&=&\big(K\setminus B\big)\,\cup\, \big[B\cap\big(\pa^*V_1\cup\pa^*U_2\cup\cdots\cup\pa^*U_J\big)\big]\,,
    \end{eqnarray}
    and show that $(K',E')\in\KK_{\rm B}$, $K'\cup (E')^\one$ is $\C$-spanning $\wire$, and
    \begin{equation}
      \label{the difference}
      \F_{\rm bk}(K',E')-\F_{\rm bk}(K,E)\le P(V_1;B)-P(U_1;B)\,.
    \end{equation}
    As a consequence of \eqref{the difference}, \eqref{lambda minimality ineq for F} and $|E\Delta E'|=|U_1\Delta V_1|$, we find of course that $P(U_1;\Om)\le P(V_1;\Om)+\Lambda\,|U_1\Delta V_1|$, thus showing that $U_1$ is a $(\Lambda,r_0)$-perimeter minimizer in $\Om$.

    \medskip

    Proving that $(K',E')\in\KK_{\rm B}$ is immediately reduced to showing that $B\cap\pa^*E'$ is $\H^n$-contained in $B\cap(\pa^*V_1\cup\pa^*U_2\cup\cdots\cup\pa^*U_J)$ thanks to $\H^n(K\cap\pa B)=0$. Now, on taking into account that, by \eqref{E cup F} and \eqref{E minus F}, $\pa^*(X\cup Y)$ and $\pa^*(X\setminus Y)$ are both $\H^n$-contained in $(\pa^*X)\cup(\pa^*Y)$, and thus $\pa^*(X\Delta Y)$ is too, we easily see that
    \[
    B\cap\pa^*E'=B\cap\pa^*[V_1\Delta(E\cup U_1)]\shn(B\cap\pa^*V_1)\cup(B\cap\pa^*(E\cup U_1))\,.
    \]
    However, $B\cap(E\cup U_1)=B\setminus(\cup_{j=2}^JU_j)$, so that $\pa^*X=\pa^*(\R^{n+1}\setminus X)$ gives
    \[
    B\cap\pa^*(E\cup U_1)=B\cap\pa^*(\cup_{j=2}^JU_j)\shn B\cap\cup_{j\ge2}\pa^*U_j\,,
    \]
    where we have used again the $\H^n$-containment of $\pa^*(X\cup Y)$ in $(\pa^*X)\cup(\pa^*Y)$. This proves that $(K',E')\in\KK_{\rm B}$.

    \medskip

    To prove that $K'\cup(E')^\one$ is $\C$-spanning $\wire$, we show that the set $S$ defined by
    \[
    S=\big((K\cup E)\setminus B\big)\cup\big(\cl(B)\cap\cup_{j\ge 2}\pa U_j\big)\,,
    \]
    is $\H^n$-contained in $K'\cup(E')^\one$ and is $\C$-spanning $\wire$.

    \medskip

    To prove that $S$ is $\H^n$-contained in $K'\cup(E')^\one$, we start by noticing that $(K\cup E)\setminus\cl(B)$ is $\H^n$-equivalent to $(K\cup E^\one\cup \pa^*E)\setminus\cl(B)\subset K\cup E^\one$ (by $(K,E)\in\KK_{\rm B}$), whereas $|(E\Delta E')\setminus B|=0$ implies $(E^\one\Delta (E')^\one)\setminus\cl(B)=\varnothing$: hence $S\setminus\cl(B)$ if $\H^n$-contained in $K'\cup(E')^\one$. Next, by \eqref{boundaries agree up to null} and by definition of $K'$,
    \[
    S\cap B= B\cap \cup_{j\ge 2}\pa U_j\ehn B\cap\cup_{j\ge 2}\pa^* U_j\subset K'\,.
    \]
    Finally, by $\H^n(K\cap\pa B)=0$, \eqref{one inclusion}, and Federer's theorem, $(S\cap\pa B)\setminus K$ is $\H^n$-equivalent to $(E^\one\cap\pa B)\setminus K$, where $E^\one\cap A=(E')^\one\cap A$ in an open neighborhood $A$ of $\pa B$ thanks to $U_1\Delta V_1\cc B$.

    \medskip

    To prove that $S$ is $\C$-spanning $\wire$, since $S$ is relatively closed in $\Om$ and thanks to Theorem \ref{theorem definitions equivalence}, we only need to check that $S\cap\g(\SS^1)\ne\varnothing$ for every $\g\in\C$. Since $(K\cup E)\cap\g(\SS^1)\ne\varnothing$ for every $\g\in\C$, this is immediate unless $\g$ is such that $S\cap\g(\SS^1)\setminus B=\varnothing$; in that case, however, Lemma \ref{dldrg lemma} implies the existence of $j\ne k$ such that $\g(\SS^1)\cap B\cap\pa U_j$ and $\g(\SS^1)\cap B\cap\pa U_k$ are both non-empty. Since either $j\ge 2$ or $k\ge 2$, we conclude by \eqref{one inclusion} that $\g(\SS^1)\cap B\cap K'\ne\varnothing$, thus completing the proof.

    \medskip

    We are thus left to prove the validity of \eqref{the difference}. Keeping \eqref{energy is perimeters!} and $\F_{\rm bk}(K',E';B)\le\F_{\rm bd}(K',E';B)$ into account, this amounts to showing that
    \begin{equation}
      \label{left to lambda}
      \F_{\rm bd}(K',E';B)=\H^n(B\cap\pa^*E')+2\,\H^n\big(B\cap K'\setminus\pa^*E'\big)= P(V_1;B)+\sum_{j=2}^J P(U_j;B)\,.
    \end{equation}
    To this end we notice that by \eqref{E delta F} and $B\cap E'=B\cap[V_1\Delta(E\cup U_1)]$ we have
    \begin{eqnarray*}
    B\cap\pa^*E'&\ehn& B\cap\big\{\pa^*V_1\cup\pa^*(E\cup U_1)\big\}
    \\
    &\ehn&B\cap\big\{(\pa^*V_1)\,\cup\,(U_1^\zero\cap\pa^*E)\,\cup\,(E^\zero\cap\pa^*U_1)\big\}\,,
    \end{eqnarray*}
    where we have used \eqref{E cup F} and $\H^n(\{\nu_E=\nu_{U_1}\})=0$ (as $E\cap U_1=\varnothing$). By \eqref{caccioppoli partitions} and \eqref{caccioppoli exactly two}, since $\{B\cap E,B\cap U_j\}_{j=1}^N$ is a Caccioppoli partition of $B$, we have
    \[
    U_1^\zero\cap\pa^*E=(\pa^*E)\cap\bigcup_{j\ge 2}(\pa^*U_j)\,,
    \qquad
    E^\zero\cap\pa^*U_1=(\pa^*U_1)\cap\bigcup_{j\ge 2}(\pa^*U_j)\,,
    \]
    so that
    \begin{eqnarray*}
    B\cap\pa^*E'&\ehn&B\cap\Big\{(\pa^*V_1)\cup \Big(\big[(\pa^*E)\cup(\pa^*U_1)\big]\cap\bigcup_{j\ge 2}(\pa^*U_j)\Big)\Big\}\,,
    \\
    B\cap(K'\setminus\pa^*E')&\ehn& B\cap\Big(\bigcup_{j\ge 2}\pa^*U_j\Big)\setminus\big[(\pa^*E)\cup(\pa^*U_1)\big]\,.
    \end{eqnarray*}
    We thus find
    \begin{eqnarray*}
      &&\H^n(B\cap\pa^*E)+2\,\H^n(B\cap(K'\setminus\pa^*E'))
      \\
      &&=P(V_1;B)+2\,\H^n\Big(\Big(\bigcup_{j\ge 2}\pa^*U_j\Big)\setminus(\pa^*E\cup\pa^*U_1)\Big)
      +\H^n\Big(\Big(\bigcup_{j\ge 2}\pa^*U_j\Big)\cap(\pa^*E\cup\pa^*U_1)\Big)
      \\
      &&=P(V_1;B)+\sum_{j\ge 2}P(U_j;B)\,,
    \end{eqnarray*}
    that is \eqref{left to lambda}.

    \medskip

    \noindent {\it Step five}: In this final step we prove conclusions (iv) and (v). To this end we fix $x\in[\Om\cap(\pa E\setminus\pa^*E)]\setminus\Sigma$, and recall that, by conclusion (iv)$_\a$, there are $r>0$, $\nu\in\SS^n$, $u_1,u_2\in C^{1,\a}(\DD_r^\nu(x);(-r/4,r/4))$ ($\a\in(0,1/2)$ arbitrary) such that $u_1(x)=u_2(x)=0$, $u_1\le u_2$ on $\DD_r^\nu(x)$, $\{u_1<u_2\}$ and ${\rm int}\{u_1=u_2\}$ are both non-empty, and
    \begin{eqnarray}\label{grafico K proof 2}
    \CC_r^\nu(x)\cap K&=&\cup_{i=1,2}\big\{y+u_i(y)\,\nu:y\in\DD_r^\nu(x)\big\}\,,
    \\\label{grafico pa E proof 2}
    \CC_r^\nu(x)\cap \pa^*E&=&\cup_{i=1,2}\big\{y+u_i(y)\nu:y\in\{u_1<u_2\}\big\}\,,
    \\\label{grafico E proof 2}
    \CC_r^\nu(x)\cap E&=&\big\{y+t\,\nu:y\in\{u_1<u_2\}\,,u_1(x)<t<u_2(x)\big\}\,.
    \end{eqnarray}
    We claim that $(u_1,u_2)$ has the minimality property
    \begin{equation}\label{FB prob}
      \A(u_1,u_2)\le \A(w_1,w_2):=\int_{\DD_r^\nu(x)}\sqrt{1+|\nabla w_1|^2}+\sqrt{1+|\nabla w_2|^2}\,,
    \end{equation}
    among all pairs $(w_1,w_2)$ with $w_1,w_2\in{\rm Lip}(\DD_r^\nu(x);(-r/2,r/2))$ that satisfy
    \begin{equation}
      \label{FB prob comp}
      \left\{\begin{split}
        &w_1\le w_2\,,\quad\mbox{on $\DD_r^\nu(x)$}\,,
        \\
        &w_k=u_k\,, \quad\mbox{on $\pa\DD_r^\nu(x)$, $k=1,2$}\,,
      \end{split}\right .
      \qquad \int_{\DD_r^\nu(x)}w_2-w_1=\int_{\DD_r^\nu(x)}u_2-u_1\,.
    \end{equation}
    Indeed, starting from a given a pair $(w_1,w_2)$ as in \eqref{FB prob comp}, we can define $(K'\cap\CC_r^\nu(x),E'\cap\CC_r^\nu(x))$ by replacing $(u_1,u_2)$ with $(w_1,w_2)$ in \eqref{grafico K proof 2} and \eqref{grafico E proof 2}, and then define $(K',E')\in\KK_{\rm B}$ by setting $K'\setminus\CC_r^\nu(x)=K\setminus\CC_r^\nu(x)$ and $E'\setminus\CC_r^\nu(x)=E\setminus\CC_r^\nu(x)$. Since $\pa\CC_r^\nu\setminus(K'\cup E')=\pa\CC_r^\nu\setminus(K\cup E)$ it is easily seen (by a simple modification of Lemma \ref{dldrg lemma} where balls are replaced by cylinders) that $(K',E')$ is $\C$-spanning $\wire$. Since $|E'|=|E|$, the minimality of $(K,E)$ in $\Psi_{\rm bk}(v)$ implies that $\F_{\rm bk}(K,E)\le\F_{\rm bk}(K',E')$, which readily translates into \eqref{FB prob}.

    \medskip

    Recalling that both $A_0={\rm int}\{u_1=u_2\}$ and $A_+=\{u_1<u_2\}$ are non-empty open subsets of $\DD_r^\nu(x)$, and denoting by ${\rm MS}(u)[\vphi]=\int_{\DD_r^\nu(x)}\nabla\vphi\cdot[(\nabla u)/\sqrt{1+|\nabla u|^2}]$ the distributional mean curvature operator, we find that
    \begin{eqnarray}\nonumber
    {\rm MS}(u_1)+{\rm MS}(u_2)=0\,,&&\qquad\mbox{on $\DD_r^\nu(x)$}\,,
    \\\nonumber
    {\rm MS}(u_k)=0\,,&&\qquad\mbox{on $A_0$ for each $k=1,2$}\,,
    \\\label{ms u2 u1 constant}
    {\rm MS}(u_2)=-{\rm MS}(u_1)=\l\,,&&\qquad\mbox{on $A_+$}\,,
    \end{eqnarray}
    for some constant $\l\in\R$; in particular, $u_1,u_2\in C^\infty(A_0)\cap C^\infty(A_+)$.
    We notice that it must be
    \begin{equation}
      \label{lambda is negative}
      \l<0\,.
    \end{equation}
    Indeed, arguing by contradiction, should it be that $\l\ge0$, then by \eqref{ms u2 u1 constant} we find ${\rm MS}(u_2)\ge0$ and ${\rm MS}(u_1)\le 0$ on $A_+$. Since $A_+$ is open an non-empty, there is an open ball $B\subset A_+$ such that $\pa B\cap\pa A_+=\{y_0\}$. Denoting by $x_0$ the center of $B$ and setting $\nu=(x_0-y_0)/|x_0-y_0|$, by $u_1\le u_2$, $u_1(y_0)=u_2(y_0)$ and $u_k\in C^1(\DD_r^\nu(x))$ we find that $\nabla u_1(y_0)=\nabla u_2(y_0)$. At the same time, by applying Hopf's lemma in $B$ at $y_0$, we see that since ${\rm MS}(u_2)\ge0$ and ${\rm MS}(u_1)\le 0$ on $B$, it must be $\nu\cdot\nabla u_2(y_0)<0$ and $\nu\cdot\nabla u_1(y_0)>0$, against $\nabla u_1(y_0)=\nabla u_2(y_0)$.

    \medskip

    By \eqref{ms u2 u1 constant}, \eqref{lambda is negative}, and $u_2\ge u_1$ on $\DD_r^\nu(x)$ we can apply the sharp regularity theory for the double membrane problem developed in \cite[Theorem 5.1]{silvestre} and deduce that $u_1,u_2\in C^{1,1}(\DD_r^\nu(x))$. Next we notice that, for every $\vphi\in C^\infty_c(A_+)$, and setting $u_+=u_2-u_1$,
    \[
    2\,\l\,\int_{A_+}\vphi={\rm MS}(u_2)[\vphi]-{\rm MS}(u_1)[\vphi]=
    \int_{A_+} {\rm A}(x)[\nabla u_+]\cdot \nabla\vphi\,,
    \]
    where we have set, with $f(z)=\sqrt{1+|z|^2}$,
    \[
    {\rm A}(x)=\int_0^1\,\nabla^2f\big(s\,\nabla u_2(x)+(1-s)\,\nabla u_1(x)\big)\,ds\,.
    \]
    In particular, $u_+\in C^{1,1}(\DD_r^\nu(x))$ is a non-negative distributional solution of
    \[
    \Div({\rm A}(x)\nabla u_+)=-2\,\l\,,\qquad\mbox{on $A_+$}\,,
    \]
    with a strictly positive right-hand side (by \eqref{lambda is negative}) and with ${\rm A}\in {\rm Lip}(A_+;\R^{n\times n}_{\rm sym})$ uniformly elliptic. We can thus apply the regularity theory for free boundaries developed in \cite[Theorem 1.1, Theorem 4.14]{FocardiGelliSpadaro} to deduce that
    \[
    {\rm FB}=\DD_r^\nu(x)\cap\pa\{u_+=0\}=\DD_r^\nu(x)\cap\pa\{u_2=u_1\}\,,
    \]
    can be partitioned into sets ${\rm Reg}$ and ${\rm Sing}$ such that ${\rm Reg}$ is relatively open in ${\rm FB}$ and such that for every $z\in{\rm Reg}$ there are $r>0$ and $\beta\in(0,1)$ such that $B_r(x)\cap{\rm FB}$ is a $C^{1,\beta}$-embedded $(n-1)$-dimensional manifold, and such that ${\rm Sing}=\cup_{k=0}^{n-1}{\rm Sing}_k$ is relatively closed in ${\rm FB}$, with each ${\rm Sing}_k$ locally $\H^k$-rectifiable in $\DD_r^\nu(x)$. Since, by \eqref{grafico pa E proof 2},
    \[
    \CC_r^\nu(x)\cap(\pa E\setminus\pa^*E)=\big\{y+u_1(y)\,\nu:y\in{\rm FB}\big\}
    \]
    and $u_1\in C^{1,1}(\DD_r^\nu(x))$, we conclude by a covering argument that $\Om\cap(\pa E\setminus\pa^*E)$ has all the required properties, and complete the proof of the theorem.
    \end{proof}

\section{Equilibrium across transition lines in wet foams (Theorem \ref{theorem foams})}\label{section foams proof}

\begin{proof}
  [Proof of Theorem \ref{theorem foams}] Let $\Omega\subset\R^{n+1}$ be open and let $(K_*,E_*)\in\KK_{\rm foam}$. We can find $(K,E)\in\KK$ such that $K$ is $\H^n$-equivalent to $K_*$, $E$ Lebesgue equivalent to $E_*$, and $K\cap E^\one=\varnothing$ by repeating with minor variations the considerations made in step one of the proof of Theorem \ref{thm existence EL for bulk section} (we do not have to worry about the $\C$-spanning condition, but have to keep track of the volume constraint imposed for each $U_i$, which can be done by using the volume-fixing variations for clusters from \cite[Part IV]{maggiBOOK}). In proving the regularity part of the statement, thanks to Theorem \ref{theorem decomposition}-(a) we can directly work with balls $B\cc\Om$ having radius less than $r_0$ (with $r_0$ as in \eqref{def of wet foams}), and consider the open connected components $\{U_i\}_i$ of $B$ induced by $K\cup E$. Using Lemma \ref{lemma breakdown of F via open components} and, again, volume-fixing variation techniques in place of the theory of homotopic spanning, we can proceed to prove analogous statement to \eqref{K char}, \eqref{finitely many in BR}, \eqref{normalization}, and \eqref{lambda minimality ineq for C}, thus proving the $(\Lambda,r_0)$-minimality of each $U_i$ in $B$. The claimed $C^{1,\a}$-regularity of each $U_i$ outside of a closed set $\Sigma$ with the claimed dimensional estimates follows then from De Giorgi's theory of perimeter minimality \cite{DeGiorgiREG,Tamaniniholder,maggiBOOK}.
\end{proof}

    \appendix

    \section{Equivalence of homotopic spanning conditions}\label{appendix equivalence of} In Theorem \ref{theorem definitions equivalence} we prove that, when $S$ is a closed set, the notion of ``$S$ is $\C$-spanning $\wire$'' introduced in Definition \ref{def homot span borel} boils down to the one in Definition \ref{def homot span closed}. We then show that the property of being $\C$-spanning is stable under reduction to the rectifiable part of a Borel set, see Lemma \ref{lemma rectfiable spanning}.

    \begin{theorem}\label{theorem definitions equivalence}
      Given a closed set $\wire\subset\R^{n+1}$, a spanning class $\C$ for $\wire$, and a set $S$ relatively closed in $\Om$, the following two properties are equivalent:

      \medskip

      \noindent {\bf (i):} for every $\g\in\C$, we have $S\cap\g(\SS^1)\ne\varnothing$;

      \medskip

      \noindent {\bf (ii):} for every $(\gamma,\Phi, T)\in \mathcal{T}(\C)$ and for $\H^1$-a.e. $s\in \mathbb{S}^1$, we have
        \begin{eqnarray}\label{spanning borel}
          &&\mbox{for $\H^n$-a.e. $x\in T[s]$}\,,
          \\\nonumber
          &&\mbox{$\exists$ a partition $\{T_1,T_2\}$ of $T$ with $x\in\partial^e T_1 \cap \partial^e T_2$}\,,
          \\ \nonumber
          &&\mbox{and s.t. $S \cup  T[s]$ essentially disconnects $T$ into $\{T_1,T_2\}$}\,.
        \end{eqnarray}

        \noindent In particular, $S$ is $\C$-spanning $\wire$ according to Definition \ref{def homot span closed} if and only if it does so according to Definition \ref{def homot span borel}.
    \end{theorem}

    \begin{remark}[$x$-dependency of $\{T_1,T_2\}$]\label{remark x dep}
    {\rm In the situation of Figure \ref{fig disco} it is clear that the same choice of $\{T_1,T_2\}$ can be used to check the validity of \eqref{spanning borel} at every $x\in T[s]$. One may thus wonder if it could suffice to reformulate \eqref{spanning borel} so that the partition $\{T_1,T_2\}$ is independent of $x$. The simpler example we are aware of and that shows this simpler definition would not work is as follows. In $\R^3$, let $\wire$ be a closed $\de$-neighborhood of a circle $\Gamma$, let $U$ be the open $\de$-neighborhood of a loop with link number {\it three} (or higher {\it odd} number) with respect to $\wire$, let $K$ be the disk spanned by $\Gamma$, and let $S=\Omega\cap[(K\setminus U)\cup\pa U]$, see
    \begin{figure}
    \input{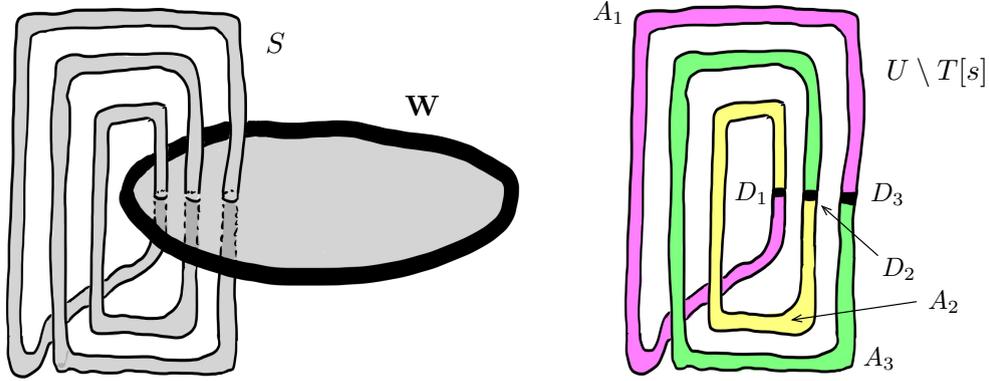}
    \caption{\small{The situation in Remark \ref{remark x dep}. The components $A_1$, $A_2$ and $A_3$ (depicted in purple, yellow, and green respectively) of $U\setminus T[s]$ are bounded by the three disks $\{D_i\}_{i=1}^3$ (depicted as boldface segments).}}
    \label{fig triple}
    \end{figure}
    Figure \ref{fig triple}. Now consider a ``test tube'' $T$ which compactly contains $U$ and is such that, for every $s$, $U\cap T[s]$ consists of three disks $\{D_i\}_{i=1}^3$. Since $U\subset T$, the property ``$S\cup T[s]$ essentially disconnects $T$ into $\{T_1,T_2\}$ in such a way that $T[s]\subset T\cap\pa^e T_1\cap\pa^e T_2$'' would immediately imply ``$U\cap(S\cup T[s])=U\cap T[s]$ essentially disconnects $T\cap U=U$ into $\{U_1,U_2\}$ with $U\cap T[s]\subset U\cap\pa^e U_1\cap\pa^e U_2$'', where $U_i=T_i\cap U$ (see step one in the proof of Theorem \ref{theorem spanning with partition} for a formal proof of this intuitive assertion). However, the latter property does not hold. To see this, denoting by $\{A_i\}_{i=1}^3$ the three connected components of $U\setminus T[s]$, we would have $U_1=A_i\cup A_j$ and $U_2=A_k$ for some choice of $i\ne j\ne k\ne i$, whereas, independently of the choice made, $U\cap\pa^e U_1\cap\pa^e U_2$ always fails to contain one of the disks $\{D_i\}_{i=1}^3$: for example, if $U_1=A_1\cup A_2$ and $U_2=A_3$, then $U\cap\pa^e U_1\cap\pa^e U_2=D_2\cup D_3$, and $D_1$ is entirely missed. We conclude that the set $S$ just constructed, although clearly $\C$-spanning $\wire$ in terms of Definition \ref{def homot span closed}, fails to satisfy the variant of \eqref{spanning borel} where a same partition $\{T_1,T_2\}$ is required to work for $\H^n$-a.e. choice of $x\in T[s]$.}
    \end{remark}

    \begin{proof}[Proof of Theorem \ref{theorem definitions equivalence}] {\it Step one}: We prove that (ii) implies (i). Indeed, if there is $\g\in\C$ such that $S\cap\g(\SS^1)=\varnothing$, then, $S$ being closed, we can find $(\g,\Phi,T)\in\T(\C)$ such that $\dist(S,T)>0$. By (ii), there is $s\in\SS^1$ such that $S \cup  T[s]$ essentially disconnects $T$. By $\dist(S,T)>0$ we see that $(S \cup  T[s])\cap T= T[s]$, so that $T[s]$ essentially disconnects $T$, a contradiction.

    \medskip

    \noindent {\it Step two}: We now prove that (i) implies (ii). To this end we consider an arbitrary $(\g,\Phi,T)\in\T(\C)$ and aim at proving the existence of $J$ of full $\H^1$-measure in $\SS^1$ such that, if $s\in J$, then \eqref{spanning borel}  holds.

    \medskip

    This is trivial, with $J=\SS^1$, if $|S\cap T|=|T|$. Indeed, in this case, we have $T=S^{\one}\cap T$, that, combined with $S$ being closed, implies $T=S \cap T$. In particular, $S\cup T[s]=T$ for every $s\in\SS^1$, and since, trivially, $T$ essentially disconnects $T$, the conclusion follows.

    \medskip

    We thus assume that $|S\cap T|<|T|$: in particular,
    \[
    U=T\setminus S
    \]
    is a non-empty, open set, whose connected components are denoted by $\{U_i\}_{i\in I}$ ($I$ a countable set). By the Lebesgue points theorem, $\L^{n+1}$-a.e. $x\in T$ belongs either to $U^\zero$ or to $U$. Then, by the smoothness of $\Phi$ and by the area formula, we can find a set $J$ of full $\H^1$-measure in $\SS^1$ such that
    \begin{equation}
      \label{choice of J 1}
      \H^n\big(T[s]\setminus(U^\zero\cup U)\big)=0\,,\qquad\forall s\in J\,.
    \end{equation}
    In particular, given $s\in J$, we just need to prove \eqref{spanning borel} when either $x\in T[s]\cap U^\zero$ or $x\in T[s]\cap U$. Before examining these two cases we also notice that we can further impose on $J$ that
    \begin{eqnarray}\label{choice of s}
     \H^n\Big( T[s]\cap\Big[\pa^e U\cup\pa^e S\cup\big(U^{\one}\setminus U\big)
      \cup \bigcup_{i\in I}\big(U_i^{\one}\setminus U_i\big)\Big]\Big)=0\,,\qquad \forall s\in J\,.
    \end{eqnarray}
    Indeed, again by the Lebesgue points theorem, the sets $\pa^e U$, $\pa^e S$, $U^{\one}\setminus U$, and $\cup_{i\in I}U_i^{\one}\setminus U_i$ are all $\L^{n+1}$-negligible.

    \medskip

    \noindent {\it Case one, $x\in  T[s]\cap U^\zero$}: To fix ideas, notice that $U^\zero\ne\varnothing$ implies $|S\cap T|>0$, and in particular  $S$ has positive Lebesgue measure.
    Given an arbitrary $s'\in J\setminus\{s\}$ we denote by $\{I_1,I_2\}$ the partition of $\SS^1$ bounded by $\{s,s'\}$, and then consider the Borel sets
    \[
    T_1=\Phi(I_1\times B_1^n)\cap S\,,\qquad T_2=\Phi(I_2\times B_1^n)\cup\,\Big(\Phi(I_1\times B_1^n)\setminus S\Big)\,.
    \]
    We first notice that $\{T_1,T_2\}$ is a non-trivial partition of $T$: Indeed $|T_1|>0$ since $x$ has density $1/2$ for $\Phi(I_1\times B_1^n)$ and (by $x\in U^\zero$) density $1$ for $S\cap T$; at the same time $|T_2|=|T\setminus T_1|\ge |T\setminus S|>0$. Next, we claim that
    \begin{equation}
      \label{heee}
      \mbox{$T^{\one}\cap\pa^eT_1\cap\pa^eT_2$ is $\H^n$-contained in $S$}\,.
    \end{equation}
    Indeed, since $\Phi(I_1\times B_1^n)$ is an open subset of $T$ with $T\cap\pa[\Phi(I_1\times B_1^n)]=T[s]\cup T[s']$, and since $\pa^eT_1$ coincides with $\pa^eS$ inside the open set $\Phi(I_1\times B_1^n)$, we easily see that
    \begin{eqnarray*}
    T^{\one}\cap\pa^eT_1\cap\pa^eT_2&=&T\cap\pa^eT_1=T\cap\pa^e\big(\Phi(I_1\times B_1^n)\cap S\big)
    \\
    &\subset&\big(\Phi(I_1\times B_1^n)\cap \pa^e S\big)\cup\Big(\big(T[s]\cup T[s']\big)\setminus S^\zero\Big)\,.
    \end{eqnarray*}
    Now, on the one hand, by $\H^n(\pa^eS\cap( T[s]\cup T[s']))=0$ (recall \eqref{choice of s}), it holds
    \[
    \mbox{$\big( T[s]\cup T[s']\big)\setminus S^\zero$ is $\H^n$-contained in $T\cap S^{\one}$}\,;
    \]
    while, on the other hand, by $\Om\cap\pa^eS\subset\Om\cap\pa S\subset\Om\cap S$ (since $S$ is closed in $\Om$) and by $\Phi(I_1\times B_1^n)\subset T\subset\Om$, we also have that $\Phi(I_1\times B_1^n)\cap \pa^e S\subset T\cap S$; therefore
    \[
    \mbox{$T^{\one}\cap\pa^eT_1\cap\pa^eT_2$ is $\H^n$-contained in $T\cap(S\cup S^{\one})= T\cap S$}\,,
    \]
    where we have used that $S$ is closed to infer $S^{\one}\subset S$. Having proved \eqref{heee} and the non-triviality of $\{T_1,T_2\}$, we conclude that $S$ (and, thus, $S\cup T[s]$) essentially disconnects $T$ into $\{T_1,T_2\}$. We are left to prove that  $x\in T\cap\pa^eT_1\cap\pa^eT_2$. To this end, we notice that $x\in T[s]\cap(T\setminus S)^\zero$ and $\Phi(I_1\times B_1^n)\subset T$ imply
    \[
    |T_1\cap B_r(x)|=|\Phi(I_1\times B_1^n)\cap S\cap B_r(x)|=|\Phi(I_1\times B_1^n)\cap B_r(x)|+{\rm o}(r^{n+1})=\frac{|B_r(x)|}2+{\rm o}(r^{n+1})\,,
    \]
    so that $x\in(T_1)^\half\subset\pa^e T_1$; since $T\cap\pa^eT_1=T\cap\pa^eT_1\cap\pa^eT_2$ and $x\in T$ we conclude the proof in the case when $x\in T[s]\cap U^\zero$.

    \medskip

    \noindent {\it Case two, $x\in  T[s]\cap U$}: In this case there exists $i\in I$ such that $x\in U_i$, and, correspondingly, we claim that
    \begin{eqnarray}\label{partition V1V2}
      &&\mbox{$\exists\{V_1,V_2\}$ a non-trivial Borel partition  of $U_i\setminus T[s]$}\,,
      \\\nonumber
      &&\mbox{s.t. $x\in \pa^eV_1\cap\pa^eV_2$ and $T\cap(\pa V_1\cup\pa V_2)\subset S\cup T[s]$}\,.
    \end{eqnarray}
    Given the claim, we conclude by setting $T_1=V_1$ and $T_2=V_2\cup(T\setminus U_i)$. Indeed, since $V_2\cap U_i=T_2\cap U_i$ with $U_i$ open implies $U_i\cap\pa^eV_1=U_i\cap\pa^eT_1$, we deduce from \eqref{partition V1V2} that
    \[
    x\in U_i\cap\pa^eV_1\cap\pa^eV_2=U_i\cap \pa^eT_1\cap\pa^e T_2\,;
    \]
    at the same time, $S\cup T[s]$ essentially disconnects $T$ into $\{T_1,T_2\}$ since, again by \eqref{partition V1V2},
    \[
    T^{\one}\cap \pa^eT_1\cap\pa^e T_2=T\cap \pa^eT_1 = T\cap \pa^eV_1\subset T\cap\pa V_1\subset S\cup T[s]\,.
    \]
    We are thus left to prove \eqref{partition V1V2}. To this end, let us choose $r(x)>0$ small enough to have that $B_{r(x)}(x)\subset U_i$, and that $B_{r(x)}(x)\setminus T[s]$ consists of exactly two connected components $\{V_1^x,V_2^x\}$; in this way,
    \begin{equation}
      \label{x in mezzo}
      x\in (V_1^x)^\half\cap (V_2^x)^\half\,.
    \end{equation}
    Next, we define
    \begin{eqnarray*}
      &&\mbox{$V_1=$ the connected component of $U_i\setminus T[s]$ containing $V_1^x$}\,,
      \\
      &&V_2=U_i\setminus(T[s]\cup V_1)\,.
    \end{eqnarray*}
    Clearly $\{V_1,V_2\}$ is a partition of $U_i\setminus T[s]$, and, thanks to $\pa V_1\cup\pa V_2\subset T[s]\cup\pa U_i$, we have
    \[
    T\cap(\pa V_1\cup\pa V_2)\subset T\cap ( T[s]\cup\pa U_i)\subset S\cup T[s]\,.
    \]
    Therefore \eqref{partition V1V2} follows by showing that $|V_1|\,|V_2|>0$. Since $V_1$ contains the connected component $V_1^x$ of $B_{r(x)}(x)\setminus T[s]$, which is open and non-empty, we have $|V_1|>0$. Arguing by contradiction, we assume that
    \[
    |V_2|=|U_i\setminus(T[s]\cup V_1)|=0\,.
    \]
    Since $V_1$ is a connected component of the open set $U_i\setminus T[s]$ this implies that
    \[
    U_i\setminus T[s]=V_1\,.
    \]
    Let $x_1\in V_1^x$ and $x_2\in V_2^x$ (where $V_1^x$ and $V_2^x$ are the two connected components of $B_{r(x)}(x)\setminus T[s]$). Since $V_1$ is connected and $\{x_1,x_2\}\subset U_i\setminus T[s]=V_1$, there is a smooth embedding $\g_1$ of $[0,1]$ into $V_1$ with $\g_1(0)=x_1$ and $\g_1(1)=x_2$. Arguing as in \cite[Proof of Lemma 10, Step 2]{DGM} using Sard's theorem, we may modify $\gamma_1$ by composing with a smooth diffeomorphism such that the modified $\gamma_1$ intersects $\pa B_{r(x)}(x)$ transversally at finitely many points. Thus $\gamma_1([0,1]) \setminus \cl B_{r(x)}(x)$ is partitioned into finitely many curves $\gamma_1((a_i,b_i))$ for disjoint arcs $(a_i,b_i)\subset [0,1]$. Since $B_{r(x)}(x)\setminus T[s]$ is disconnected into $V_1^x$ and $V_2^x$ and $\gamma_1$ is disjoint from $T[s]$, there exists $i$ such that, up to interchanging $V_1^x$ and $V_2^x$, $\gamma(a_i)\in \cl V_1^x \cap \pa B_{r(x)}(x)$ and $\gamma(b_i)\in \cl V_2^x \cap \pa B_{r(x)}(x)$. Let us call $\tilde{\gamma_1}$ the restriction of $\gamma_1$ to $[a_i,b_i]$. Next, we choose a smooth embedding $\g_2$ of $[0,1]$ into $B_{r(x)}(x)$ such that $\g_2(0)=\tilde{\gamma}_1(a_i)$, $\g_2(1)=\tilde{\gamma}_1(b_i)$, and $\g_2([0,1])$ intersects $T[s]\cap B_{r(x)}(x)$ at exactly one point, denoted by $x_{12}=\g_2(t_0)$, with
    \begin{align}\label{nonvanishing derivative}
        \g_2'(t_0)\neq 0\,.
    \end{align}
    Since $\tilde{\gamma_1}((a_i,b_i)) \cap \cl B_{r(x)}(x)=\varnothing$ and $\gamma_2([0,1]) \subset \cl B_r(x)$, we can choose $\g_2$ so that the concatenation of $\g_1$ and $\g_2$ defines a smooth embedding $\g_*$ of $\SS^1$ into $U_i\subset T$. Up to reparametrizing we may assume that $\g_*(1)=x_{12}$. Since $\g_1([0,1])\subset V_1$ and $V_1\cap(S\cup T[s])=\varnothing$, we have that
    \begin{equation}
      \label{single point}
      \g_*(\SS^1)\cap (S\cup T[s])=\g_2([0,1])\cap(S\cup T[s])=\{x_{12}\}\subset T[s]\cap B_{r(x)}(x)\,.
    \end{equation}
    A first consequence of \eqref{single point} is that $\g_*(\SS^1)\cap S=\varnothing$. Similarly, the curve $\g_{**}:\mathbb{S}^1\to \Omega$ defined via $\g_{**}(t)=\g_{*}(\overline{t})$ ($t\in\SS^1$) where the bar denotes complex conjugation, has the same image as $\gamma_*$ and thus satisfies $\gamma_{**}(\mathbb{S}^1) \cap S = \varnothing$ as well. Therefore, in order to obtain a contradiction with $|V_2|=0$, it is enough to prove that either $\g_*\in\C$ or $\g_{**}\in\C$. To this end we are now going to prove that one of $\g_*$ or $\g_{**}$ is homotopic to $\g$ in $T$ (and thus in $\Om$), where $\g$ is the curve from the tube $(\g,\Phi,T)\in\T(\C)$ considered at the start of the argument.

    \medskip

    Indeed, let $\pp:\SS^1\times B_1^n\to\SS^1$ denote the canonical projection $\pp(t,x)=t$, and consider the curves $\s_*=\pp\circ\Phi^{-1}\circ\g_*:\SS^1\to\SS^1$ and $\s_{**}=\pp\circ\Phi^{-1}\circ\g_{**}$. By \eqref{single point}, $\s_*^{-1}(\{s\})=\{1\}$, and $1$ is a regular point of $\s_{*}$ by \eqref{nonvanishing derivative} and since $\Phi$ is a diffeomorphism. Similarly, $\s_{**}^{-1}(\{s\})=\{1\}$ and $1$ is a regular point of $\s_{**}$. Now by our construction of $\g_{**}$, exactly one of $\g_*$ or $\g_{**}$ is orientation preserving at $1$ and the other is orientation reversing. So we may compute the winding numbers of $\s_*$ and $\s_{**}$ via (see e.g. \cite[pg 27]{milnor}):
    \begin{align}\notag
        \mbox{deg}\,\s_* = \mbox{sgn}\, \det D\s_*(1) = - \mbox{sgn}\, \det D\s_{**}(1) = - \mbox{deg}\,\s_{**}\in \{+1,-1\}\,.
    \end{align}
    If we define $\s=\pp\circ\Phi^{-1}\circ\g$, then $\s$ has winding number $1$, and so is homotopic in $\mathbb{S}^1$ to whichever of $\s_*$ or $\s_{**}$ has winding number $1$. Since $\Phi$ is a diffeomorphism of $\SS^1\times B_1^n$ into $\Om$, we conclude that $\g$ is homotopic relative to $\Omega$ to one of $\g_{*}$ or $\g_{**}$, and, thus, that $\g^*\in\C$ or $\g_{**}\in \C$ as desired.
    \end{proof}

    \section{Convergence of every minimizing sequence of $\Psi_{\rm bk}(v)$}\label{appendix every} In proving Theorem \ref{thm existence EL for bulk} we have shown that every minimizing sequence $\{(K_j,E_j)\}_j$ of $\Psi_{\rm bk}(v)$ has a limit $(K,E)$ such that, denoting by $B^{(w)}$ a ball of volume $w$, it holds
    \begin{align}\notag
    \Psi_{\rm bk}(v)=\Psi_{\rm bk}(|E|)+P(B^{(v-|E|)})\,,\qquad\Psi_{\rm bk}(|E|)=\F_{\rm bk}(K,E)\,,
    \end{align}
    with both $K$ and $E$ bounded. In particular, minimizers of $\Psi_{\rm bk}(v)$ can be constructed in the form $(K\cup\pa B^{(v-|E|)}(x),E\cup B^{(v-|E|)}(x))$ provided $x$ is such that $B^{(v-|E|)}(x)$ is disjoint from $K\cup E\cup \wire$. This argument, although sufficient to prove the existence of minimizers of $\Psi_{\rm bk}(v)$, it is not sufficient to prove the convergence of every minimizing sequence of $\Psi_{\rm bk}(v)$, i.e., to exclude the possibility that $|E|<v$. This is done in the following theorem at the cost of assuming the $C^2$-regularity of $\pa\Om$. This result will be important in the companion paper \cite{MNR2}.

    \begin{theorem}\label{theorem cant lose volume}
      If $\wire$ is the closure of a bounded open set with $C^2$-boundary, $\C$ is a spanning class for $\wire$, and $\ell<\infty$, then for every $v>0$ and every minimizing sequence $\{(K_j,E_j)\}_j$ of $\Psi_{\rm bk}(v)$ there is a minimizer $(K,E)$ of $\Psi_{\rm bk}(v)$ such that $K$ is $\H^n$-rectifiable and, up to extracting subsequences and as $j\to\infty$,
      \begin{equation}
        \label{what K and E do}
          E_j\to E\,,\qquad \mu_j\weakstar  \H^n\mres(\Om\cap \pa^*E)+ 2\,\H^n\mres(K\cap E^\zero)\,,
      \end{equation}
      where $\mu_j=\H^n\mres(\Om\cap \pa^*E_j)+2\,\H^n\mres(\RR(K_j)\cap E_j^\zero)$.
    \end{theorem}

    \begin{proof} By step three in the proof of Theorem \ref{thm existence EL for bulk section}, there is $(K,E)\in\KK_{\rm B}$ satisfying \eqref{what K and E do} and such that $K$ and $E$ are bounded, $(K,E)$ is a minimizer of $\Psi_{\rm bk}(|E|)$, $K$ is $\H^n$-rectifiable, and $|E|\le v$; moreover, if $v>|E|$, then there is $x\in\R^{n+1}$ such that $B^{(v-|E|)}(x)$ is disjoint from $K\cup E\cup \wire$ and $(K',E')=(K\cup\pa B^{(v-|E|)}(x),E\cup B^{(v-|E|)}(x))$ is a minimizer of $\Psi_{\rm bk}(v)$. We complete the proof by deriving a contradiction with the $v^*=v-|E|>0$ case. The idea is to relocate $B^{(v^*)}(x)$ to save perimeter by touching $\pa \wire$ or $\pa E$; see Figure \ref{fig gammadelta}.

    \medskip

    First of all, we claim that $K=\Om\cap\partial E$. If not, since $(K,E)$ and $(K',E')$ respectively are minimizers of $\Psi_{\rm bk}(|E|)$ and $\Psi_{\rm bk}(v)$, then there are $\l,\l'\in\R$ such that $(K,E)$ and $(K',E')$ respectively satisfy \eqref{first var cap theorem section} with $\l$ and $\l'$.  By localizing \eqref{first var cap theorem section} for $(K',E')$ at points in $\Om\cap\pa^*E$ we see that it must be $\l=\l'$; by localizing at points in $\pa B^{(v-|E|)}(x)$, we see that $\l$ is equal to the mean curvature of $\pa B^{(v-|E|)}(x)$, so that $\l>0$; by arguing as in the proof of \cite[Theorem 2.9]{KMS2} (see \cite{novackGENMIN} for the details), we see that if $K\setminus(\Om\cap\pa E)\ne\varnothing$, then $\l\le 0$, a contradiction.

    \medskip

    Having established that $K=\Om\cap\pa E$, we move an half-space $H$ compactly containing $\cl(E)\cup\wire$ until the boundary hyperplane  $\pa H$ first touches $\cl(E)\cup\wire$. Up to rotation and translation, we can thus assume that $H=\{x_{n+1}>0\}$ and
    \begin{equation}
      \label{first touch}
      0\in\cl(E)\cup\wire\subset \cl(H)\,.
    \end{equation}
    We split \eqref{first touch} into two cases, $0\in\Om\cap\pa E$ and $0\in\wire$, that are then separately discussed for the sake of clarity. In both cases we set $x=(x',x_{n+1})\in\R^n\times\R\equiv\R^{n+1}$, and set
    \begin{eqnarray*}
      {\bf C}_\de&=&\{x:x_{n+1}\in(0,\de)\,,|x'|<\de\}\,,
      \\
      {\bf L}_\de&=&\{x:|x'|=\de,x_{n+1}\in(0,\de)\}\,,
      \\
      {\bf T}_\de&=&\{x:x_{n+1}=\de\,,|x'|<\de\}\,,
      \\
      {\bf D}_\de&=&\{x:x_{n+1}=0\,,|x'|<\de\}\,,
    \end{eqnarray*}
    for every $\de>0$.

    \medskip

    \noindent {\it Case one, $0\in\Om\cap\pa E$}: In this case, by the maximum principle \cite[Lemma 3]{delgadinomaggi}, \eqref{first var cap theorem section}, and the Allard regularity theorem, we can find $\de_0>0$ and $u\in C^2({\bf D}_{\de_0};[0,\de_0])$ with $u(0)=0$ and $\nabla u(0)=0$ such that ${\bf C}_{\de_0}\cc\Om$ and
    \begin{eqnarray}
     \label{E is locally epigraph}
      &&E\cap{\bf C}_{\de_0}=\big\{x\in{\bf C}_{\de_0}:\de_0>x_{n+1}>u(x')\big\}\,,
      \\
      \nonumber
      &&(\pa E)\cap{\bf C}_{\de_0}=\big\{x\in{\bf C}_{\de_0}:x_{n+1}=u(x')\big\}\,.
    \end{eqnarray}
    Since $0\le u(x')\le C\,|x'|^2$ for some $C=C(E)$, if we set
    \begin{equation}
      \label{def of Gamma delta}
      \Gamma_\de=\Big\{x\in{\bf C}_\de:0<x_{n+1}<u(x')\Big\}\,,\qquad\de\in(0,\de_0)\,,
    \end{equation}
    then we have
    \begin{eqnarray}\label{cylindrical volume estimate}
      |\Gamma_\de|\!\!&\le&\!\! C\,\de^{n+2}\,,
      \\\label{cylindrical perimeter estimate}
      P\big(\Gamma_\de;{\bf L}_\de\big)\!\!&\le&\!\! C\,\de^{n+1}\,.
    \end{eqnarray}
    We then set
    \begin{equation}
      \label{def of Edelta}
      E_\de=E \cup \Gamma_\de\cup\big(B_{r_\delta}(z_\delta) \setminus H\big)\,,
    \end{equation}
    see
    \begin{figure}
    \input{gammadelta.pstex_t}
    \caption{\small{(a): the construction of $E_\de$ when $0\in\Om\cap\pa E$; (b) the construction of $E_\de$ when $0\in\wire$.}}
    \label{fig gammadelta}
    \end{figure}
    Figure \ref{fig gammadelta}-(a), where $r_\de>0$ and $z_\de\in\R^{n+1}\setminus\cl(H)$ are uniquely determined by requiring that, first,
    \begin{eqnarray}
    \label{def of zeta delta 1}
      \cl(B_{r_\de}(z_\de))\cap\pa H=\pa {\bf C}_\de\cap\pa H=\big\{x:x_{n+1}=0\,,|x'|\leq \de\big\}\,,
    \end{eqnarray}
    and, second, that
    \begin{equation}
      \label{def of zeta delta 2}
      |E_\de|=v\,.
    \end{equation}
    To see that this choice is possible, we first notice that, since $E\cap\Gamma_\de=\varnothing$, \eqref{def of zeta delta 2} is equivalent to
    \begin{equation}
      \label{def of zeta delta 2bis}
      \big|B_{r_\delta}(z_\delta) \setminus H\big|=v-|E|-|\Gamma_\de|=v^*-|\Gamma_\de|\,.
    \end{equation}
    Taking \eqref{cylindrical volume estimate} into account we see that \eqref{def of zeta delta 1} and \eqref{def of zeta delta 2bis} uniquely determine $z_\de\in\R^{n+1}$ and $r_\de>0$ as soon as $\de_0$ is small enough to guarantee $v^*-|\Gamma_{\de_0}|>0$. In fact, by \eqref{cylindrical volume estimate}, $v^*-|\Gamma_\de|\to v^*>0$ with $\H^n(\pa {\bf C}_\de\cap\pa H)\to 0$ as $\de\to 0^+$, so that, up to further decrease $\de_0$, we definitely have $z_\de\not\in H$, and
    \begin{equation}
      \label{size of rdelta}
    \Big|r_\de-\Big(\frac{v^*}{\om_{n+1}}\Big)^{1/(n+1)}\Big|\le C\,\de^{n+2}\,,
    \end{equation}
    where $C=C(E,n,v^*)$.

    \medskip

    We now use the facts that $K\cup E^\one$ is $\C$-spanning $\wire$ and that $E\subset E_\de$ to prove that
    \begin{equation}
      \label{def of Kdelta}
      (K_\de,E_\de)=((\Omega \cap \pa^* E_\delta) \cup (K \cap E_\de^\zero),E_\de)
    \end{equation}
    is such that $K_\de\cup E_\de^\one$ is $\C$-spanning $\wire$ (and thus is admissible in $\Psi_{\rm bk}(v)$ by \eqref{def of zeta delta 2}). To this end, it is enough to show that
    \begin{align}\label{final containment}
    K\cup E^\one \shn K_\delta \cup E_\delta^\one\,.
    \end{align}
    Indeed, by $E\subset E_\delta$ and Federer's theorem \eqref{federer theorem} we have
    \begin{align}\label{e edelta containment}
       E^\one \subset E_\de^\one\,,\qquad E_\de^\zero \subset E^\zero\,,\qquad  E^\one \cup \pa^* E \shn E_\de^\one \cup \pa^* E_\de\,.
    \end{align}
    (Notice indeed that $\pa^*E\subset E^\half\subset\R^{n+1}\setminus E_\de^\zero$). Next, using in order Federer's theorem \eqref{federer theorem}, \eqref{e edelta containment} and $K\subset \Omega$, and the definition of $K_\delta$, we have
    $$
    E^\one \cup (K\setminus E_\de^\zero)\ehn E^\one \cup [K \cap (\pa^* E_\de \cup E_\de^\one)]\shn E_\de^\one \cup (\Om \cap \pa^* E_\de) \subset E_\de^\one \cup K_\de\,.
    $$
    But $K\cap E_\delta^\zero \subset K_\delta$ by definition, which combined with the preceding containment completes the proof of \eqref{final containment}. Having proved that $(K_\de,E_\de)$ is admissible in $\Psi_{\rm bk}(v)$, we have
    \begin{equation}
      \label{start}
      \F_{\rm bk}(K,E)+P(B^{(v^*)})=\Psi_{\rm bk}(v)\le\F_{\rm bk}(K_\de,E_\de)\,.
    \end{equation}
    By \eqref{start}, the definition of $K_\delta$, and \eqref{e edelta containment}, we find
    \begin{eqnarray*}
       &&P(E;\Om) + 2\,\H^n(K \cap E^\zero) + P(B^{(v^*)}) \leq P(E_\delta;\Om) + 2\,\H^n(K_\delta \cap E_\delta^\zero)
       \\
       && \hspace{1.4cm}\le P(E_\delta;\Om) + 2\,\H^n(K \cap E_\delta^\zero)   \leq P(E_\delta;\Om) + 2\,\H^n(K \cap E^\zero)\,,
    \end{eqnarray*}
    from which we deduce
    \begin{align}\label{bad estimate}
    P(E;\Omega) + P(B^{(v^*)}) \leq P(E_\de;\Om)\,.
    \end{align}
    We now notice that $E_\de$ coincides with $E$ in the open set $\Om\cap H\setminus\cl({\bf C}_\de)$, and with $B_{r_\de}(z_\de)$ in the open set $\R^{n+1}\setminus\cl(H)$, so that
    \begin{eqnarray*}
      &&\Big(\Om\cap H\setminus\cl({\bf C}_\de)\Big)\cap\pa^*E_\de=\Big(\Om\cap H\setminus\cl({\bf C}_\de)\Big)\cap\pa^*E\,,
      \\
      &&
      \big(\Om\setminus\cl(H)\big)\cap\pa^*E_\de=  \big(\pa B_{r_\delta}(z_\delta)\big)\setminus\cl(H)\,,
    \end{eqnarray*}
    and \eqref{bad estimate} is equivalent to
    \begin{eqnarray}\label{bad estimate2}
    &&P\big(E;\Omega\cap(\pa H\cup\cl({\bf C}_\de)\big) + P(B^{(v^*)})
    \\\nonumber
    &&\leq P\big(E_\de;\Om\cap(\pa H\cup\cl({\bf C}_\de)\big)+P(B_{r_\de}(z_\de);\R^{n+1}\setminus\cl(H))\,.
    \end{eqnarray}
    In fact, it is easily proved that $(\pa^*E)\cap(\pa H)\setminus\cl({\bf C}_\de)=(\pa^*E_\de)\cap(\pa H)\setminus\cl({\bf C}_\de)$ (which is evident from Figure \ref{fig gammadelta}), so that \eqref{bad estimate2} readily implies
    \begin{eqnarray}\label{bad estimate3}
    P(B^{(v^*)})
    \leq P\big(E_\de;\Om\cap\cl({\bf C}_\de)\big)+P(B_{r_\de}(z_\de);\R^{n+1}\setminus\cl(H))\,.
    \end{eqnarray}
    Now, ${\bf C}_\de\cc\Om$. Moreover, by \eqref{E is locally epigraph}, we have that ${\bf T}_\de$ (the top part of $\pa{\bf C}_\de$) is contained in $E^\one\subset E_\de^\one$, and is thus $\H^n$-disjoint from $\pa^*E_\de$. Similarly, again by \eqref{E is locally epigraph} we have $E\cup\Gamma_\de={\bf C}_\de$, and thus ${\bf D}_\de\subset (E\cup\Gamma_\de)^\half$; at the same time, by \eqref{def of zeta delta 1} we have ${\bf D}_\de\subset(B_{r_\de}(z_\de)\setminus H)^\half$; therefore ${\bf D}_\de\subset E_\de^\one$, and thus ${\bf D}_\de$ is $\H^n$-disjoint from $\pa^*E_\de$. Finally, again by $E\cup\Gamma_\de={\bf C}_\de$ we see that $P(E_\de;{\bf C}_\de)=0$. Therefore, in conclusion,
    \begin{align}\label{first combo}
    P\big(E_\de;\Om\cap\cl({\bf C}_\de)\big)=P(E_\de;{\bf L}_\de)=P(\Gamma_\de;{\bf L}_\de)\le C\,\de^{n+1}\,,
    \end{align}
    where we have used again first \eqref{E is locally epigraph}, and then \eqref{cylindrical perimeter estimate}. Combining \eqref{bad estimate3}-\eqref{first combo} we get
    \begin{eqnarray}\label{bad estimate4}
    P(B^{(v^*)})\le P(B_{r_\de}(z_\de);\R^{n+1}\setminus\cl(H))+C\,\de^{n+1}\,.
    \end{eqnarray}
    Finally, by \eqref{def of zeta delta 1}, \eqref{cylindrical volume estimate}, and \eqref{size of rdelta} we have
    \[
    P(B_{r_\de}(z_\de);\R^{n+1}\setminus\cl(H))\le P(B^{(v^*)})-C(n)\,\de^n\,;
    \]
    by combining this estimate with \eqref{bad estimate4}, we reach a contradiction for $\de$ small enough.

    \medskip

    \noindent {\it Case two, $0\in\wire$}: In this case, by the $C^2$-regularity of $\pa\Om$ we can find $\de_0>0$ and $u\in C^2({\bf D}_{\de_0};[0,\de_0])$ with $u(0)=0$ and $\nabla u(0)=0$ such that
    \begin{eqnarray}
    \label{wire is locally epigraph}
      &&\wire\cap{\bf C}_{\de_0}=\big\{x\in{\bf C}_{\de_0}:\de_0>x_{n+1}>u(x')\big\}\,,
      \\\nonumber
      &&(\pa\Om)\cap{\bf C}_\de=\big\{x\in{\bf C}_{\de_0}:x_{n+1}=u(x')\big\}\,.
    \end{eqnarray}
    We have $0\le u(x')\le C\,|x'|^2$ for every $|x'|<\de_0$ (and some $C=C(\wire)$), so that defining $\Gamma_\de$ as in \eqref{def of Gamma delta} we still obtain \eqref{cylindrical volume estimate} and \eqref{cylindrical perimeter estimate}. We then define $E_\de$, $r_\de$, and $z_\de$, as in \eqref{def of Edelta}, \eqref{def of zeta delta 1} and \eqref{def of zeta delta 2}. Notice that now $E$ and $\Gamma_\de$ may not be disjoint (see Figure \ref{fig gammadelta}-(b)), therefore \eqref{def of zeta delta 2} is not equivalent to \eqref{def of zeta delta 2bis}, but to
    \[
    \big|B_{r_\delta}(z_\delta) \setminus H\big|=v-|E|-|\Gamma_\de\setminus E|=v^*-|\Gamma_\de\setminus E|\,.
    \]
    This is still sufficient to repeat the considerations based on \eqref{def of zeta delta 1} and \eqref{cylindrical volume estimate} proving that $r_\de$ and $z_\de$ are uniquely determined, and satisfy \eqref{size of rdelta}. We can repeat the proof that $(K_\de,E_\de)$ defined as in \eqref{def of Kdelta} is admissible in $\Psi_{\rm bk}(v)$ (since that proof was based only on the inclusion $E\subset E_\de$), and thus obtain \eqref{bad estimate}. The same considerations leading from \eqref{bad estimate} to \eqref{bad estimate3} apply in the present case too, and so we land on
    \begin{eqnarray}\label{bad estimate3bis}
    P(B^{(v^*)})
    \leq P\big(E_\de;\Om\cap\cl({\bf C}_\de)\big)+P(B_{r_\de}(z_\de);\R^{n+1}\setminus\cl(H))\,.
    \end{eqnarray}
    Now, by \eqref{wire is locally epigraph}, ${\bf T}_\de$ is contained in $\wire$, so that $P(E_\de;{\bf T}_\de)=0$. At the same time, if $x=(x',0)\in{\bf D}_\de\cap\Om$, then $u(x')>0$, and thus $x\in (E_\de\cap H)^\half$; since, by \eqref{def of zeta delta 1}, we also have $x\in (E_\de\setminus H)^\half$, we conclude that ${\bf D}_\de\cap\Om\subset E_\de^\one$, and thus that
    \[
    P\big(E_\de;\Om\cap\cl({\bf C}_\de)\big)=P\big(E_\de;\Om\cap{\bf L}_\de\big)\le \H^n(\Om\cap{\bf L}_\de)\le C\,\de^{n+1}\,,
    \]
    where we have used $0\le u(x')\le C\,|x'|^2$ for every $|x'|<\de_0$ again. We thus deduce from \eqref{bad estimate3bis} that
    \[
    P(B^{(v^*)})\le P(B_{r_\de}(z_\de);\R^{n+1}\setminus\cl(H))+C\,\de^{n+1}\,,
    \]
    and from here we conclude as in case one.
    \end{proof}

    \section{An elementary lemma}\label{sec: geometric remark appendix}

    In this appendix we provide a proof of Lemma \ref{dldrg lemma}. The proof is an immediate corollary of a geometric property of closed $\mathcal{C}$-spanning sets (see \eqref{dichot 1}-\eqref{dichot 2} below) first proved in $\mathbb{R}^{n+1}$ for $n\geq 2$ \cite[Lemma 4.1]{delederosaghira}. Here we extend this property to the plane. The difference between $\mathbb{R}^2$ and $\mathbb{R}^{n+1}$ for $n\geq 2$ stems from a part of the argument where one constructs a new admissible spanning curve by modifying an existing one inside a ball. Specifically, ensuring that the new curve does not intersect itself requires an extra argument in $\mathbb{R}^2$.

    \begin{lemma}\label{dldrg appendix 2d lemma}
        Let $n\geq 1$, $\wire\subset \mathbb{R}^{n+1}$ be closed, $\mathcal{C}$ be a spanning class for $\wire$, $S\subset \Omega := \mathbb{R}^{n+1}\setminus \wire$ be relatively closed and $\mathcal{C}$-spanning $\wire$, and $B_r(x) \cc \Omega$. Let $\{\Gamma_i \}_i$ be the countable family of equivalence classes of $\pa B_r(x) \setminus S$ determined by the relation:
        \begin{align}\label{dldrg relation}
            y\sim x \iff \mbox{$\exists \tilde{\gamma} \in C^0([0,1],\cl B_r(x) \setminus S):\tilde{\gamma}(0)=y$, $\tilde{\gamma}(1)=z$, $\tilde{\gamma}((0,1)) \subset B_r(x)$}\,.
        \end{align}
    Then if $\gamma \in \mathcal{C}$, either
    \begin{align}\label{dichot 1}
        \gamma \cap (S \setminus B_r(x))\neq \emptyset
    \end{align}
    or there exists a connected component $\sigma$ of $\gamma \cap \cl B_r(x)$ which is homeomorphic to an interval and such that
    \begin{align}\label{dichot 2}
    \mbox{the endpoints of $\sigma$ belong to two distinct equivalence classes of $\pa B_r(x) \setminus S$.}
    \end{align}
    In particular, the conclusion of Lemma \ref{dldrg lemma} holds.
    \end{lemma}

    \begin{remark}
        The planar version of Lemma \ref{dldrg appendix 2d lemma} allows one to extend the main existence result \cite[Theorem 2.7]{delederosaghira} to $\mathbb{R}^2$.
    \end{remark}

    \begin{proof}[Proof of Lemma \ref{dldrg appendix 2d lemma}]
    The proof is divided into two pieces. First we show how to deduce Lemma \ref{dldrg lemma} from the fact that at least one of \eqref{dichot 1}-\eqref{dichot 2} holds. Then we show in $\mathbb{R}^2$ that \eqref{dichot 2} must hold whenever \eqref{dichot 1} does not, completing the lemma since the case $n\geq 2$ is contained in \cite[Lemma 4.1]{delederosaghira}.
    \par
    \medskip
    \noindent{\it Conclusion of Lemma \ref{dldrg lemma} from \eqref{dichot 1}-\eqref{dichot 2}}: We must show that either $\gamma(\mathbb{S}^1) \setminus B_r(x)\neq \varnothing$ or that it intersects at least two open connected components of $B_r(x) \setminus S$. If $\gamma(\mathbb{S}^1) \setminus B_r(x)\neq \varnothing$ we are done, so suppose that $\gamma(\mathbb{S}^1) \setminus B_r(x)= \varnothing$. Then \eqref{dichot 2} must be true, so that the endpoints of some arc $\sigma=\gamma((a,b))\subset B_r(x)$ for an interval $(a,b)\subset \SS^1$ belong to distinct equivalence classes. Choose $\rho$ small enough so that $B_\rho(\gamma(a)) \cup B_\rho(\gamma(b))\subset \Omega \setminus S$ and $a'$, $b'\in I$ such that $\gamma(a') \in  B_\rho(\gamma(a))$ and $\gamma(b') \in  B_\rho(\gamma(b))$. If $\gamma(a')$ and $\gamma(b')$ belonged to the same open connected component of $B_r(x) \setminus S$, we would contradict \eqref{dichot 2}, so they belong to different components as desired.
    \par
    \medskip
    \noindent{\it Verification of \eqref{dichot 1}-\eqref{dichot 2} in $\mathbb{R}^2$}: As in \cite[Lemma 10]{DLGM}, we may reduce to the case where $\gamma$ intersects $\pa B_r(x)$ transversally at finitely many points $\{\gamma(a_k)\}_{k=1}^K \cup \{\gamma( b_k)\}_{k=1}^K$ such that $\gamma \cap B_r(x) = \cup_k \gamma((a_k,b_k))$ and $\{[a_k,b_k]\}_k$ are mutually disjoint closed arcs in $\mathbb{S}^1$. If \eqref{dichot 1} holds we are done, so we assume that
        \begin{align}\label{doesnt hit the outside}
        \gamma \cap  S \setminus  B_r(x)= \varnothing
        \end{align}
        and prove \eqref{dichot 2}. Note that each pair $\{\gamma(a_k),\gamma(b_k)\}$ bounds two open arcs in $\pa B_r(x)$; we make a choice now as follows. Choose $s_0 \in \pa B_r(x) \setminus \cup_k \{\gamma(a_k),\gamma(b_k)\}$. Based on our choice of $s_0$, for each $k$ there is a unique open arc $\ell_k\subset \pa B_r(x)$ such that $\pa_{\pa B_r(x)}\ell_k=\{\gamma(a_k),\gamma(b_k)\}$ and $s_0 \notin \cl_{\pa B_r(x)}\ell_k$. We claim that
    \begin{align}\label{arc claim}
        \mbox{if $k\neq k'$, then either $\ell_k \cc \ell_{k'}$ or $\ell_{k'}\cc \ell_k$}\,.
    \end{align}

    \medskip

    \noindent{\it To prove \eqref{arc claim}}: We consider simple closed curves $\gamma_k$ with images $\gamma((a_k,b_k)) \cup \cl_{\pa B_r(x)} \ell_k$. By the Jordan curve theorem, each $\gamma_k$ defines a connected open subset $U_k$ of $B_r(x)$ with $\pa U_k \cap \pa B_r(x) = \cl_{\pa B_r(x)} \ell_k$. Aiming for a contradiction, if \eqref{arc claim} were false, then for some $k\neq k'$, either
    \begin{align*}
        &\mbox{$\gamma(a_k) \in \ell_{k'}\subset \cl U_{k'}$ and $\gamma(b_k)\in \pa B_r(x) \setminus \cl_{\pa B_r(x)}\ell_{k'}\subset \pa B_r(x) \setminus \cl U_{k'}$ or }\quad \\
        &\mbox{$\gamma(b_k) \in \ell_{k'}\subset \cl U_{k'}$ and $\gamma(a_k)\in \pa B_r(x) \setminus \cl_{\pa B_r(x)}\ell_{k'}\subset \pa B_r(x) \setminus \cl U_{k'}$}\,;
    \end{align*}
    in particular, $\gamma((a_k,b_k))$ has non-trivial intersection with both the open sets $U_{k'}$ and $B_r(x) \setminus \cl U_{k'}$. By the continuity of $\gamma$ and the connectedness of $(a_k,b_k)$, we thus deduce that $\gamma((a_k,b_k)) \cap \pa U_{k'} \neq \varnothing$. Upon recalling that $\gamma((a_k,b_k))\subset B_r(x)$, we find $\gamma((a_k,b_k)) \cap \pa U_{k'}\cap B_r(x)=\gamma((a_k,b_k)) \cap \gamma((a_{k'},b_{k'}))\neq \varnothing$. But this contradicts the fact that $\gamma$ smoothly embeds $\SS^1$ into $\Omega$. The proof of \eqref{arc claim} is finished.

    \medskip

    Returning to the proof of \eqref{dichot 2}, let us assume for contradiction that
    \begin{align}\label{similar endpoints}
        \gamma(a_k) \sim \gamma(b_k)\quad \forall 1\leq k \leq K\,.
    \end{align}
    We are going to use \eqref{doesnt hit the outside}, \eqref{arc claim}, and \eqref{similar endpoints} to create a piecewise smooth embedding $\overline{\gamma}:\mathbb{S}^1\to \Omega$ which is a homotopic deformation of $\gamma$ (and thus approximable by elements in $\mathcal{C}$) such that $\overline{\gamma}\cap S = \varnothing$. After reindexing the equivalence classes $\Gamma_i$, we may assume that $\{\Gamma_1,\dots, \Gamma_{I_{\gamma}} \}$ are those equivalence classes containing any pair $\{\gamma(a_k),\gamma(b_k)\}$ for $1\leq k \leq K$. We will construct $\overline{\gamma}$ in steps by redefining $\gamma$ on those $[a_k,b_k]$ with images under $\gamma$ having endpoints belonging to the same $\Gamma_i$. For future use, let $\Omega_i$ be the equivalence classes of $B_r(x)\setminus S$ determined by the relation \eqref{dldrg relation}. Note that they are open connected components of $B_r(x) \setminus S$.

    \medskip

    \noindent{\it Construction corresponding to $\Gamma_1$}: Relabelling in $k$ if necessary, we may assume that $\{1,\dots , K_1\}$ for some $1\leq K_1 \leq K$ are the indices such that $\{\gamma(a_k),\gamma(b_k)\}\subset \Gamma_1$. By further relabelling and applying \eqref{arc claim} we may assume: first, that $\ell_1$ is a ``maximal" arc among $\{\ell_1,\dots,\ell_{K_1}\}$, in other words
    \begin{align}\label{maximal arc}
     \mbox{for given $k\in\{2,\dots K_1\}$, either $\ell_1 \cap \ell_k=\varnothing$ or $\ell_k \cc \ell_1$}\,;
    \end{align}
    and second, that for some $K_1^1\leq K_1$, $\{\ell_2,\dots, \ell_{K_1^1}\}$ are those arcs contained in $\ell_1$. Since $\Omega_1$ is open and connected, we may connect $\gamma(a_1)$ to $\gamma(b_1)$ by a smooth embedding $\overline{\gamma}_1:[a_1,b_1] \to \cl B_r(x) \setminus S$ with $\overline{\gamma}_1((a_1,b_1))\subset \Omega_1$. Also, by the Jordan curve theorem, $\ell_1 \cup \overline{\gamma}_1$ defines an open connected subset $W_1$ of $B_r(x)$ with $\pa W_1 \cap S = \varnothing$. Using \eqref{arc claim}, we now argue towards constructing pairwise disjoint smooth embeddings $\overline{\gamma}_k:[a_k,b_k]\to \Gamma_1 \cup \Omega_1$.
    \par
    \medskip
    We first claim that
    \begin{align}\label{its path connected}
    \mbox{$W_1\setminus S$ is path-connected}\,.
    \end{align}
    To prove \eqref{its path connected}, consider any $y,z\in W_1\setminus S$. Since $\Omega_1\supset W_1\setminus S$ is open and path-connected, we may obtain continuous $\tilde{\gamma}:[0,1]\to \Omega_1$ connecting $y$ and $z$. If $\tilde{\gamma}([0,1])\subset W_1\setminus S$, we are done. Otherwise, $\varnothing\neq \tilde{\gamma}\cap (\Omega_1 \setminus (W_1\setminus S))= \Omega_1 \setminus W_1 $, with the equality following from $\Omega_1 \cap S = \varnothing$. Combining this information with $\tilde{\gamma}(\{0,1\})\subset W_1\setminus S$, we may therefore choose $[\delta_1,\delta_2]\subset (0,1)$ to be the smallest interval such that $\tilde{\gamma}([0,1]\setminus [\delta_1,\delta_2]) \subset W_1\setminus S$. On $(\delta_1,\delta_2)$, we redefine $\tilde{\gamma}$ using the fact that $\tilde{\gamma}(\{\delta_1,\delta_2\})\subset \pa W_1\cap B_r(x)=\overline{\gamma}_1((a_1,b_1))$ by letting $\tilde{\gamma}((\delta_1,\delta_2)) = \overline{\gamma}_1(I)$, where $\overline{\gamma}_1(I)$ has endpoints $\tilde{\gamma}(\delta_1)$ and $\tilde{\gamma}(\delta_2)$ and $I\subset (a_1,b_1)$. The modified $\tilde{\gamma}$ is a concatenation of continuous curves and is thus continuous; furthermore, $\tilde{\gamma}^{-1}(W_1\setminus S) = [0,\delta_1) \cup (\delta_2,1]$. It only remains to ``push" $\tilde{\gamma}$ entirely inside $W_1\setminus S$, which we may easily achieve by projecting $\tilde{\gamma}((\delta_1-\varepsilon,\delta_2+\varepsilon))$ inside $W_1\setminus S$ for small $\varepsilon$ using the distance function to the smooth curve $\overline{\gamma}_1(a_1,b_1)=\pa W_1 \cap B_r(x) \subset B_r(x) \setminus S$. This completes \eqref{its path connected}.
    \par
    \medskip
    But now since $W_1\setminus S$ is path-connected and open, we may connect any two points in it by a smooth embedding of $[0,1]$, which in particular allows us to connect $\gamma(a_2)$ and $\gamma(b_2)$ by smooth embedding $\overline{\gamma}_2:[a_2,b_2]\to \cl W_1\setminus S$ with $\overline{\gamma}_2((a_2,b_2)) \subset W_1\setminus S$. Let $W_2$ be the connected open subset of $W_1$ determined by the Jordan curve $\overline{\gamma}_2 \cup \ell_2$. Arguing exactly as in \eqref{its path connected}, $W_2\setminus S$ is open and path-connected, so we can iterate this argument to obtain mutually disjoint embeddings $\overline{\gamma}_k:[a_k,b_k]\to \cl W_1 \setminus S\subset \Gamma_1 \cup \Omega_1$ with $\overline{\gamma}_k((a_k,b_k)) \subset \Omega_1$ for $1\leq k \leq K_1^1$.
    \par
    \medskip
    Next, let $\ell_{K_1^1+1}$ be another maximal curve with endpoints in $\Gamma_1$. The same argument as in proving \eqref{its path connected} implies that $\Omega_1 \setminus \cl W_1$ is path-connected, and so $\gamma(a_{K_1^1+1})$, $ \gamma(b_{K_1^1+1})$ may be connected by a smooth embedding $\overline{\gamma}_{K_1^1+1}:[a_{K_1^1+1},b_{K_1^1+1}]\to (\Gamma_1 \cup \Omega_1)\setminus \cl W_1$, that, together with $\ell_{K_1^1+1}$, defines a connected domain $W_{K_1^1+1}\subset \Omega_1$ by the Jordan curve theorem. In addition, $W_{K_1^1+1}\cap W_1=\varnothing$ since $(\ell_2 \cup \overline{\gamma}_{K_1^1+1})\cap \cl W_1=\varnothing$ by \eqref{maximal arc} and the definition of $\overline{\gamma}_{K_1^1+1}$. Repeating the whole iteration procedure for those intervals contained in $\ell_{K_1^1+1}$ and then the rest of the maximal arcs, we finally obtain mutually disjoint embeddings $\overline{\gamma}_k:[a_k,b_k]\to \Gamma_1\cup \Omega_1$ with $\overline{\gamma}_k((a_k,b_k))\subset \Omega_1$ as desired for $1\leq k \leq K_1$.
    \par
    \medskip
    \noindent{\it Conclusion of the proof of \eqref{dichot 2}}: Repeating the $\Gamma_1$ procedure for $\{\Gamma_2,\dots,\Gamma_{I_\gamma}\}$ and using the mutual pairwise disjointness of $\Gamma_i$, we obtain mutually disjoint embeddings $\overline{\gamma}_k:[a_k,b_k]\to \cl B_r(x) \setminus S$ with $\overline{\gamma}_k((a_k,b_k)) \subset B_r(x) \setminus S$ for $1\leq k \leq K_1$. We define $\overline{\gamma}:\mathbb{S}^1\to \Omega$ by
    \[\overline{\gamma}(t)= \begin{cases}
          \gamma(t) & t\in \mathbb{S}^1\setminus \cup [a_k,b_k] \\
          \overline{\gamma}_k(t) & t\in [a_k,b_k]\,,\,\,\,1\leq k \leq K\,.
       \end{cases}
    \]
    Since $\overline{\gamma}=\gamma$ outside $B_r(x) \cc \Omega$, $\overline{\gamma}$ is homotopic to $\gamma$ relative to $\Omega$. Furthermore, $\overline{\gamma}$ is piecewise smooth and homotopic to $\gamma$, and so it can be approximated in the $C^0$ norm by $\{\gamma_j\}\subset \mathcal{C}$. However, by \eqref{doesnt hit the outside} and the construction of $\overline{\gamma}_k$, $\overline{\gamma}\cap S = \varnothing$, which implies that $S \cap \gamma_j = \varnothing$ for large $j$. This contradicts the fact that $S$ is $\mathcal{C}$-spanning $\wire$, and so \eqref{dichot 2} is true.
    \end{proof}

    \bibliographystyle{alpha}
    \bibliography{references}
    \end{document}